\documentclass{mod}
\usepackage{url}
\usepackage{setspace}
\usepackage{scrextend}
\usepackage{cite}
\usepackage[shortlabels]{enumitem}
\usepackage{hyperref}
\usepackage{caption}
\usepackage{subcaption}

\usepackage{amsmath}
\usepackage{amsthm}
\usepackage{amssymb}
\DeclareMathAlphabet{\mathbbm}{U}{bbm}{m}{n}

\usepackage{mathtools}
\usepackage{mathrsfs}

\usepackage[all]{xy}
\usepackage[toc,page]{appendix}
\usepackage{titlesec}
\usepackage{upgreek}

\usepackage{tocloft}

\usepackage{tocbasic}
\DeclareTOCStyleEntry[
beforeskip=.2em plus 1pt,
entryformat=\normalfont ,
pagenumberformat=\bfseries
]{tocline}{section}

\usepackage{ stmaryrd }

\titleformat{\paragraph}[runin]{\small\sffamily\bfseries
}{}{}{}[]

\newcounter{mainthm}
\newtheorem{mainthm}[mainthm]{Theorem}

\newcounter{mainconj}
\newtheorem{mainconj}[mainconj]{Conjecture}

\newtheorem{thm}{Theorem}[section]

\theoremstyle{plain}
\newtheorem{lem}[thm]{Lemma}
\newtheorem{prop}[thm]{Proposition}

\newtheorem{defn-thm}[thm]{Definition-Theorem}

\newtheorem{defn-lem}[thm]{Definition-Lemma}

\theoremstyle{definition}
\newtheorem{defn}[thm]{Definition}
\newtheorem{convention}
[thm]{Convention}

\newtheoremstyle{rmk}
{5pt}
{5pt}
{}
{}
{\itshape}
{}
{.5em}
{}

\theoremstyle{rmk}
\newtheorem{rmk}[thm]{Remark}

\newtheoremstyle{note}
{8pt}
{5pt}
{\itshape}
{10pt}
{\bfseries}
{}
{.5em}
{}

\theoremstyle{note}

\setlist[description]{font=
	\normalfont
	\itshape
	\space}

\newcommand{\UD}{\mathscr{UD}}

\newcommand{\ev}{\mathrm{ev}}

\newcommand{\f}{\mathfrak f}

\DeclareMathOperator{\dist}{dist}

\DeclareMathOperator{\trop}{\mathfrak{trop}}
\DeclareMathOperator{\Hom}{Hom}

\DeclareMathOperator{\val}{\mathsf{v}}

\titleformat{\subsection}[runin]{
	\bfseries\itshape\normalsize}{\thesubsection \ }{0em}{}[\mbox{ . } ]
\titlespacing{\subsection}
{0pt}
{2.25ex plus 1ex minus .2ex}
{0pt}

\titleformat{\subsubsection}[runin]{
	\itshape\normalsize}{\thesubsubsection \ }{0em}{}[\mbox{} ]
\titlespacing{\subsubsection}
{0pt}
{2.25ex plus 1ex minus .2ex}
{0pt}

\begin{document}
	\setlength{\parindent}{15pt}	\setlength{\parskip}{0em}
	
	\title{Family Floer mirror space for local SYZ singularities}
	\author[Hang Yuan]{Hang Yuan}

	\begin{abstract} {\sc Abstract:}  
		We give a mathematically precise statement of the SYZ conjecture between mirror space pairs and prove it for any toric Calabi-Yau manifold with the Gross Lagrangian fibration.
		To date, it is the first time we realize the SYZ proposal with singular fibers beyond the topological level.
		The dual singular fibration is explicitly written and proved to be compatible with the family Floer mirror construction. 
Moreover, we discover that the Maurer-Cartan set of a singular Lagrangian is only a strict subset of the corresponding dual singular fiber. This responds negatively to the previous expectation and leads to new perspectives of SYZ singularities. As extra evidence, we also check some computations for a well-known folklore conjecture for the Landau-Ginzburg model.
		\end{abstract}

	\maketitle

{\hypersetup{linkcolor=black}
	\textrm{\tableofcontents}
}

%
%

\vspace{-0.4em}

\section{Introduction}

Mirror symmetry is a mysterious relationship, discovered by string physicists, between pairs of Calabi-Yau manifolds $X, X^\vee$. The mathematical interest in mirror symmetry began since the enumerative prediction of Candelas et al \cite{MirrorSymmetryOrigin}.
Nowadays, the two major approaches to the mathematical mirror symmetry are the Kontsevich's Homological Mirror Symmetry (HMS) \cite{KonICM} and the Strominger-Yau-Zaslow conjecture (SYZ) \cite{SYZ}.
These two ideas focus on different aspects of mirror symmetry beyond enumeration problems.
The SYZ conjecture explains why a pair $X,X^\vee$ should be mirror to each other geometrically based on the physical idea of the T-duality; meanwhile, the HMS conjecture predicts a categorical equivalence between the Fukaya category of $X$ {(A side)} and the derived category of coherent sheaves of $X^\vee$ {(B side)}. It is expected to be the underlying principle behind the enumerative prediction.

The work of Joyce \cite{Joyce_Singularity} implies the strong form of the SYZ conjecture cannot hold yet \cite[p191]{Joyce_book}, see also \cite{MSClayII}. This is because there are serious issues to match singular loci. Another issue for the SYZ idea is that what we mean by `dual tori' is unclear. If we want the T-duality to be useful in constructing mirrors, Gross's topological mirror symmetry \cite{Gross_topo_MS} tells us that the SYZ conjecture may be somewhat topological.

Following Fukaya's family Floer program \cite{FuFamily, FuAbelian} and Kontsevich-Soibelman's non-archimedean mirror symmetry proposal \cite{KSTorus, KSAffine}, we propose a modified mathematically precise SYZ statement with an emphasis on both the aspects of symplectic topology and non-archimedean analytic topology:

\begin{mainconj}
	\label{conjecture_NA_SYZ}
	Given any Calabi-Yau manifold $X$,
	\begin{enumerate}
		\item[(a)] there exists a Lagrangian fibration $\pi: X \to B$ onto a topological manifold $B$ such that the $\pi$-fibers are graded with respect to a holomorphic volume form $\Omega$;
		\item[(b)] there exists a tropically continuous surjection $f:  \mathscr Y \to B$ 
		from an analytic space $\mathscr Y$ over the Novikov field $\Lambda=\mathbb C((T^{\mathbb R}))$ onto the same base $B$;
	\end{enumerate}
\vspace{0.3em}
satisfying the following 

\begin{itemize}
	\item[\textbf{(i)}] $\pi$ and $f$ have the same singular locus skeleton $\Delta$ in $B$;
	\item[\textbf{(ii)}]  $\pi_0=\pi|_{B_0}$ and $f_0=f|_{B_0}$ induce the same integral affine structures on $B_0=B\setminus \Delta$.
\end{itemize}
\end{mainconj}


\begin{mainthm}
	\label{Main_thm_oversimplify_this_paper}
	Conjecture \ref{conjecture_NA_SYZ} holds for the Gross special Lagrangian fibration $\pi$ (with singularities) in any toric Calabi-Yau manifold. Moreover, the analytic space $\mathscr Y$ embeds into an algebraic variety.
\end{mainthm}

\begin{rmk}
The SYZ mirror construction has been well-studied by Auroux and many others \cite{AuTDual,Gross_topo_MS,CLL12,chan2016gross,AAK_blowup_toric,GHK15,KSAffine,KSTorus,hong2018immersed}, etc. We apologize for not being able to give a full list. This paper is indebted to various strategies of our predecessors, but we also want to humbly highlight a key limitation of the previous works: the lack of a good notion for the \textit{dual} SYZ fibration, especially concerning singular fibers. We aim to further explore this aspect.
For example, a slightly new geometric input involves the monodromy information for the A-side wall-crossing studies, cf. \S \ref{ss_topological_disk}. Besides, to make a B-side fibration with reasonable matching conditions, we utilize certain non-archimedean geometry, cf. \S \ref{ss_NA_integrable_affinoid_torus}.
\end{rmk}

\begin{rmk}
	The statement is briefly explained here.
	The \textit{tropical continuity} of $f$ in (b) is as introduced by Chambert-Loir and Ducros \cite[(3.1.6)]{Formes_Chambert_2012}.
	However, for clarity, one might first interpret $f:\mathscr Y\to B$ as just a continuous map for the Berkovich analytic topology in $\mathscr Y$ \cite{Berkovich_2012spectral, Berkovich1993etale} and the usual manifold topology in $B$. 
	Due to Kontsevich and Soibelman, we can define the smooth/singular points of $f$, and the smooth part $f_0=f|_{B_0}$ is called an \textit{affinoid torus fibration}; see also \cite[\S 3]{NA_nonarchimedean_SYZ}. Just like the Arnold-Liouville's theorem, any affinoid torus fibration induces an \textit{integral affine structure} on $B_0$ \cite[\S 4]{KSAffine}.
\end{rmk}

\begin{rmk}
A referee raised doubts about Conjecture \ref{conjecture_NA_SYZ}, suggesting that certain cohomological obstructions might prevent a Calabi-Yau variety from admitting a Lagrangian torus fibration. The rigid Calabi-Yau examples considered by Candelas-Derrick-Parkes \cite{candelas1993generalized} were cited as potential counterexamples due to the absence of maximal degenerations, with the remark that such a case "\textit{does not have a mirror space, only a mirror component in the derived category of some Fano variety}." In response, it may be clarified that this work is directed towards exploring a precise definition of the "mirror space". The original SYZ conjecture has also consistently involved Lagrangian fibrations and related concepts.
In fact, the statement of Conjecture \ref{conjecture_NA_SYZ} should be seen as a guiding formulation rather than an absolute claim, open to future refinement as new insights may emerge. While maximal degenerations are a major method for constructing Lagrangian fibrations, prioritizing a single method and overlooking others can lead to a narrow perspective. Other methods, including those using symplectic techniques, have been explored in works such as \cite{AAK_blowup_toric, Lag_3_torus, evans2021constructing}.
For any program aiming to understand mirror symmetry mathematically, we believe that the key question is whether there are sufficient examples to support the program, rather than highlighting universality and exception from the outset.
Developing constructions without concrete examples risks creating theoretical "castles in the air".
The referee also expressed concern that this work represents only an incremental step, as it does not provide a complete proof of mirror symmetry. In response, we note that a purely algebro-geometric approach, while valuable, has inherent limitations in addressing Lagrangian submanifolds within the Fukaya category. Thus, a symplectic method of mirror construction is essential for systematic progress toward the proof of homological mirror symmetry.
The originality of this approach lies in bridging symplectic and non-archimedean geometry, supported by further examples \cite{Yuan_conifold,Yuan_A_n} of Conjecture \ref{conjecture_NA_SYZ} and related applications \cite{Yuan_e.g._FamilyFloer,Yuan_unobs,Yuan_c_1}.
\end{rmk}

The statement of Conjecture \ref{conjecture_NA_SYZ} is mathematically precise and does not mention any mirror symmetry actually. But, we should think $(X,\pi)$ and $(\mathscr Y, f)$ are mirror to each other, and we can make sense of T-duality with an extra Floer-theoretic condition (iii) to specify what we mean by dual tori:

\begin{itemize}
		\item[\textbf{(iii)}] 
		\textit{$f_0$ is isomorphic to the canonical dual affinoid torus fibration $\pi_0^\vee$ associated to $\pi_0$}
\end{itemize}

In the set-theoretic level, if we set $L_q=\pi^{-1}(q)$ and write $U_\Lambda$ for the unit circle in $\Lambda=\mathbb C((T^{\mathbb R}))$ with the non-archimedean norm, then the $\pi_0^\vee$ is simply the following obvious map:
\begin{equation}
	\label{X_0_vee_set_eq}
X_0^\vee \equiv \bigcup_{q\in B_0} H^1(L_q; U_\Lambda) \to B_0
\end{equation}
Family Floer theory with quantum correction further equips $X_0^\vee$ with a non-archimedean analytic structure sheaf such that $\pi_0^\vee$ becomes an affinoid torus fibration (see \S \ref{s_family_review} or \cite{Yuan_I_FamilyFloer}).
It is unique up to isomorphism, so we can say it is canonical, and the meaning of (iii) is also precise.
Note that a change from $U_\Lambda$ to $U(1)$ exactly goes back to the conventional T-duality picture (e.g. \cite{Gross_topo_MS,AuTDual}).
In the level of non-archimedean analytic structure, while the local analytic charts of $\pi_0^\vee$ have been predicted for a long time \cite{FuCounting,FuCyclic,Tu}, the local-to-global \textit{analytic} gluing for $\pi_0^\vee$ is recently achieved in \cite{Yuan_I_FamilyFloer}. Finally, we introduce the following notion:

\begin{defn}
	\label{SYZmirror_defn}
	In the situation of Conjecture \ref{conjecture_NA_SYZ}, if the conditions \textbf{(i) (ii) (iii)} hold and the analytic space $\mathscr Y$ embeds into (the analytification $Y^{\mathrm{an}}$ of) an algebraic variety $Y$ over $\Lambda$ of the same dimension, then we say $Y$ is \textbf{\textit{SYZ mirror}} to $X$.
\end{defn}


\subsection{Main result}
For clarity, we focus on a fundamental example of Theorem \ref{Main_thm_oversimplify_this_paper}, and the general result is stated later in \S \ref{ss_generalizations_introduction}.
We state:

\begin{thm}
	\label{Main_this_paper_fundamental_example}
	The algebraic variety
	\[
	Y=\{(x_0,x_1,y_1,\dots, y_{n-1})\in \Lambda^2\times (\Lambda^*)^{n-1} \mid x_0x_1=1+y_1+\cdots +y_{n-1}\}
	\]
	is SYZ mirror to $X=\mathbb C^n\setminus \{z_1\cdots z_n=1\}$.
\end{thm}

\begin{rmk}
	\label{Auroux_T_duality_complement_rmk}
	The mirror space $Y$ is expected by the works of Abouzaid-Auroux-Katzarkov \cite{AAK_blowup_toric} and Chan-Lau-Leung \cite{CLL12}.
	If we take $\bar X=\mathbb C^n$ instead of $X$ without removing the divisor, the mirror will be the same $Y$ with an extra superpotential $W=x_1$ \cite{AuTDual,Au_Special}.
	It will be further discussed in \S \ref{ss_folklore_introduction}.
\end{rmk}


\subsection{Relation to the literature}
\label{ss_literature}

The integral affine structure matching condition \textbf{(ii)} hinders the direct application of Kontsevich and Soibelman's construction \cite{KSAffine}. Indeed, the integral affine coordinates from a Lagrangian fibration subtly depend on the symplectic form; similarly, on the non-archimedean B-side, there is also a delicate story about the induced integral affine structure from $f$. For instance, deforming $\psi$ in our solution (\ref{explicit_formula_introduction_eq}) alters the integral affine structure, even if the monodromy around the singular locus may be unchanged.
This subtle point necessitates detailed calculations, even though we are able to write down explicitly the formula of the solution $f$ as in (\ref{explicit_formula_introduction_eq}) (cf. Remark \ref{strategry_rmk}).

Within the literature, an approach has been presented to construct an affinoid torus fibration (away from a singular locus) using Berkovich retraction. This method draws inspiration from birational geometry \cite{NA_nonarchimedean_SYZ}.
But, our construction of affinoid torus fibration uses a different method and is grounded in the Floer-theoretic analysis of a Lagrangian fibration in view of \textbf{(iii)}.

The underlying principle follows Auroux's framework \cite{AuTDual,Au_Special} of wall-crossing (see also \cite{cho2017localized, cho2018gluing,hong2018immersed}, etc.). The primary differences are that Gromov's compactness guarantees convergence only over the Novikov field rather than $\mathbb C$, and that non-archimedean geometry is required to interpret another version of torus fibration, which also induces an integral affine structure on the base.
In particular, we study \textit{two} fibrations on distinct spaces simultaneously, rather than focusing on a single fibration. To the best of our knowledge, the existing literature of studying two \textit{singular} fibrations simultaneously in the context of the SYZ conjecture may trace back to Gross's work on topological mirror symmetry \cite{Gross_topo_MS}.

\subsection{Sketch of proof of Theorem \ref{Main_this_paper_fundamental_example} omitting Floer-theoretic condition (iii)}
\label{ss_one_page_proof_intro}

Let's provide a glimpse into the structure of the solution to Conjecture \ref{conjecture_NA_SYZ}. Despite an explicit answer, a comprehensive proof and detailed calculations are still necessary and deferred to the main body of this paper.

We restrict the standard symplectic form in $\mathbb C^n$ to $X$ and consider the special Lagrangian fibration
\begin{equation}
	\label{pi_intro_eq}
\textstyle
\pi:X\to\mathbb R^n, \quad (z_1,\dots, z_n)\mapsto
( \tfrac{|z_1|^2-|z_n|^2}{2}, \dots, \tfrac{|z_{n-1}|^2-|z_n|^2}{2}, \log |z_1\cdots z_n -1| )
\end{equation}
Denote by $B_0$ and $\Delta$ the smooth and singular loci of $\pi$ in the base $B=\mathbb R^n$. There is a continuous map $\psi:\mathbb R^n \to \mathbb R$, smooth in $B_0$, so that $(\tfrac{|z_1|^2-|z_n|^2}{2}, \dots, \tfrac{|z_{n-1}|^2-|z_n|^2}{2}, \psi\circ \pi )$ forms a set of action coordinates locally over $B_0$.
Roughly, the function $\psi$ indicates the symplectic areas of holomorphic disks in $\mathbb C^n$ bounded by the $\pi$-fibers parameterized by the base points in $\mathbb R^n$ (Figure \ref{figure_area_sing_fiber}).



\begin{proof}[Sketch of proof of Theorem \ref{Main_this_paper_fundamental_example} omitting \emph{(iii)}]
	Consider an analytic domain $\mathscr Y=\{ |x_1|<1\}$ in $Y^{\mathrm{an}}$. 
	Define a topological embedding $j: \mathbb R^n \to\mathbb R^{n+1}$ assigning $q=(q_1,\dots, q_{n-1}, q_n)=(\bar q,q_n)$ in $\mathbb R^n$ to
	\[
	j(q)=(\min\{-\psi(q), -\psi(\bar q, 0) \}+\min\{0,\bar q\}  \ , \ \min\{  \psi(q),  \psi(\bar q,0) \} \ , \  \bar q)
	\]
	Define a tropically continuous map $F:Y^{\mathrm{an}}\to\mathbb R^{n+1}$ by
	\[
	F(x_0,x_1,y_1,\dots, y_{n-1})=(F_0,F_1, \val(y_1),\dots, \val(y_{n-1}))
	\]
	where $\mathsf v:\Lambda\to\mathbb R\cup\{\infty\}$ is the non-archimedean valuation and
	\[
	\begin{cases}
		F_0=\min \{
		\val(x_0), -\psi(\val(y_1),\dots, \val(y_{n-1}),0)+\min\{0,\val(y_1),\dots, \val(y_{n-1})\} \ \}  
		\\
		F_1=\min\{
		\val(x_1), \ \ \ \psi (\val(y_1),\dots, \val(y_{n-1}),0) \ \}
	\end{cases}
	\]
We can check $j(\mathbb R^n)=F(\mathscr Y)$ (cf. Figure \ref{figure_visualize_image_j}). Then, we can define
\begin{equation}
	\label{explicit_formula_introduction_eq}
	f=j^{-1}\circ F|_{\mathscr Y} \ : \mathscr Y\to \mathbb R^n
\end{equation}
This is a variant of \cite[\S 8]{KSAffine} that further includes the data of the symplectic form; see also \S \ref{ss_defn_f_B_fibration}.
By detailed calculations (Section \ref{sss_describe_F} and Theorem \ref{fibration_preserving_thm}), we will find: the smooth / singular loci and the induced integral affine structure of $f$ \textbf{\textit{all}} precisely agree with those of $\pi$ (\ref{pi_intro_eq}).
Except the duality condition (iii), the proof is complete.
\end{proof}

\begin{rmk}
	\label{strategry_rmk}
	Here we adopt the strategy of Kontsevich and Soibelman \cite[\S 8]{KSAffine}. Rather than seeking the desired Berkovich-continuous map $f: \mathscr Y\to B$, we find an alternative $F:\mathscr Y\to \mathbb R^N$ for some larger $N$ such that the image of $F$ is identified with the singular integral affine manifold $B$ through a map $j$.
	This embeds $B$ into a larger Euclidean space $\mathbb R^N$ to unfold the singularities (Figure \ref{figure_visualize_image_j}).
	While this approach might seem ad-hoc for the singular part, the smooth part $f_0\cong \pi_0^\vee$ remains canonical by the family Floer condition \textbf{(iii)}. We hypothesize that the tropical continuity condition might ensure the uniqueness of the singular extension from $f_0$ to $f$ due to the piecewise-smooth nature of the reduced symplectic geometry, but this will be addressed in future work.
	Without the guidance from Floer theory, constructing the appropriate $f$ is quite challenging. At least, directly replicating the example by Kontsevich-Soibelman seems difficult to match the singular integral affine structure from the Lagrangian fibration $\pi$, given its intricate dependence on the given symplectic form $\omega$ (cf. \S \ref{ss_literature}).
\end{rmk}

\begin{figure}[h]
	\centering
	\begin{subfigure}{0.55\textwidth}
		\centering
		\includegraphics[scale=0.13]{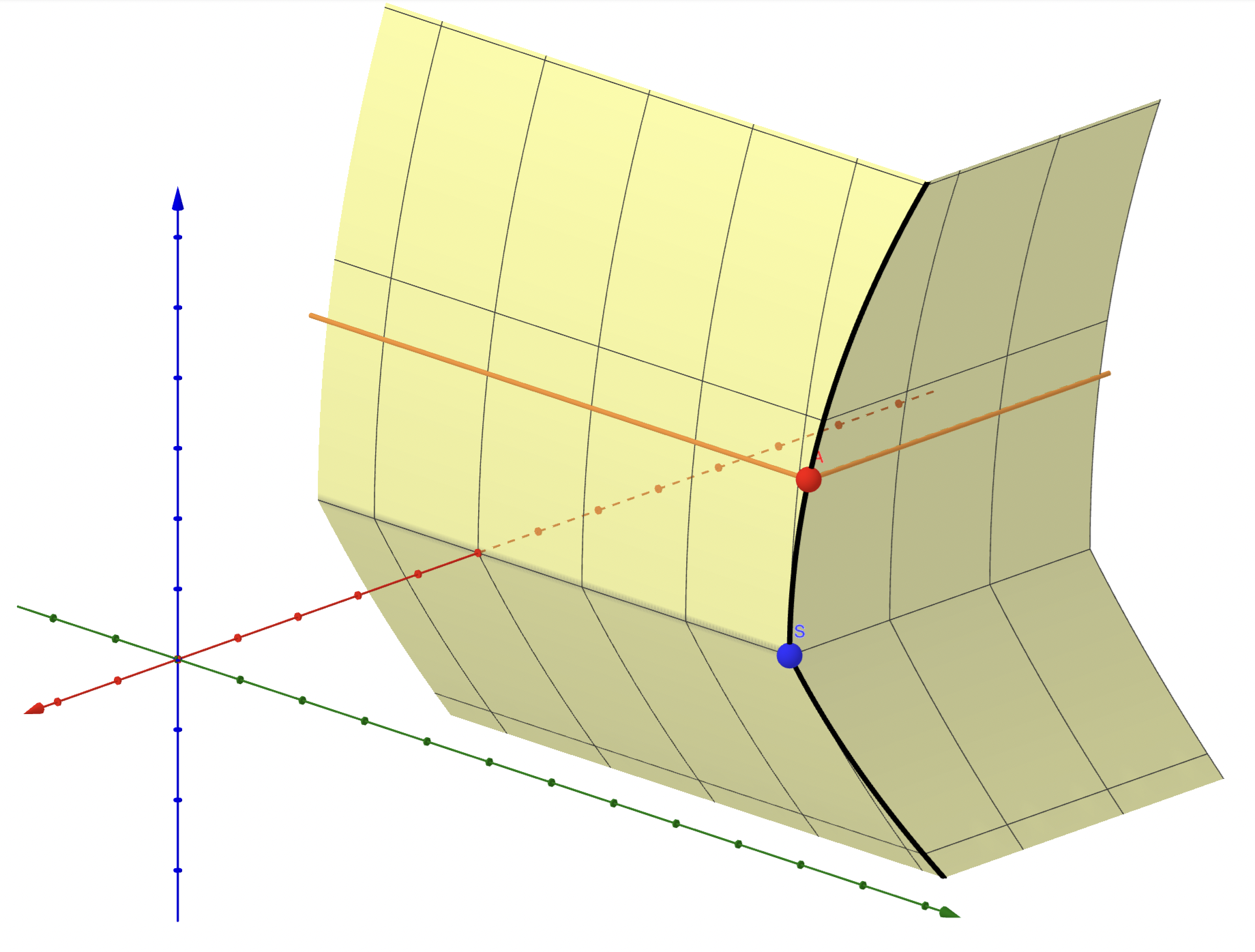}
	\end{subfigure}
	\begin{subfigure}{0.43\textwidth}
		\includegraphics[scale=0.13]{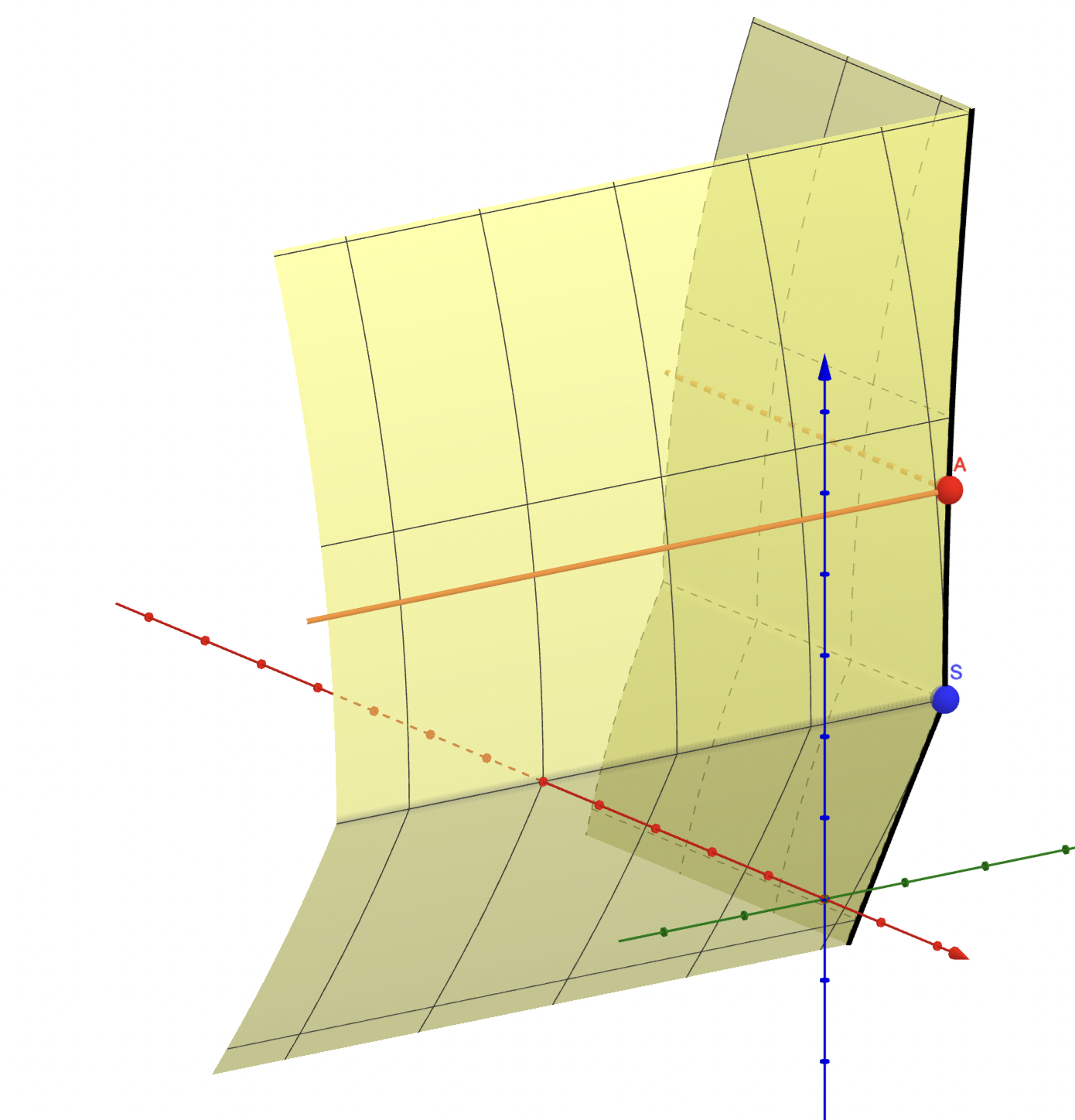}
	\end{subfigure}
	\caption{
		\footnotesize The image $j(B)=F(\mathscr Y)$ in $\mathbb R^{3}$ for $n=2$: It morally visualizes the integral affine structure.
	}
	\label{figure_visualize_image_j}
\end{figure}

\subsection{Outline of the construction}

The existence of singular Lagrangian fibers may induce two types of quantum corrections of holomorphic disks as illustrated in red and yellow in Figure \ref{figure_area_sing_fiber}. The red disk meets the singular fiber at an interior point, while the yellow disk meets it at an boundary point.
We will discuss the \textit{Floer aspect} and the \textit{non-archimedean analytic aspect} about them respectively.

\subsubsection{Floer aspect: dual affinoid torus fibration. } 
Denote by $\pi_0:X_0\to B_0$ the smooth part of the Gross Lagrangian fibration $\pi:X\to B$.
The non-archimedean dual torus fibration depends on the entire ambient space $X$, as the disks may extend beyond the region $X_0=\pi^{-1}(B_0)$.
By the Floer aspect of this type of quantum correction, we can canonically associate to $(X,\pi_0)$ an analytic space $X_0^\vee$ with an \textit{abstract} dual affinoid torus fibration $\pi_0^\vee: X_0^\vee\to B_0$ (\S \ref{s_family_review} or \cite{Yuan_I_FamilyFloer}), unique up to isomorphism, so that its induced integral affine structure on $B_0$ agrees with the one induced by $\pi_0$ and the set of closed points in $X_0^\vee$ are given by (\ref{X_0_vee_set_eq}).
Meanwhile, there is the other \textit{concrete} affinoid torus fibration $f_0:\mathscr Y_0\to B_0$, i.e. the smooth part of the analytic fibration $f$ in (\ref{explicit_formula_introduction_eq}).

The initial step for our version of T-duality can be stated in a single relation as follows:

\begin{prop}
		\label{Main_proposition_over_B_0}
	There is an isomorphism of affinoid torus fibration
$
\pi_0^\vee \cong f_0
$.
\end{prop}

The former $\pi_0^\vee$ is constructed canonically but abstractly, while the latter $f_0$ is ad hoc but concrete.
We can even explicitly write down an analytic embedding $g:X_0^\vee\to \Lambda^2\times (\Lambda^*)^{n-1}$ with $\pi_0^\vee=f_0\circ g$. The map $g$ identifies $X_0^\vee$ with the analytic domain $\mathscr Y_0$ in $Y$. 
In view of (\ref{X_0_vee_set_eq}), 
any closed point in $\mathscr Y_0$ can be realized as a local $U_\Lambda$-system $\mathbf y$ in some $H^1(L_q;U_\Lambda)$.
The explicit definition formula of $g$ is in \S \ref{ss_g_analytic_embedding}.
\begin{equation}
	\label{g_embedding_intro_eq}
\xymatrix{
	X_0^\vee \ar[rr]^{g}_{\cong} \ar[dr]_{\pi_0^\vee} &  & \mathscr Y_0 \ar[dl]^{f_0}
	\\
	& B_0 & 
}
\end{equation}

\begin{rmk}
	\label{only_place_Floer_intro_rmk}
	The \textit{\textbf{only}} place we need the Floer theory is basically an identification between the family Floer mirror analytic space $X_0^\vee$ and an `adjunction' analytic space $T_+\sqcup T_- /\sim$ obtained by gluing two analytic open domains $T_\pm \subsetneq (\Lambda^*)^n$ which correspond to the Clifford / Chekanov tori respectively (Remark \ref{floer_omit_rmk}).
	Such a simplification is not easy but now enables us to apply the idea in \cite[Lemma 3.1]{GHK_birational} where two copies of complex tori $(\mathbb C^*)^n$ are glued instead; further combining the non-archimedean picture in \cite[p.44]{KSAffine} enables us to discover the desired embedding $g$ (see a reader guide in Remark \ref{floer_omit_rmk}.)
\end{rmk}

\begin{rmk}
	\label{min_tropical_rmk}
	The tropical polynomial
	$
	\min\{0,q_1,\dots, q_{n-1}\}
	$
	plays the leading roles in the singularities of both the A and B sides.
	First, the induced tropical hypersurface (i.e. the above minimum value attains at least twice) exactly describes the singular locus skeleton of the Gross Lagrangian fibration $\pi$ in (\ref{pi_intro_eq}).
	Second, $\val(1+y_1+\cdots +y_{n-1})$ is either $>$ or $=$ $\min\{0,\val(y_1),\dots, \val(y_{n-1})\}$ by the non-archimedean triangle inequality. The ambiguity case $>$ happens only if the minimum attains at least twice as well. After some effort, one may find that this ambiguity is the very reason of the singularity of $f$ in (\ref{explicit_formula_introduction_eq}).
	In general (\S \ref{ss_generalizations_introduction}),
 the desired dual map $f$ is almost the same as (\ref{explicit_formula_introduction_eq}) but using another tropical polynomial.
\end{rmk}

\begin{figure}
	\centering
	\begin{subfigure}{0.5\textwidth}
		\centering
		\includegraphics[scale=0.32]{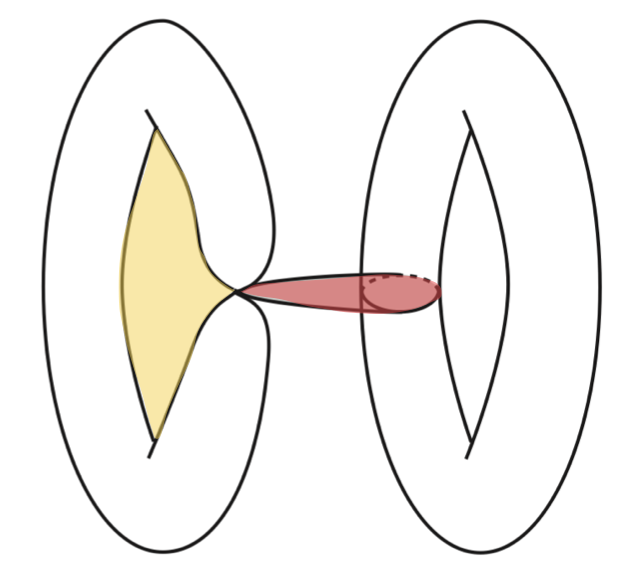}
	\end{subfigure}
	\begin{subfigure}{0.45\textwidth}
		\centering
		\includegraphics[scale=0.3]{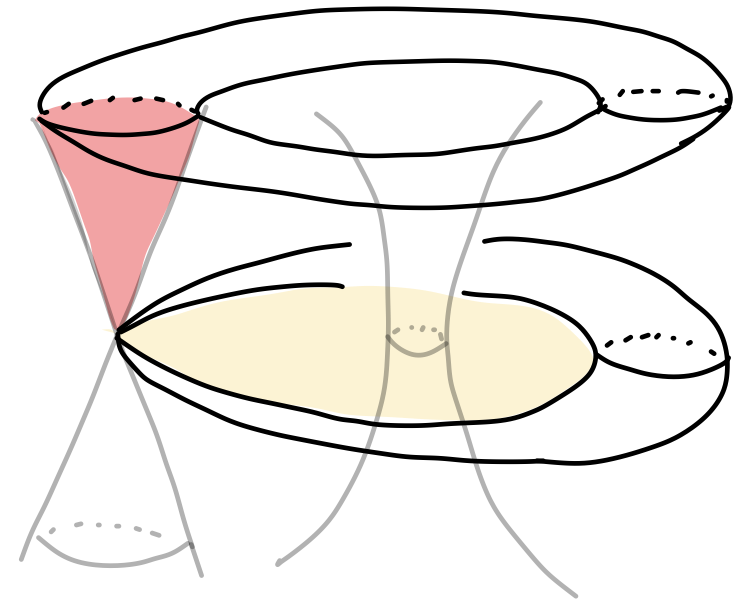}
	\end{subfigure}
	\caption{
		\footnotesize Two types of quantum corrections in red and yellow, meeting the singular fibers at the interior / boundary points of the disk domain respectively. The right side concerns the Lagrangian fibration $\pi$ in (\ref{pi_intro_eq}) for $n=2$ and follows Auroux \cite[5.1]{AuTDual}.}
	\label{figure_area_sing_fiber}
\end{figure}

\subsubsection{Non-archimedean analytic aspect: dual singular fibers. }
\label{sss_dual_singular_intro}

The significance of Proposition \ref{Main_proposition_over_B_0} lies on the fact that the abstract affinoid torus fibration $\pi_0^\vee: X_0^\vee\to B_0$ has an explicit model $f_0$, via $g$, which naturally has an obvious tropically continuous extension $f$ in (\ref{explicit_formula_introduction_eq}) fitting in the diagram below.
\begin{equation}
	\label{g_embedding_extend_singular_intro_eq}
	\xymatrix{
		 X_0^\vee \ \ \ar@{^{(}->}[rr]^{g} \ar[d]_{\pi_0^\vee} & & \mathscr Y \ar[d]_f \\
		B_0 \ \ \ar@{^{(}->}[rr] & & B
	}
\end{equation}
The left vertical arrow $\pi_0^\vee$ is the family Floer dual affinoid torus fibration $\pi_0^\vee$ in (iii).
The upper horizontal map $g$ follows Gross-Hacking-Keel in \cite[Lemma 3.1]{GHK_birational}. The right vertical arrow $f$ generalizes Kontsevich-Soibelman's singular model in \cite[\S 8]{KSAffine}. Finally, the top right corner $\mathscr Y$ agrees with many previous results \cite{AuTDual,Au_Special,AAK_blowup_toric,CLL12,Abouzaid_Sylvan,Gammage_localSYZ}, etc.


More importantly, the dual singular fibers do naturally capture the data of Lagrangian singular fibers:
The other type of quantum correction
refers to the holomorphic disks whose boundary meet the singular fibers (Figure \ref{figure_area_sing_fiber}, yellow), and they are all reflected by the formula of $f$ (\ref{explicit_formula_introduction_eq}).
In summary, under the control of the canonical analytic structure on $(X_0^\vee, \pi_0^\vee)$ from the Floer aspect of the first type of quantum correction, the second type correction deduces the singular analytic extension.
When we extend $f_0$ to $f$ \textit{tropically continuously} \cite[(3.1.6)]{Formes_Chambert_2012},
the topological extension $B_0\xhookrightarrow{} B$ controls the analytic extension overhead (\S \ref{ss_NA_integrable_affinoid_torus}).
Note that any affinoid torus fibration is tropically continuous.

Our explicit description of the $f$ in (\ref{explicit_formula_introduction_eq}) can depict all the dual singular fibers simultaneously, and the general results in \S \ref{ss_generalizations_introduction} will even offer more local models of SYZ singularities.
The aspects of non-archimedean analytic topology seem to outweigh the Floer-theoretic considerations around the singular locus $\Delta=B\setminus B_0$.
Moreover, we will astonishingly discover in the set level that (see \S \ref{ss_singular_discussion})
\begin{equation}
	\label{negate_MC}
	\text{Dual singular fiber} \  \  \supsetneq \ \ \text{Maurer-Cartan set of the singular fiber}
\end{equation}
based on the work of Hong-Kim-Lau \cite{hong2018immersed}.
This justifies our standpoint in \cite{Yuan_I_FamilyFloer} that going beyond the usual Maurer-Cartan picture \cite{Tu,FuCounting,FuBerkeley} is necessary to produce the global mirror analytic structure.
A well-known fact in the area of homological algebra is that the homotopy equivalences among $A_\infty$ algebras induce bijections of Maurer-Cartan sets. However, this offers merely a local or set-theoretic approximation and is not sufficient for the local-to-global construction of a non-archimedean analytic space. By definition, the latter is built by matching the affinoid spaces instead.

\subsection{Main result in general}
\label{ss_generalizations_introduction}

Our method is very powerful in that the same ideas for Theorem \ref{Main_this_paper_fundamental_example} with some basics of tropical and toric geometry can obtain more general results with very little extra effort.

Denote by $N$ and $M$ two lattices of rank $n$ that are dual to each other. Set $N_{\mathbb R}=N\otimes \mathbb R$ and $M_{\mathbb R}=M\otimes \mathbb R$.
Let $\Sigma$ be a simplicial smooth fan with maximal cones $n$-dimensional in $N_{\mathbb R}$. Assume $v_1,\dots, v_n$ are the primitive generators of the rays in a maximal cone in $\Sigma$, so they form a $\mathbb Z$-basis of $N$. Denote by $v_1^*,\dots, v_n^*$ the dual basis of $M$.
Denote the remaining rays in $\Sigma$ by $v_{n+1},\dots, v_{n+r}$ for $r\ge 0$, and we set $v_{n+a}=\sum_{j=1}^n k_{aj} v_j$ for $k_{aj}\in\mathbb Z$ and $1\le a\le r$.
Assume the toric variety $\mathcal X_\Sigma$ associated to $\Sigma$ is Calabi-Yau; namely, there exists $m_0\in M$ so that $\langle m_0,v\rangle =1$ for any ray $v$ in $\Sigma$. Then, $m_0=v_1^*+\cdots +v_n^*$ and $\sum_{j=1}^n k_{aj}=1$ for any $1\le a \le r$.
Let $w=z^{m_0}$ be the toric character associated to $m_0$, and 
$
\mathscr D:=\{w=1\}
$
is an anti-canonical divisor in $\mathcal X_\Sigma$.
We equip $\mathcal X_\Sigma$ with a toric K\"ahler form $\omega$, and the moment map $\mu: \mathcal X_\Sigma\to M_{\mathbb R}$ is onto a polyhedral $P$ described by a collection of inequalities as follows:
\begin{equation}
	\label{polyhedral_P_intro_eq}
P: \quad \langle m, v_i\rangle +\lambda_i\ge 0 \qquad m\in M_{\mathbb R}
\end{equation}
where the $v_i$'s run over all the rays in $\Sigma$ and $\lambda_i\in\mathbb R$.
The sublattice $\bar N=\{ n\in N\mid \langle m_0,n\rangle =0\}$ has a basis $\sigma_s=v_s-v_n$ for $1\le s<n$. The action of $\bar N\otimes \mathbb C^*$ preserves $\mathscr D$ and gives a moment map $\bar\mu$ onto $\bar M_{\mathbb R}:=M_{\mathbb R}/\mathbb Rm_0$ such that $p\circ \mu=\bar \mu$ for the projection $p: M_{\mathbb R}\to \bar M_{\mathbb R}$. One can show $p$ induces a homeomorphism $\partial P\cong \bar M_{\mathbb R}$.
We identify
$\bar M_{\mathbb R}:=M_{\mathbb R}/\mathbb Rm_0$ with a copy of $\mathbb R^{n-1}$ in $M_{\mathbb R}\cong \mathbb R^n$ consisting of $(m_1,\dots, m_n)$ with $m_n=0$.
Now, the \textit{Gross special Lagrangian fibration} \cite{Gross_ex} refers to $\pi:=(\bar \mu, \log|w-1|)$ on
$
X:=\mathcal X_\Sigma\setminus \mathscr D
$ (see also \cite{AAK_blowup_toric,CLL12}).
The singular locus of $\pi$ takes the form $\Delta=\Pi\times \{0\}$ where $\Pi$ is the tropical hypersurface (Figure \ref{figure_tropical}) in $\bar M_{\mathbb R}\cong \mathbb R^{n-1}$ associated to the tropical polynomial that is decided by the data in (\ref{polyhedral_P_intro_eq}):
\begin{equation}
	\label{tropical_polynomial_introduction_eq}
	h_{\mathrm{trop}}(q_1,\dots, q_{n-1})= \min\Big\{ \lambda_n,  \{q_k+\lambda_k\}_{1\le k<n}, \{\textstyle \sum_{s=1}^{n-1}k_{as}q_s+\lambda_{n+a}\}_{1\le a\le r}
	\Big\}
\end{equation}
Recall that the Novikov field $\Lambda=\mathbb C((T^{\mathbb R}))$ is algebraically closed. Let $\Lambda_0$ and $\Lambda_+$ be the valuation ring and its maximal ideal.
On the other hand, given the A side data above, we consider
\begin{equation}
	\label{tropical_original_h_introduction_eq}
	h(y_1,\dots, y_{n-1})=
	T^{\lambda_n}(1+\delta_n) + \sum_{s=1}^{n-1}T^{\lambda_s} y_s (1+\delta_s)  + \sum_a T^{\lambda_{n+a}} (1+\delta_{n+a}) \prod_{s=1}^{n-1} y_s^{k_{as}}  
\end{equation}
where $\delta_i\in \Lambda_+$ is given by some virtual counts so that $\mathsf v(\delta_i)$ is the smallest symplectic area of the sphere bubbles meeting the toric divisor $D_i$ associated to $v_i$ (see \S \ref{s_generalization} for the details).

\begin{rmk}
	\label{delta_i_zero_condition_rmk}
 In general, the Cho-Oh's result \cite{Cho_Oh} is not enough to decide the $\delta_i$'s. But, whenever $D_i$ is non-compact, one can use the maximal principle to show that $\delta_i=0$ (c.f. \cite[\S 5.2]{CLL12}).
 The expressions of $\delta_i$ may be also interpreted via the inverse mirror maps due to \cite{chan2016gross}.
\end{rmk}

A key observation is that the tropicalization of $h$ in (\ref{tropical_original_h_introduction_eq}) is precisely $h_{\mathrm{trop}}$ in (\ref{tropical_polynomial_introduction_eq}), since $\delta_i\in\Lambda_+$. This picture is lost if we only work over $\mathbb C$. By Definition \ref{SYZmirror_defn}, let's state a more general result:

\begin{thm}
	\label{Main_this_paper_general}
	The algebraic variety
	\[
	Y=\{(x_0,x_1,y_1,\dots, y_{n-1})\in \Lambda^2\times (\Lambda^*)^{n-1} \mid x_0x_1= h(y_1,\dots, y_{n-1}) \}
	\]
	is SYZ mirror to $X=\mathcal X_\Sigma \setminus \mathscr D$.
\end{thm}

The proof is almost identical to that of Theorem \ref{Main_this_paper_fundamental_example}. The key dual singular fibration $f$ is written in the same way as (\ref{explicit_formula_introduction_eq}), replacing the tropical polynomial $\min\{0,q_1,\dots, q_{n-1}\}$ by $h_{\mathrm{trop}}$ (Remark \ref{min_tropical_rmk}).
So, for legibility, we focus on Theorem \ref{Main_this_paper_fundamental_example} in the main body and delay the generalization to \S \ref{s_generalization}.

\subsection{Examples and SYZ converse}

	The statement of Theorem \ref{Main_this_paper_general} gives rise to a lot of examples.
	
\begin{figure}
	\centering
	\includegraphics[scale=0.3]{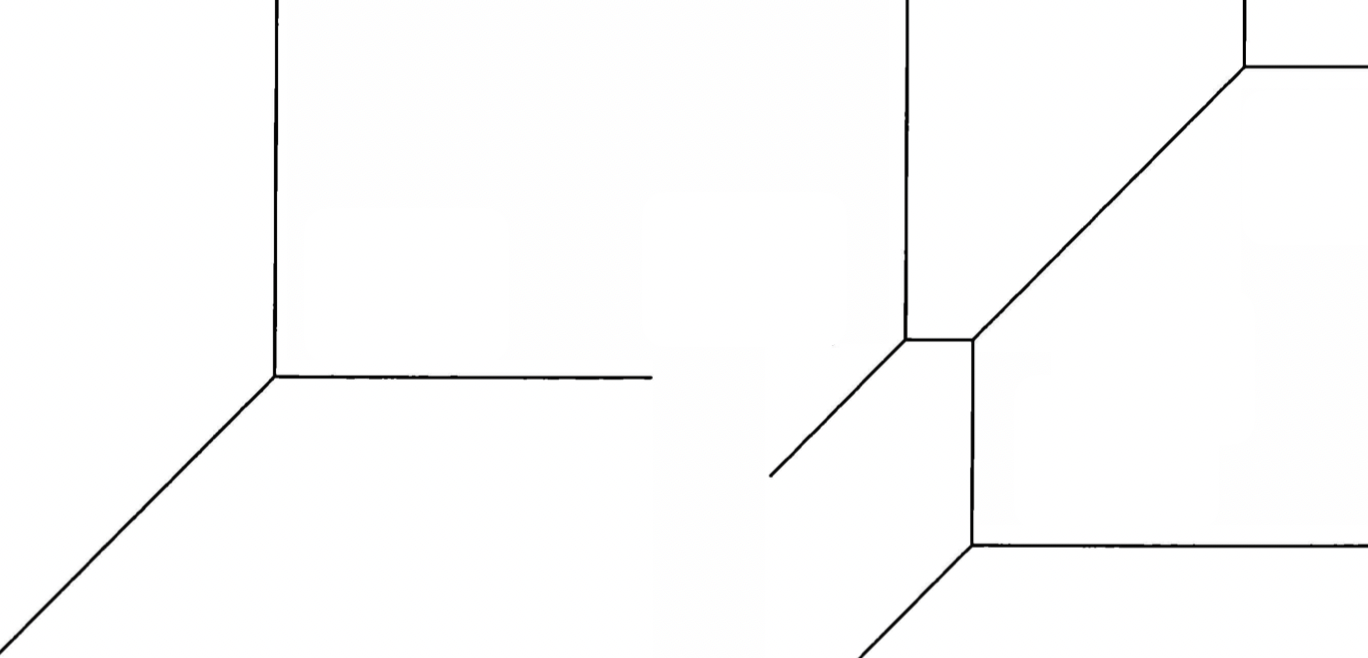}
	\caption{Tropical hypersurfaces of $h_{\mathrm{trop}}$ in \S \ref{sss_trop_Cn} and \S \ref{sss_trop_314} respectively}
	\label{figure_tropical}
\end{figure}
	
\subsubsection{}
	\label{sss_ex_SYZ_converse_intro}
		We begin with a general remark for a version of \textit{SYZ converse}. One can use the Laurent polynomial $h$ to recover the polyhedral $P$ as follows. Consider the polyhedral $P'$ in $\bar M_{\mathbb R}\oplus \mathbb R\cong \mathbb R^{n-1}\oplus \mathbb R$ defined by
		$ q_n+h_{\mathrm{trop}}(q_1,\dots,q_{n-1})\ge 0
		$.
		Namely, it is defined by $q_n+\lambda_n\ge 0$, $q_n+q_s+\lambda_s\ge 0$ ($1\le s<n$), and $q_n+\sum_{s=1}^{n-1} k_{as}q_s+\lambda_{n+a}\ge 0$ ($1\le a\le r$) due to (\ref{tropical_polynomial_introduction_eq}).
		Then, the isomorphism $\bar M_{\mathbb R}\oplus \mathbb R\cong M_{\mathbb R}$,
$(q_1,\dots, q_{n-1},q_n)\mapsto (q_1+q_n,\dots, q_{n-1}+q_n, q_n)$, can naturally identify $P'$ with $P$ in (\ref{polyhedral_P_intro_eq}).

	\subsubsection{}
		\label{sss_trop_Cn}
		Back to Theorem \ref{Main_this_paper_fundamental_example},
we have $r=0$, $\lambda_i=0$, and $\delta_i=0$.
Then, by (\ref{tropical_original_h_introduction_eq}), $h=1+y_1+\cdots+y_{n-1}$, and its tropicalization is $h_{\mathrm{trop}}=\min\{0,q_1,\dots, q_{n-1}\}$; see Figure \ref{figure_tropical}.
Compare also Remark \ref{min_tropical_rmk}.
We can also check the polyhedral $P'\cong P$ is identified with the first quadrant in $\mathbb R^n$.
	
	\subsubsection{}
	Consider the fan $\Sigma$ in $\mathbb R^3$ generated by $v_1=(1,0,0)$, $v_2=(0,1,0)$, $v_3=(0,0,1)$, and $v_4=(1,-1,1)$.
	So, $n=3$ and $r=1$. Note that $v_4=v_1-v_2+v_3$, i.e. $k_{11}=k_{13}=1$ and $k_{12}=-1$.
	The corresponding toric variety is the conifold $\mathcal X_\Sigma=\mathcal O_{\mathbb P^1}(-1)\oplus \mathcal O_{\mathbb P^1}(-1)$. 
We equip $\mathcal X_\Sigma$ with a toric K\"ahler form $\omega$, and the moment polyhedral $P$ is defined by (\ref{polyhedral_P_intro_eq}) for some arbitrary $\lambda_1,\dots, \lambda_4\in\mathbb R$.
	By (\ref{tropical_original_h_introduction_eq}), $h(y_1,y_2)=T^{\lambda_3} +T^{\lambda_1}y_1+T^{\lambda_2}y_2+T^{\lambda_4} y_1 y_2^{-1}$.
	First, we may assume $\lambda_3=0$. Also, replacing $y_i$ by $T^{-\lambda_i}y_i$, we may assume $\lambda_1=\lambda_2=0$.
	So, $(\lambda_1,\lambda_2,\lambda_3,\lambda_4)=(0,0,0, \lambda)$ for some $\lambda\in\mathbb R$, and $h=1+y_1+y_2+T^\lambda y_1y_2^{-1}$. This retrieves the case of \cite[5.3.2]{CLL12}, if we replace the Novikov parameter $T$ by some $t\in\mathbb C$. But, we remark that the shapes of the tropical hypersurfaces of $h_{\mathrm{trop}}$ or the moment polyhedrals $P'\cong P$ obtained by \S \ref{sss_ex_SYZ_converse_intro} are different for $\lambda>0$ and $\lambda<0$, which seems lost over $\mathbb C$.
	
	\subsubsection{} 
	\label{sss_trop_314}
	Consider the Laurent polynomial $h(y_1,y_2)=y_1+T^{-1}y_2+T^{3.14}+T^2y_1^2+y_1y_2+T^2y_2^2$.
	By \S \ref{sss_ex_SYZ_converse_intro}, we get the fan $\Sigma$ in $\mathbb R^3$ generated by $v_1=(1,0,0)$, $v_2=(0,1,0)$, $v_3=(0,0,1)$, $v_4=(2,0,-1)$, $v_5=(1,1,-1)$, $v_6=(0,2,-1)$. Also, we get the polyhedral $P$ defined by (\ref{polyhedral_P_intro_eq}) with respect to these $v_i$'s and the numbers $\lambda_1=0$, $\lambda_2=-1$, $\lambda_3=3.14$, $\lambda_4=2$, $\lambda_5=0$, $\lambda_6=2$. Note that the above $h$ is carefully chosen so that the tropical hypersurface associated to $h_{\mathrm{trop}}$ does not enclose a bounded region (Figure \ref{figure_tropical}, right). This ensures the toric variety $\mathcal X_\Sigma$ has no compact irreducible toric divisor (c.f. Remark \ref{delta_i_zero_condition_rmk}).
	By Theorem \ref{Main_this_paper_general} and \S \ref{sss_ex_SYZ_converse_intro}, we achieve a version of SYZ converse construction from B side to A side. Notice that we truly have infinitely many such examples.

\subsection{Further evidence: a folklore conjecture}
\label{ss_folklore_introduction}
We have extra evidence supporting our proposed SYZ statement for the following folklore conjecture, recognized by Auroux, Kontsevich, and Seidel \cite[\S 6]{AuTDual}:

\begin{mainconj}
	\label{conjecture_folklore}
	The critical values of the mirror Landau-Ginzburg superpotential on $X^\vee$ (B side) are the eigenvalues of the quantum multiplication by the first Chern class on $X$ (A side).
\end{mainconj}

Recall the dual affinoid torus fibration $\pi_0^\vee:X_0^\vee \to B_0$ only relies on the Maslov-0 holomorphic disks in $X$ bounded by $\pi_0$-fibers.
It often happens that the $\pi_0$ can be placed in a larger ambient space $\overline X$, adding more Maslov-2 holomorphic disks but adding no Maslov-0 ones.
By the general theory in \cite{Yuan_I_FamilyFloer},
the family Floer mirror associated to the same fibration $\pi_0$, placed in the larger $\overline X$ yet, is given by the same analytic space $X_0^\vee$ and the same affinoid torus fibration $\pi_0^\vee: X_0^\vee \to B_0$ but equipped with an additional function $W_0^\vee$ on $X_0^\vee$ (\S \ref{ss_Maslov_zero_determine}).
Note that the $W_0^\vee$ as well as its critical points and critical values all depend on the K\"ahler form $\omega$. When we deform $\omega$, the image of these critical points under the fibration map $\pi$ may wildly change.
But, in the present paper, we only focus on the examples and refer to \cite{Yuan_c_1} for the general theory.

Recall $\mathscr Y_0=f_0^{-1}(B_0)\cong X_0^\vee$ via the analytic embedding $g$ in (\ref{g_embedding_intro_eq}).
In our case, the LG superpotential turns out to be polynomial, and from Definition \ref{SYZmirror_defn} it follows that $\mathscr Y_0$ is Zariski dense in the algebraic variety $Y$ (cf. \cite{payne2009fibers}). Thus, the $W_0^\vee$ extends to a polynomial superpotential $W^\vee$ on the whole $Y$. 
Note that the $W^\vee$ relies on the larger ambient space $\overline X$, although the $Y$ does not.
For various ambient spaces $\overline X$, we have various $W^\vee$.
The computations in \cite{Yuan_e.g._FamilyFloer} over the Novikov field $\Lambda= \mathbb C((T^{\mathbb R}))$ rather than just over $\mathbb C$ (eg. \cite{PT_mutation}) is quite crucial to verify Conjecture \ref{conjecture_folklore} here. The details of computations will be in Appendix \S \ref{s_folklore}, and we just write down the results here:

\begin{enumerate}
\item Suppose $X=\mathbb C^n\setminus\{z_1\cdots z_n=1\}$ and $\overline X=\mathbb {CP}^n$. Then, the LG superpotential is 
\[
W^\vee = x_1+\frac{T^{E(\mathcal H)}  x_0^n}{ y_1\cdots y_{n-1}}
\]
defined on the algebraic variety $Y=\{x_0x_1=1+y_1+\cdots+y_{n-1}\}$,
where $\mathcal H$ is the class of a complex line and $E(\mathcal H)=\frac{1}{2\pi}\omega\cap \mathcal H$ is the symplectic area.
By direct computations, one can check the critical points of this $W^\vee$ (for the logarithmic derivatives) are
\[
\begin{cases}
	x_0 =  T^{-\frac{E(\mathcal H)}{n+1}} e^{-\frac{2\pi i s}{n+1}}	&		\\
	x_1 = 	n T^{\frac{E(\mathcal H)}{n+1}	} e^{\frac{2\pi i s}{n+1}} & 	\\
	y_1=\cdots=y_{n-1}=1
\end{cases}
\qquad s\in\{0,1,\dots, n\}
\]
They are all contained in the same dual fiber over the base point that can be in either the Clifford or Chekanov chambers, relying on $\omega$ and $\pi$. 
Moreover, one can check
the critical values are 
\[
(n+1) T^{\frac{E(\mathcal H)}{n+1}} e^{\frac{2\pi i s}{n+1}}, \qquad s\in\{0,1,\dots, n\}
\]
which agrees with the $c_1$-eigenvalues in the quantum cohomology $QH^*(\mathbb {CP}^n; \Lambda)$.

\item Suppose $X=\mathbb C^n\setminus \{z_1\cdots z_n=1\}$ and $\overline X=\mathbb {CP}^m\times  \mathbb {CP}^{n-m}$ for $0<m<n$. Then,
\[
W^\vee
=
x_1 +  \frac{T^{E(\mathcal H_1)} x_0^m}{y_1\cdots y_m} +  \frac{T^{E(\mathcal H_2)} x_0^{n-m}}{ y_{m+1}\cdots y_{n-1}}
\]
on the same $Y$,
where $\mathcal H_1, \mathcal H_2$ are the classes of a complex line in $\mathbb {CP}^m$ and $\mathbb {CP}^{n-m}$ respectively.
The corresponding critical values are
\[
(m+1) T^{\frac{E(\mathcal H_1)}{m+1}} e^{\frac{2\pi i r}{m+1}} + (n-m+1) T^{\frac{E(\mathcal H_2)}{n-m+1}} e^{\frac{2\pi i s}{n-m+1}} 
\]
for $r\in\{0,1,\dots, m\}$ and $s\in\{0,1,\dots, n-m\}$ and agree with the $c_1$-eigenvalues.
Moreover, we want to further study the locations of the critical points of $W^\vee$.
For simplicity, let's assume $n=2$ and $m=1$, then
$
W^\vee=x_1+ \frac{T^{E(\mathcal H_1)} x_0}{y} +T^{E(\mathcal H_2)} x_0
$
is defined on $Y=\{x_0x_1=1+y\}$.
We have four critical points of $W^\vee$ given by
\[
\begin{cases}
	x_0=\pm	T^{\frac{-E(\mathcal H_2)}{2}}  	 \\
	x_1= \pm T^{\frac{E(\mathcal H_1)}{2}} \pm T^{\frac{E(\mathcal H_2)}{2}} \\
	y_1= \pm T^{\frac{E(\mathcal H_1)-E(\mathcal H_2)}{2}}
\end{cases}
\]
They are contained in the $f$-fiber over the base point
$
\hat q=\left(\frac{E(\mathcal H_1)-E(\mathcal H_2)}{2}, a_\omega \right)
$
in $B\equiv \mathbb R^2$ where $a_\omega$ is some value that relies on the K\"ahler form $\omega$.
\end{enumerate}

We really obtain infinitely many LG superpotentials on $Y$ parameterized by the various K\"ahler forms, and all of them will satisfy the folklore conjecture.
In the case (2) above, it may happen that $E(\mathcal H_1)\neq E(\mathcal H_2)$ while $a_\omega=0$, then the base point $\hat q$ meets the wall.
In general, the base points of critical points rely on $\omega$ or $\pi$, and the walls of Maslov-0 disks rely on $J$.
The family Floer non-archimedean analytic structure does not rely on $J$ up to isomorphism \cite{Yuan_I_FamilyFloer}. For instance, the formula of $f$ in (\ref{explicit_formula_introduction_eq}) clearly does not rely on $J$ and cannot detect the walls of Maslov-0 disks (relying on $J$).
To sum up, the Maslov-0 correction is usually unavoidable (cf. \cite[\S 5]{Auroux2015infinitely} \cite[Example 3.3.2]{Au_Special}) and gives rise to some non-archimedean analytic structure which needs to be remembered in the Floer theory along the way.
Given non-trivial Maslov-0 disks, all the previous arguments for Conjecture \ref{conjecture_folklore} will fail.
A major new challenge is that the LG superpotential is now locally only well-defined up to affinoid algebra isomorphisms, or more precisely up to family Floer transition maps in the language of \cite{Yuan_I_FamilyFloer}.
Fortunately, based on the ud-homotopy and canceling tricks in \cite{Yuan_I_FamilyFloer}, a conceptual proof of Conjecture \ref{conjecture_folklore} with Maslov-0 corrections has been achieved in \cite{Yuan_c_1}.
Admittedly, its value may rely on future examples we would find, but there are clearly other examples by Theorem \ref{Main_this_paper_general} with similar computations.
 The new ideas in \cite{Yuan_c_1} will also inspire future studies, such as a more global picture of the closed-string mirror symmetry that asserts the quantum cohomology of $X$ is isomorphic to the Jacobian ring of $W^\vee$ (cf. \cite{FOOO_bookblue}).


\subsection*{Acknowledgment}
We thank Mohammed Abouzaid for encouraging discussions about \cite{Abouzaid_Sylvan}.
We thank Denis Auroux and Yingdi Qin for some key ideas for Lemma \ref{symplectic_area_increasing_lem}.
We thank Antoine Chambert-Loir for a correspondence that explains \cite{Formes_Chambert_2012}, especially the notion of the tropically continuous map.
We thank Jonathan David Evans for a helpful correspondence about \cite{evans_2021_book}.
We thank Kenji Fukaya and Eric Zaslow for their interests,  conversations, and supports.
We thank Siu-Cheong Lau for a useful correspondence about \cite{hong2018immersed}.
We thank Yang Li for his interest and useful discussion.
We thank Enrica Mazzon and Yueqiao Wu for patient explanations of the basics of Berkovich's theory.
We thank Sam Payne for a helpful explanation of \cite{payne2009fibers}.
We thank Tony Yue Yu for an insightful suggestion of Section 8 of Kontsevich-Soibelman's paper \cite{KSAffine} and for many useful lessons of non-archimedean geometry.

\section{A side: the Gross Lagrangian fibration}
\label{s_A_side}

\subsection{Lagrangian fibration}
\label{ss_Gross_fib}

We begin with the Gross's construction in \cite{Gross_ex}.
Define the divisors $D_j=\{z_j=0\}$ ($1\le \ell\le n$) and $\mathscr D=\{z_1z_2\cdots z_n=1\}$ in $\mathbb C^n$. We set $X=\mathbb C^n\setminus \mathscr D$.
Consider the $T^{n-1}$-action given by
\begin{equation}
	\label{S_1_action}
	t \cdot (z_1,\dots, z_n) \mapsto (z_1,\dots, z_{k-1}, e^{it} z_k, z_{k+1},\dots, z_{n-1}, e^{-it} z_n) \qquad 1\le k<n
\end{equation}
for $t\in S^1\cong \mathbb R/2\pi \mathbb Z$.
Take the standard symplectic form $\omega=d\lambda$ on $\mathbb C^n$. Let
$
\bar\mu=(h_1,\dots, h_{n-1}): \mathbb C^n\to\mathbb R^{n-1}
$
be a moment map for the $T^{n-1}$-action in (\ref{S_1_action}). One can check the vector fields
\[
\textstyle
\mathfrak X_k= i \big(\bar z_k\frac{\partial}{\partial\bar z_k } - z_k\frac{\partial}{\partial z_k}\big)
-
i \big(\bar z_n\frac{\partial}{\partial\bar z_n } - z_n\frac{\partial}{\partial z_n}\big)
\]
satisfy
$
\iota(\mathfrak X_k)\omega = d h_k
$
for the Hamiltonian functions
\[
h_k(z)=\tfrac{1}{2} \big(|z_k|^2- |z_n|^2 \big) \qquad 1\le k<n
\]
Define $h_n(z)=\log |z_1\cdots z_n-1|$ and $B=\mathbb R^n$.
The \textit{Gross Lagrangian fibration} (with respect to the holomorphic $n$-form $\Omega=(z_1\cdots z_n-1)^{-1} dz_1\wedge \cdots \wedge dz_n$, c.f. \cite{Gross_ex}) refers to the following:
\[
\pi\equiv (\bar\mu, h_n): X \to  B, \quad  (z_1,\dots, z_n)\mapsto 
\big(
h_1,\dots, h_{n-1} , h_n
\big)
\]
Note that the $\pi$ maps $X$ onto $B$.
Note also that the $\pi$-fibers are preserved by the complex conjugation $z_i\mapsto \bar z_i$.
The $\pi$-fiber over $q\in B$ is denoted by $L_q:=\pi^{-1}(q)$, and we write 
\[
q=(\bar q, q_n)=(q_1,\dots, q_{n-1},q_n)
\]
for the standard Euclidean coordinates in $B \equiv \mathbb R^n$.
Let
$
P_{ij}=\{z_i=z_j=0\}=D_i\cap D_j
$ and $\Delta_{ij}=\pi(P_{ij})$.
The \textbf{discriminant locus} of $\pi$ is precisely $\Delta=\bigcup\Delta_{ij}$, and the smooth locus of $\pi$ is $B_0:=B\setminus \Delta$. We denote the restriction of $\pi$ over $B_0$ by:
\[
	\pi_0:=\pi|_{B_0} :X_0\to B_0
\]
where $X_0:=\pi_0^{-1}(B_0)\subset X$. 
Moreover, we actually have $\Delta=\Pi\times \{0\}$, where $\Pi$ is the tropical hyperplane that consists of those $\bar q \in\mathbb R^{n-1}$ such that $\min\{0,q_1,\dots, q_{n-1}\}$ is attained twice.
The tropical hyperplane $\Pi$ separates the subset $H: =(\mathbb R^{n-1}\setminus \Pi)\times \{0\}$ of the base $B$ into $n$ different connected components 
\begin{equation}
	\label{H_k_wall_eq}
	H_k= \{ \bar q \in \mathbb R^{n-1}\setminus \Pi \mid \min (0,q_1,\dots, q_{n-1})=q_k\}
\end{equation}
for $1\le k<n$ and 
\[
H_n=\{ \bar q \in\mathbb  R^{n-1}\setminus \Pi \mid \min(0,q_1,\dots, q_{n-1}) =0\}
\]
For example, when $n=3$, we have $H_1=\{q_1<0, q_1<q_2\}$, $H_2=\{q_2<0, q_1>q_2\}$, and $H_3=\{q_1>0,q_2>0\}$ (see the left side of Figure \ref{figure_tropical}).
When $n=2$, we have $H_1=(-\infty, 0)$ and $H_2=(0,+\infty)$.
Notice that $\bar H_i\cap \bar H_j=\Delta_{ij}=\pi(D_i\cap D_j)$, and $H_j\subset \pi(D_j)$ for $1\le j\le n$.

We usually call the subset $H$ in $B_0$ the \textbf{wall}, because
a torus fiber $L_q$ over $q=(\bar q, q_n)\in B_0$ bounds a non-trivial Maslov index zero holomorphic disk if and only if $q_n=0$, i.e. $q\in H$ (see e.g. \cite[Example 3.3.1]{Au_Special} or \cite[\S 5]{AuTDual}).
Next, we define
\[
B_+=\{q=(q_1,\dots, q_n )\mid q_n>0\}=\mathbb R^{n-1}\times (0,+\infty)
\]
\[
B_-=\{q=(q_1, \dots, q_n)\mid q_n<0\} =\mathbb R^{n-1}\times (-\infty,0)
\]
Then, $B_0=B_+ \sqcup B_-\sqcup \bigsqcup_{j=1}^n H_j$.
We call the $B_+$ (resp. $B_-$) the chamber of Clifford tori (resp. Chekanov tori).
We say the torus fiber $L_q$ is of \textit{Clifford type} if $q_n>0$ and it is of \textit{Chekanov type} if $q_n<0$ which is introduced in \cite{Chekanov1996LagrangianTI, eliashberg1997problem} and many others.

The Arnold-Liouville's theorem implies the existence of the {action-angle coordinates} near any Lagrangian torus fiber.
Let $(M,\omega)$ be a symplectic manifold of dimension $2n$.
Let $M_U\subset M$ be an open subset.
The following two theorems are standard; see \cite{action_angle} and \cite{evans_2021_book}.

\begin{thm}
	\label{action_angle_coordinates_thm}
	Let $\pi=(H_1,\dots, H_n) : M_U  \to U$ be an integrable Hamiltonian system over a contractible domain $U$ with only smooth torus fibers. Then, there exists a diffeomorphism $\chi: U \to V\subset \mathbb R^n$ such that $\chi \circ  \pi$ generates a Hamiltonian torus action on $M$. In other words, there is a local coordinate system $(I, \alpha)=(I_1,\dots, I_n,\alpha_1,\dots,\alpha_n):  \pi^{-1}(U) \to U\times (\mathbb R/ 2\pi \mathbb Z)^n$ such that $I=\chi\circ \pi$ and $\omega=\sum_{i=1}^n dI_i\wedge d\alpha_i$ on $\pi^{-1}(U)$.
\end{thm}

We call $I=(I_1,\dots, I_n)$ and $\alpha=(\alpha_1,\dots,\alpha_n)$ the \textit{action-angle coordinates}.
Two sets of them differ by an integral affine transformation.
If $\pi:M_U\to U$ is a smooth Lagrangian torus fibration in $M_U\subset M$ over a contractible domain $U$, then a set of the action coordinates is decided by a choice of the $\mathbb Z$-basis of $H_1(\pi^{-1}(q); \mathbb Z)$ for some $q\in U$. In reality, we have the following.

\begin{thm}
	\label{action_coordinate_thm}
	Let $\lambda$ be some 1-form in $M_U$ such that $\omega=d\lambda$.
	Assume $\sigma_i=\sigma_i(q)$ are closed curves in the torus fiber $\pi^{-1}(q)\cong T^n$, depending smoothly on $q$, such that they form a basis in $H_1(\pi^{-1}(q); \mathbb Z)$. Then
	$\psi_i=  \frac{1}{2\pi} \int_{\sigma_i} \lambda$ defines a diffeomorphism $\chi=(\psi_1,\dots,\psi_n):U\to \mathbb R^n$ such that a set of action coordinates on $M_U$ are given by $I_i=\psi_i\circ \pi$.
\end{thm}

We will call such a $\chi$ an \textit{integral affine chart}.
If we pick a different $\lambda'$ with $d\lambda'=d\lambda$, then by Stokes' theorem, $\int_{\sigma_i}(\lambda'-\lambda)$ is constant, and the action coordinates just differ by constants.
The $\psi_i$'s in Theorem \ref{action_coordinate_thm} are also called the \textit{flux maps}.
The formula for a set of action coordinates can be complicated in general (c.f. \cite{vu2003semi}, see also \cite[Theorem 6.7]{evans_2021_book}). A key idea used in this paper is that the action coordinates can be often described by linear combinations of symplectic areas of disks in some (possibly larger) ambient space; see \S \ref{ss_action_coordinates_-} below.
The advantage is that one can intuitively see an integral affine transformation in terms of disk bubbling.

\subsection{Topological disks}
\label{ss_topological_disk}
Now, we go back to our situation.
To connect with the Floer theory, we are interested in whether the loops in $\pi_1(L_q)$ can be realized as (linear combinations of) the boundaries of disks in $X=\mathbb C^n\setminus \mathscr D$ or in other reasonable larger ambient symplectic manifold $\overline X$, e.g. $\mathbb C^n$ or $\mathbb {CP}^n$.
Beware that at this moment, we just focus on topological disks rather than holomorphic disks.
We consider the following local systems over $B_0$:
\begin{equation}
	\label{local_system_H_2_H_1_eq}
	\mathscr R_1:=R^1\pi_*(\mathbb Z)\equiv \bigcup_{q\in B_0} \pi_1(L_q), \qquad  \mathscr R_2 :=\mathscr R_2(\overline X):=\bigcup_{q\in B_0} \pi_2( \overline X,L_q)
\end{equation}

	\textit{We must understand the monodromy behavior of the above local systems across multiple components of the walls} (cf. Figure \ref{figure_tropical}). In contrast, most of existing literature only study these disks across a single component.
	Although the latter is sufficient for mirror \textit{space} identifications (cf. \cite{AAK_blowup_toric,CLL12} and \S \ref{ss_literature}), we must understand the former monodromy data for mirror \textit{fibration} realizations.
	Note also that the monodromy for $\mathscr R_2$ contains more information than the monodromy for $\mathscr R_1$ or the integral affine structure.
	This justifies why we need to provide many additional details here, although it may be not so difficult.

The local system $\mathscr R_2=\mathscr R_2(\overline X)$ generally depends on the ambient space $\overline X$. But, we will often make this point implicit.
Taking the boundaries gives rise to a morphism $\partial: \mathscr R_2\to\mathscr R_1$.
A class in $\pi_1(L_q)$ or $\pi_2(\overline X,L_q)$ determines a local section of $\mathscr R_1$ or $\mathscr R_2$ over some contractible domain that contains $q$.
Recall that we have a natural Hamiltonian $T^{n-1}$-action that preserves the Lagrangian torus fibers. Then, for $q\in B_0$, we use $\sigma_k=\sigma_k(q)$ ($1\le k\le n-1$) to denote the class of the orbit of the $S^1$-action in (\ref{S_1_action}). They can be regarded as the global sections of $\mathscr R_1$.

To avoid the monodromy ambiguity, we will work with a covering of the base manifold $B_0$ by the {contractible} domains, such as
$B_+\cup H_\ell\cup B_-$
for $1\le \ell\le n$.
Let $\mathscr N_\ell$ denote a sufficiently small neighborhood of $H_\ell$ in the smooth locus $B_0$. 
The following two contractible domains can cover $B_0$:
\[
\textstyle
B'_\pm := B_\pm \cup \bigsqcup_{\ell=1}^n \mathscr N_\ell
\]

Whenever $q\in B_+$, the Lagrangian torus fiber $L_q$ can deform into a product torus by a Hamiltonian isotopy inside $(\mathbb C^*)^n$.
So, $\pi_2(\mathbb C^n, L_q)\cong \pi_2(\mathbb C^n, (\mathbb C^*)^n )\cong \mathbb Z^n$ has an obvious basis $\beta_1,\dots,\beta_n$ such that $\beta_i\cdot D_j=\delta_{ij}$.
We can represent $\partial\beta_j$ by a loop in the form $t\mapsto  (z_1^0,\dots, z_j^0 e^{it},\dots, z_n^0)$ in the product torus.
But abusing the notations, we still write $\partial\beta_j$ (resp. $\beta_j$) for the induced local sections of $\mathscr R_1$ (resp. $\mathscr R_2$) over $B'_+$.
To emphasize the base point, we may write $\partial\beta_j(q)$ and $\beta_j(q)$.
By (\ref{S_1_action}), we have 
\begin{equation}
	\label{sigma_k_in_B_+_eq}
\sigma_k|_{B'_+}= \partial\beta_k -\partial\beta_n \qquad 1\le k<n
\end{equation}

On the other hand, there is a preferred local section of $\mathscr R_2$ over $B_-$, denoted by
$
\hat\beta: B_- \to \mathscr R_2
$, such that
any Maslov index 2 holomorphic disk $u$ bounding $L_{q}$ for $q\in B_-$ must have $[u]=\hat\beta(q)$; see e.g. \cite[Proposition 4.2. C]{eliashberg1997problem}, \cite[Lemma 4.31]{CLL12}, or \cite{Yuan_e.g._FamilyFloer}.
It can be represented by a section of $z_1\cdots z_n$ over a disk in $\mathbb C$. Note that $\hat\beta\cdot D_i=0$ for all $i$.
Abusing the notation, we still denote by $\hat\beta$ the extension section of $\mathscr R_2$ over $B'_-$.
By studying the intersection numbers with the various divisors, it is direct to show the following purely topological result (see e.g. \cite[Lemma B.7]{Yuan_e.g._FamilyFloer}):
\begin{equation}
	\label{beta_ell_beta_hat_eq}
	\beta_\ell|_{\mathscr N_\ell}=\hat\beta|_{\mathscr N_\ell}
\end{equation}

\begin{convention}
	\label{energy_convention}
	From now on, we will adopt a slightly different convention compared to our previous works \cite{Yuan_I_FamilyFloer, Yuan_e.g._FamilyFloer, Yuan_c_1}: We define $E(\beta)=\frac{1}{2\pi} \omega\cap \beta$ instead of $E(\beta)=\omega\cap \beta$ for $\beta\in\pi_2(X,L)$.
\end{convention}

\subsection{Action coordinates}
\label{ss_action_coordinates_-}
Recall that 
$
q=(q_1,\dots, q_n)=:(\bar q, q_n)
$
are the standard coordinates on $B$, but we aim to find the different action coordinates.
On the one hand, a natural trivialization of $\mathscr R_1|_{B'_-}$ is given by the sections 
\[
\{\sigma_k|_{B'_-}: 1\le k<n \}\cup \{\partial\hat\beta\}
\]
By Theorem \ref{action_coordinate_thm}, they give rise to a set of the action coordinates in $B'_-$ by the flux maps.
The first $n-1$ fluxes are given by
$
	\frac{1}{2\pi} \int_{\sigma_k} \lambda  = \frac{1}{2}  ( |z_k|^2- |z_n|^2 ) = h_k(z)=q_k
$. The last flux for $\partial\hat\beta=\partial\hat\beta(q)$ is hard to make explicit, but it can have a geometric meaning by the Stokes' formula:
\[
\textstyle \psi_-(q):=\frac{1}{2\pi} \int_{\partial\hat\beta} \lambda =\frac{1}{2\pi} \int_{\hat\beta} d\lambda= E(\hat\beta)
\]
which is regarded as a map $\psi_-: B'_-\to \mathbb R$.
To sum up, we get an integral affine coordinate chart
\begin{equation}
	\label{chi_-_eq}
	\chi_- : B'_- \to \mathbb R^n ,\qquad q\mapsto  (q_1,\dots, q_{n-1}, \psi_-(q) ) = (\bar q, \psi_-(q))
\end{equation}

On the other hand, a natural trivialization of $\mathscr R_1|_{B'_+}$ is given by
\[
\{\sigma_k|_{B'_+}: 1\le k<n\} \cup\{\partial\beta_n\}
\]
Recall that $\sigma_k|_{B'_+}=\partial\beta_k-\partial\beta_n$ for $1\le k<n$.
The first $n-1$ fluxes are as above, and moreover
\begin{equation}
	\label{E_beta_k_beta_n_q_k}
	\textstyle
	E(\beta_k) -  E(\beta_n)
	= 
	\frac{1}{2\pi} \int_{\beta_k-\beta_n} d\lambda
	=\frac{1}{2\pi} \int_{\partial\beta_k-\partial\beta_n} \lambda=
	\frac{1}{2\pi } \int_{\sigma_k} \lambda  
	=
	q_k
\end{equation}
Define 
$
\psi_+(q)= E(\beta_n): B'_+\to\mathbb R
$.
By Theorem \ref{action_coordinate_thm}, the associated integral affine coordinate chart is 
\begin{equation}
	\label{chi_+_eq}
	\chi_+ : B'_+ \to \mathbb R^n, \qquad q \mapsto \big(q_1, \dots,  q_{n-1}, \psi_+(q) \big) = (\bar q, \psi_+(q))
\end{equation}

We can use (\ref{beta_ell_beta_hat_eq}) and (\ref{E_beta_k_beta_n_q_k}) to conclude that $\psi_+(q)=\psi_-(q)$ on $\mathscr N_n$ and $\psi_+(q)+q_\ell=\psi_-(q)$ on $\mathscr N_\ell$ for $\ell\neq n$ respectively.
By (\ref{H_k_wall_eq}), this means that for 
$
\textstyle
q\in B'_+\cap B'_-\equiv \bigsqcup_\ell \mathscr N_\ell
$,
\begin{equation}
	\label{psi_+_psi_-_min_eq}
	\psi_-(q)
	=
	\psi_+(q)+\min\{0,q_1,\dots, q_{n-1}\} 
\end{equation}
The two charts $\chi_\pm$ give rise to an integral affine structure on $B_0$ which is different from the standard Euclidean one in $B=\mathbb R^n$. 
In particular, the former integral affine structure cannot be extended to $B$.
Define
$
\upchi_\ell:=\chi_-\circ \chi_+^{-1}: \ \chi_+(\mathscr N_\ell) \to \chi_-(\mathscr N_\ell)
$ for each $\ell$, and we can check 
\begin{equation}
	\label{integral_affine_trans_ell_eq}
	\upchi_\ell(c_1,\dots, c_n) = (c_1,\dots, c_{n-1}, c_n+\min\{0,c_1,\dots, c_{n-1}\} )
\end{equation}


\subsection{The embedding $j$}
\label{ss_j_homeomorphism}

We aim to embed the integral affine manifold with singularities $B$ into a higher-dimensional Euclidean space $\mathbb R^{n+1}$. This is actually inspired by the work of Kontsevich-Soibelman (c.f. \cite[p27 \& p45]{KSAffine}). To some degree, this helps us to visualize the integral affine structure.

\subsubsection{Symplectic area. }
Thanks to (\ref{psi_+_psi_-_min_eq}), we have a well-defined continuous function on $B_0$ given by
\begin{equation}
	\label{psi_global_in_B_eq}
	\psi(q)=\psi(\bar q, q_n) =
	\begin{cases}
		\psi_+(q)+\min\{0,q_1,\dots, q_{n-1}\}, & q\in B'_+ \\
		\psi_-(q) , & q\in B'_- 
	\end{cases}
\end{equation}
If $q\in B'_+$, then $\psi(q)=\min\{E(\beta_1),\dots, E(\beta_n)\}$; if $q\in B'_-$, then $\psi(q)=E(\hat\beta)$.
The existence of the Lagrangian fibration with singularities over the whole $B=\mathbb R^n$ tells that the $\psi$ can extend continuously from $B_0$ to $B$, still denoted by 
$\psi: B \to\mathbb R
$ (c.f. the yellow disks in Figure \ref{figure_area_sing_fiber}).
The value of $\psi(q)$ can be represented by the symplectic area of a \textit{topological} disk
\[
u=u(q):  (\mathbb D, \partial\mathbb D)
\to
(\mathbb C^n, L_q)
\]
by $\psi(q)= \frac{1}{2\pi} \int u^*\omega$ (c.f. Convention \ref{energy_convention}).
It is not necessarily a holomorphic disk at this moment, and the symplectic area is purely topological.

\begin{lem}
	\label{symplectic_area_increasing_lem}
	The function $q_n\mapsto \psi(\bar q, q_n)$ is an increasing diffeomorphism from $\mathbb R$ to $(0,\infty)$.
\end{lem}

\begin{proof}
We first show the monotonicity. We may assume that $\bar q$ is a regular value of $\bar \mu$ (otherwise, since we only compare the symplectic areas, we may take a sequence of regular values approaching it).
	We study the symplectic reduction space $\Sigma:=\bar\mu^{-1}(\bar q) / T^{n-1}$. It has real dimension 2 and is endowed with the reduction form $\omega_{red}$. If we write $i: \bar\mu^{-1}(\bar q)\xhookrightarrow{} \mathbb C^n$ and $p: \bar\mu^{-1}(\bar q) \to \Sigma$ for the inclusion and quotient maps, then $i^*\omega=p^*\omega_{red}$.
	The Lefschetz fibration map $w=z_1\cdots z_n$ is invariant under the $T^{n-1}$ action and naturally induces a diffeomorphism $v: \Sigma \to\mathbb C$ with $w\circ i=v\circ p$.
	Note that $L_q$ is contained in $\bar\mu^{-1}(\bar q)$, and the topological disk $u$ is a section of $w$. Up to a homotopy in $\mathbb C^n$ (relative to $L_q$), we may require $u$ has its image contained in $\bar\mu^{-1}(\bar q)$ as well.
	Recall $h_n=\log| w-1|$, so $|w-1|=e^{q_n}$ on $L_q$.
	Now, we have the following maps of topological space-pairs:
	\[
	(\mathbb D, \partial\mathbb D)
	\xrightarrow{u} (\bar\mu^{-1}(\bar q), L_q)
	\xrightarrow{p} (\Sigma,  p(L_q) ) \xrightarrow[\cong]{v} (\mathbb C, \partial D(q_n))
	\]
	where $p(L_q)$ is identified via $v$ with the boundary circle of the disk $D(q_n)=\{\zeta\in\mathbb C\mid |\zeta-1|\le e^{q_n}\}$.
	Let $\check\omega$ is a symplectic form on $\mathbb C$ determined by $v^*\check\omega=\omega_{red}$. Then, $\int u^*\omega = \int u^* p^*\omega_{red} = \int u^* p^* v^* \check \omega$.  In other words, $2\pi \psi(q)=\int u^*\omega$ can be viewed as the symplectic area of the disk $D(q_n)$ in $(\mathbb C, \check \omega)$. No matter how $\check\omega$ looks like, we have the subset inclusion $D(q^1_n)\subset D(q^2_n)$ if $q^1_n<q^2_n$, so the $\check\omega$-area of $D(q_n)$ is increasing in $q_n$. Namely, this says $\psi(\bar q,q_n)$ is increasing in $q_n$.
	Finally, notice that we also have $\lim_{q_n\to \infty} \psi(\bar q, q_n)=\infty$ and $\lim_{q_n\to -\infty}\psi(\bar q, q_n)=0$.
\end{proof}

Recall that $q=(\bar q, q_n)$. We write
\begin{equation}
	\label{psi_0_eq}
	\psi_0(\bar q)=\psi(\bar q,0)
\end{equation}
and we define a continuous embedding
\begin{equation}
	\label{j_homeo_eq}
	j: B \to \mathbb R^{n+1} \qquad q\mapsto (\theta_0(q), \  \theta_1(q), \  \bar q )
\end{equation}
where
\begin{equation}
	\label{theta_0_1_eq}
	\begin{aligned}
	\theta_0(q_1,\dots, q_n)
		&:=
		 \min\{ -\psi(q) , -\psi_0( \bar q ) \}  + \min \{0, \bar q \} 
		 \\
	\theta_1(q_1,\dots, q_n) 
	&:=
	\min\{\ \ \ \psi(q) ,  \ \ \ \psi_0( \bar q )\}
\end{aligned}
\end{equation}
are continuous maps from $\mathbb R^n$ to $\mathbb R$. Using Lemma \ref{symplectic_area_increasing_lem} implies that $j$ is injective. The manifold structure on $B$ induces a one on the image $j(B)$.
Later, we will see the motivation behind $j$ in Section \ref{s_B_side}.

\subsubsection{Description of the image of $j$. }
\label{sss_describe_j}
We think of $B\equiv \mathbb R^n$ as the union of the $\bar q$-slices for all $\bar q\in \mathbb R^{n-1}$.
Notice that the map $j$ is `slice-preserving' in the sense that the following diagram commutes
\[
\xymatrix{
	B\equiv \mathbb R^n \, \, \ar@{^{(}->}[rr]^{j} \ar[dr] &  &  \mathbb R^{n+1} \ar[dl]
	\\
	& \mathbb R^{n-1}&
}
\]
where the left vertical arrow is $(\bar q, q_n)\mapsto \bar q$ and the right one is $(u_0,u_1,\bar q)\mapsto \bar q$.
Thus, we just need to understand the restriction of $j$ on a fixed slice $\bar q\times \mathbb R$ composed with the projection $\mathbb R^{n+1}\cong \mathbb R^2\times \mathbb R^{n-1}_{\bar q} \to\mathbb R^2$.
After taking $\psi(\bar q, \cdot) :\mathbb R\cong (0,+\infty)$ in Lemma \ref{symplectic_area_increasing_lem}, this amounts to study the induced map
\begin{equation}
	\label{jmath_eq}
r_{\bar q}: (0,+\infty) \to \mathbb R^2
\end{equation}
defined by 
\[
(0,+\infty)\ni c\mapsto (\min\{0,\bar q\}+ \min\{-c, -\psi_0(\bar q)\} , \min\{c, \psi_0(\bar q)\})
\]
Here $c$ represents $\psi(\bar q, q_n)$.
The image of $r_{\bar q}$ is a (half) broken line, denoted by $R_{\bar q}$, in $\mathbb R^2$.
Define 
\begin{equation}
	\label{a_01_eq}
	A=A(\bar q)= (a_0(\bar q), a_1(\bar q) ) := \big( \min\{0,\bar q\}-\psi_0(\bar q)  \ , \  \psi_0(\bar q)  \ \ \big)
\end{equation}
to be the corner point of $R_{\bar q}$ parameterized by $\bar q\in\mathbb R^{n-1}$.
Note that $(\bar q,0)\in \Delta$ if and only if $\bar q\in \Pi$.
Note also that $a_1(\bar q)=\psi_0(\bar q)>0$, so this broken line $R_{\bar q}$ always contains the corner point $A(\bar q)$.

\begin{figure}
	\centering
	\includegraphics[scale=0.33]{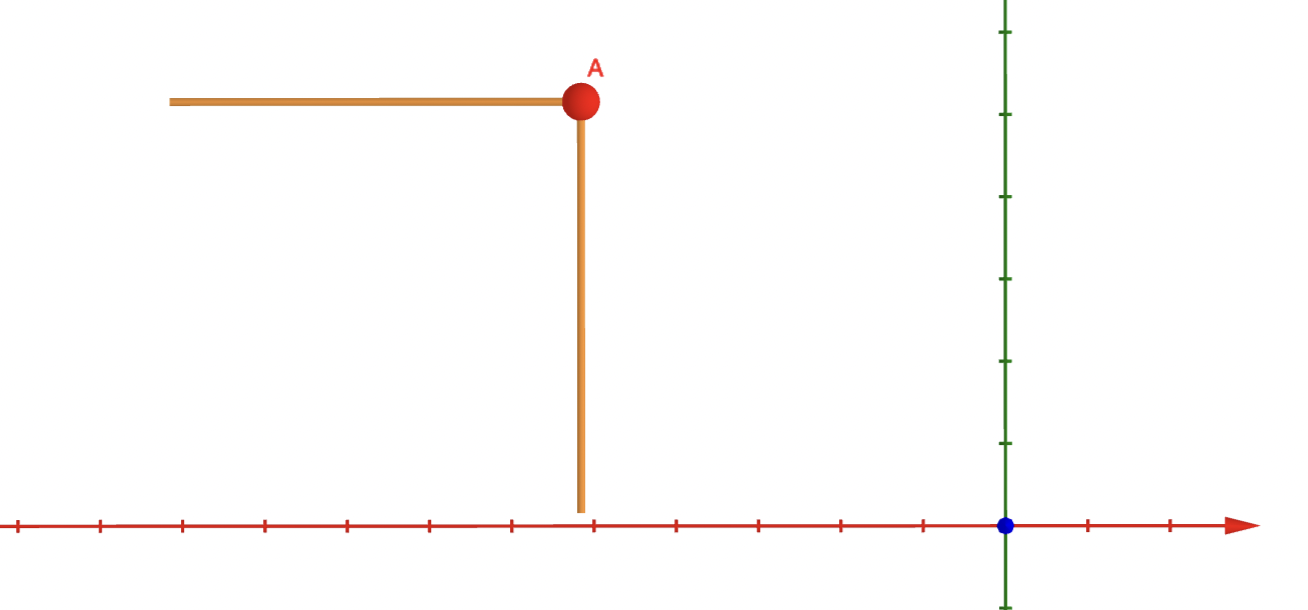}
	\caption{The corner point $A=A(\bar q)$ in the broken line $R_{\bar q}$}
	\label{figure_broken_line}
\end{figure}

\begin{rmk}
	\label{drawing_rmk}
	Intuitively, we may imagine drawing $R_{\bar q}$ in $\mathbb R^2$ with a pen as follows: as $c$ decreases from $+\infty$ to $0+$, we draw from $(-\infty,  a_1(\bar q) )$ horizontally to the corner point $A(\bar q)$ (see Figure \ref{figure_broken_line}).
	Then, we turn the pen and continue drawing vertically downwards until $(a_0(\bar q), 0)$, as $c\to 0+$.
\end{rmk}

The image $j(B)$ can be identified with the graph in $\mathbb R^2\times \mathbb R^{n-1}$ of a family of broken lines $R_{\bar q}$ in $\mathbb R^2$ parameterized by $\bar q \in\mathbb R^{n-1}$ (see Figure \ref{figure_visualize_image_j} in the introduction).
The image $j(\Delta)$ of the singular locus $\Delta\equiv \Pi\times \{0\}$ consists of the corner points in $R_{\bar q}$ for those $\bar q$ in the tropical hyperplane $\Pi$.
To sum up,
\begin{equation}
	\label{j(B)_explicit}
	j(B)=\bigcup_{\bar q\in\mathbb R^{n-1}} R_{\bar q}\times \{\bar q\}, \,  \qquad
	j(\Delta)=\{ (A(\bar q), \bar q) \mid \bar q\in \Pi \}
\end{equation}
In particular, when $n=2$, the $j(\Delta)$ only consists of a single point (i.e. the blue point $S$ in Figure \ref{figure_visualize_image_j}).

\begin{rmk}
	The `singular locus' $j(\Delta)$ in $j(B)$ is curved just like \cite[p27]{KSAffine}. Something similar is also considered in \cite[Example 3.9]{Lag_3_torus}.
	Intuitively, we start the cuts at points on a graph of a \textit{continuous} function (rather than an affine one) and make the gluing in a one-dimension-higher Euclidean space.
\end{rmk}

\section{B side: Kontsevich-Soibelman's analytic fibration}
\label{s_B_side}

In this section, we aim to develop an `analytic torus fibration with singularities' on the algebraic $\Lambda$-variety $Y$ such that its smooth part becomes the \textit{affinoid torus fibration} in the sense of \cite[\S 4.1]{KSAffine}.

\subsection{Tropicalization map}
In this paper, we exclusively consider the \textit{Novikov field} $\Lambda=\mathbb C((T^{\mathbb R}))$, a non-archimedean field that consists of all the infinite series $\sum_{i=0}^\infty a_i T^{\lambda_i}$ where $a_i\in\mathbb C$, $T$ is a formal symbol, and $\{\lambda_i\}$ is a divergent strictly-increasing sequence in $\mathbb R$. 
It has a non-archimedean valuation 
\[
\val: \Lambda \to \mathbb R\cup \{\infty\}
\]
defined by sending the above series to the smallest $\lambda_i$ with $a_i\neq 0$ and sending the zero series to $\infty$. It is equivalent to the non-archimedean norm defined by $|x|=\exp(-\val(x))$.
The multiplicative group is the subset $U_{\Lambda}=\val^{-1}(0)=\{x\in\Lambda\mid |x|=1\}$ that resembles the subgroup $U(1)\equiv S^1$ in $\mathbb C^*$.

Consider the \textit{tropicalization map}
\begin{equation}
	\label{trop_map_eq}
	\trop: (\Lambda^*)^n\to\mathbb R^n, \qquad (z_i)\mapsto (\val(z_i))
\end{equation}

	It is a continuous map with respect to the analytic topology in $(\Lambda^*)^n$ and the Euclidean topology. Note that a fiber of $\trop$ is simply a copy of $U_\Lambda^n\equiv \trop^{-1}(0)$ up to a translation $y_i\mapsto T^{c_i}y_i$; c.f. (\ref{X_0_vee_set_eq}).

\subsection{Non-archimedean integrable system}
\label{ss_NA_integrable_affinoid_torus}

Following Kontsevich and Soibelman \cite[\S 4]{KSAffine}, we introduce an analog of the notion of an integrable system in the non-archimedean analytic setting.

Let $\mathcal Y$ be an analytic space over $\Lambda$ of dimension $n$, and let $B$ be an $n$-dimensional topological manifold or a CW complex.
Let $f: \mathcal Y\to B$ a proper continuous map with respect to the analytic topology and Euclidean topology.
We call a point $p \in B$ \textit{smooth} (or \textit{$f$-smooth}) if there is a neighborhood $U$ of $p$ in $B$ such that the fibration $f^{-1}(U)\to U$ is isomorphic to $\trop^{-1}(V)\to V$ for some open subset $V\subset \mathbb R^n$. 
Here $f^{-1}(U)\cong \trop^{-1}(V)$ is an isomorphism of $\Lambda$-analytic spaces while $U\cong V$ is a homeomorphism.
\[
\xymatrix{
	f^{-1}(U) \ar[r] \ar[d]^{f}  & \trop^{-1}(V) \ar[d]^{\trop}   \\ 
	U\ar[r] & V
}
\]
Let's call it an \textit{affinoid tropical chart}, which may be also viewed as a \textit{tropical chart} in the language of Chambert-Loir and Ducros \cite[(3.1.2)]{Formes_Chambert_2012}.

Let $B_0$ denote the open subset of $f$-smooth points of $B$. 
We call $f$ an \textbf{\textit{affinoid torus fibration}} if $B_0=B$ (see \cite[\S 3.3]{NA_nonarchimedean_SYZ}). In general, we only have $B_0\subsetneq B$, but the restriction of $f$ over $B_0$, denoted by
$
f_0: f^{-1}(B_0)\to B_0,
$
 is always an affinoid torus fibration simply by definition.

The following construction is greatly influenced by Kontsevich-Soibelman \cite[\S 8]{KSAffine}. But, we have to substantially generalize and modify it for our T-duality purpose in the sense of Definition \ref{SYZmirror_defn}.
Let
\[
\psi_0:\mathbb R^ {n-1} \to  (0,+\infty)
\]
be a continuous function.
In practice, we choose $\psi_0(\bar q)=\psi(\bar q,0)$ to be the one in (\ref{psi_global_in_B_eq}, \ref{psi_0_eq}); in this case, $\psi_0>0$ as the symplectic area of a holomorphic disk.

\begin{rmk}
	\label{pure_NA_rmk}
This section can be considered in the pure non-archimedean world. The constructions still hold if the Novikov field $\Lambda$ is replaced by another non-archimedean field.
The existence of affinoid torus fibrations is also a central topic in non-archimedean geometry, and such research is very sparse \cite{KSAffine,NA_nonarchimedean_SYZ}. Now, we give a large class of \textbf{\textit{new}} examples inspired by symplectic methods and SYZ picture.
\end{rmk}

Recall that the $\Lambda$-algebraic variety $Y$ is given by the equation
$
x_0 x_1=1+y_1+\cdots+y_{n-1}
$
in $\Lambda^2_{(x_0,x_1)}\times (\Lambda^*)^{n-1}_{(y_1,\dots,y_{n-1})}$.
As said, we will not always distinguish $Y$ and its analytification.
We define:
\begin{equation}
	\label{F_KS_eq}
	F=(F_0,F_1; G_1,\dots, G_{n-1}):  Y\to \mathbb R^{n+1}
\end{equation}
as follows: given $z=(x_0,x_1,y_1,\dots,y_{n-1})$, we set 
\[
\begin{aligned}
	F_0(z)&=\min\{  
	\val(x_0),   -\psi_0(\val(y_1),\dots, \val(y_{n-1}) )+\min\{0,\val(y_1),\dots, \val(y_{n-1})\}  \} \\
	F_1(z)&= \min \{ \val(x_1), \ \ \ \psi_0(\val(y_1),\dots, \val(y_{n-1})) \} \\
	G_k(z)&=\val(y_k) \qquad \text{for} \  1\le k< n
\end{aligned}
\]
This is a \textit{\textbf{tropically continuous}} map in the sense of Chambert-Loir and Ducros \cite[(3.1.6)]{Formes_Chambert_2012}.
Roughly, this means $F$ locally takes the form $\varphi(\val(f_1),\dots, \val(f_n))$ where $f_i$'s are local invertible analytic functions and $\varphi:U\to\mathbb R^m$ is a continuous map for the Euclidean topology for some open subset $U$ of $\mathbb R^n$.
By adding other constraints on the $\varphi$, one may define the notion of tropically piecewise-linear / $C^k$, etc.

\subsubsection{Description of the image of $F$. }
\label{sss_describe_F}

Fix $\bar q=(q_1,\dots, q_{n-1})$, and define
\[
S_{\bar q}:=\{(u_0,u_1)\in\mathbb R^2\mid (u_0,u_1,\bar q)\in F(Y)\}
\]
In other words, the image of $Y$ in $\mathbb R^{n+1}\equiv \mathbb R^2\times \mathbb R^{n-1}$ under $F$ is given by
\begin{equation}
	\label{F(Y)_mathfrak_B_eq}
\mathfrak B:=F(Y)=\bigcup_{\bar q\in\mathbb R^{n-1}} S_{\bar q}\times \{\bar q\} 
\end{equation}

It suffices to describe each $S_{\bar q}$. 
Just like \S \ref{ss_Gross_fib}, let $\Pi\subset \mathbb R^{n-1}$ be the tropical hyperplane associated to $\min\{0,\bar q\}$, consisting of those points for which $\min\{0,\bar q\}$ is attained twice.
Just as (\ref{a_01_eq}), we define
\begin{equation*}
	a_0(\bar q)=\min\{0,\bar q\}-\psi_0(\bar q) \qquad \text{and} \qquad a_1(\bar q)=\psi_0(\bar q)
\end{equation*}

Let $z=(x_0,x_1,y_1,\dots, y_{n-1})$ be an arbitrary point in $Y$ with $q_k:=\val(y_k)$ for $1\le k<n$. We write $p=(u_0,u_1,\bar q)=F(z) $ for the image point.

\begin{enumerate} 
	\item [(i)]
Assume $\bar q\notin \Pi$. Then, $\val(x_0)+\val(x_1)=\val(1+y_1+\cdots +y_{n-1})=\min\{0,\bar q\}$. Eliminating $\val(x_0)$, we get
$
F_0(z)= \min\{0,\bar q\}  +\min \{-\val(x_1), -\psi_0(\bar q)\}
$
and $F_1(z)=\min\{\val(x_1), \psi_0(\bar q)\}$.
Hence, 
the $S_{\bar q}$ is simply the image of the broken line in $\mathbb R^2$ given by
\begin{equation}
	\label{broken_line_s_barq_eq} 
s_{\bar q}:  \mathbb R\to\mathbb R^2 , \quad c\mapsto \big( \min \{0,\bar q\}+ \min\{-c, -\psi_0(\bar q)\} \ \ , \quad
\min\{c,\psi_0(\bar q)\}
\big)
\end{equation}
with a corner point at $(a_0(\bar q), a_1(\bar q))$. Here $c$ represents $\val(x_1)$.
We take a small neighborhood $\bar V$ of $\bar q$ in the complement of $\Pi$ in $\mathbb R^{n-1}$.
Given the $c$ with $s_{\bar q}(c)=(u_0,u_1)$ and a sufficiently small $\epsilon>0$, we can find a neighborhood $U$ of $p$ in $\mathfrak B$ that is homeomorphic to $V:=(c-\epsilon, c+\epsilon)\times \bar V$ in $\mathbb R^n$ via the various $s_{\bar q'}$ for $\bar q'\in \bar V$. Then, under this identification $U\cong V$, 
$F^{-1}(U)$ is isomorphic to $\trop^{-1}(V)$ by forgetting $x_0$, i.e. $z\mapsto (x_1, y_1,\dots, y_{n-1})$. In this way,
the fibration $F$ also agrees with $\trop$.
In conclusion, this means $p$ is an \textit{$F$-smooth point} in the sense of \S \ref{ss_NA_integrable_affinoid_torus}.

\item[(ii)]  Assume $\bar q\in \Pi$.
Recall $F_0(z)=\min\{ \val(x_0), a_0(\bar q)\}$ and $F_1(z)=\min\{\val(x_1),a_1(\bar q)\}$.
As $\val(x_0)+\val(x_1)\ge \min\{\val(1),\val(y_1),\dots,\val(y_{n-1})\}=a_0(\bar q)+a_1(\bar q)$, one of the following cases must hold:

\begin{enumerate}
	\item [(ii-a)] If $\val(x_0)<a_0(\bar q)$, then $\val(x_1)>a_1(\bar q)$. Hence, $F_0(z)=\val(x_0) \equiv u_0$, and $F_1(z)=a_1(\bar q)$.
	Find a neighborhood $U$ of $p$ in $\mathfrak B$ in the form $U=\{(u_0', a_1(\bar q'), \bar q') \mid  u_0' \in I, \bar q'\in \bar V\}$, where a neighborhood $\bar V$ of $\bar q$ and an open interval $I$ centered at $u_0$ are both chosen small enough so that $u_0'<a_0(\bar q')$ always holds. Let $V:=I\times \bar V \cong U$, and $F^{-1}(U)$ is isomorphic to $\trop^{-1}(V)$ by forgetting $x_1$. Therefore, $p$ is $F$-smooth.
	
	\item[(ii-b)] If $\val(x_1) < a_1(\bar q)$, then $\val(x_0) >a_0(\bar q)$. Hence,
	 $F_0(z)= a_0(\bar q)$, and $F_1(z)=\val(x_1) \equiv u_1$.
	 In a similar way we can show $p$ is also $F$-smooth.
	 
	 \item [(ii-c)]
	 If both $\mathsf v(x_0)\ge a_0(\bar q)$ and $\mathsf v(x_1)\ge a_1(\bar q)$, then $(u_0,u_1)=(F_0(z),F_1(z))=(a_0(\bar q), a_1(\bar q))$ is exactly the corner point of the broken line $s_{\bar q}$ in (\ref{broken_line_s_barq_eq}).
	 One can also check $p$ is not $F$-smooth.
	 
\end{enumerate}
Hence, the $S_{\bar q}$ is still given by the broken line $s_{\bar q}$ defined in the same way as (\ref{broken_line_s_barq_eq}).
Moreover, by (ii-a) and (ii-b), the set of $F$-smooth points in $S_{\bar q}\cong S_{\bar q}\times \{\bar q\}$ include the union of
$
S_{\bar q}^+:= (-\infty, a_0(\bar q)) \times \{a_1(\bar q)\} 
$
and
$
S_{\bar q}^-:=\{a_0(\bar q)\} \times (-\infty, a_1(\bar q)) 
$.
\end{enumerate}

Combining (i) and (ii) above, we have proved the following structural result:

\begin{thm}
	\label{affinoid_torus_away_from_thm}
	The map $F$ restricts to an affinoid torus fibration over $\mathfrak B_0 \equiv \mathfrak B\setminus \hat\Delta$, where
\[
	\hat \Delta := \bigsqcup_{\bar q\in \Pi}  \{ \big(a_0(\bar q), a_1(\bar q)\big)\} \times \{\bar q\}
\]
\end{thm}

\subsection{Definition of $f$}
\label{ss_defn_f_B_fibration}

Notice that if $\jmath: \mathfrak U  \xhookrightarrow{} \mathfrak B_0$ is a topological embedding, then $\jmath^{-1} \circ F$ is also an affinoid torus fibration on its domain, by definition (\S \ref{ss_NA_integrable_affinoid_torus}).
For the base $B\equiv \mathbb R^n$ in \S \ref{ss_Gross_fib} and the $j:B\to\mathbb R^{n+1}$ in \S \ref{ss_j_homeomorphism}, a comparison between \S \ref{sss_describe_j} and \S \ref{sss_describe_F} implies the following

\begin{lem}
	\label{j(B)_lem}
	$j(B)$ agrees with the open subset \begin{equation}
		\label{hat_B_eq}
		\hat B:=\{ (u_0,u_1, \bar q)\in \mathfrak B \mid u_1>0\} \subset \mathbb R^{n+1}
	\end{equation}
Moreover, we have $j(\Delta)=\hat\Delta$.
In particular, $B=j^{-1}(\hat B)$ and $B_0=j^{-1}(\hat B \setminus \hat \Delta)$.
\end{lem}

\begin{proof}
	It suffices to compare the $\bar q$-slices for any fixed $\bar q\in \mathbb R^{n-1}$.
	First, let's indicate the coincidence between the $r_{\bar q}$ in (\ref{jmath_eq}) and the $s_{\bar q}$ in (\ref{broken_line_s_barq_eq}). The only difference is the domains are $(0,+\infty)$ and $\mathbb R$ respectively, and $R_{\bar q}=\{(u_0,u_1)\in S_{\bar q}\mid u_1>0\}$. Finally, using (\ref{j(B)_explicit}) and (\ref{F(Y)_mathfrak_B_eq}) completes the proof.
\end{proof}

Define 
$
	\mathscr Y :=F^{-1}(\hat B)
$, and it is exactly given by $\val(x_1)>0$ or equivalently $|x_1|<1$. In particular, $j(B)=F(\mathscr Y)$.
Using the topological embedding $j$ in (\ref{j_homeo_eq}), we define $f:= j^{-1}\circ F|_{\mathscr Y}$. By Theorem \ref{affinoid_torus_away_from_thm}, $f_0:=f|_{B_0}$ is an affinoid torus fibration. 
We set $\mathscr Y_0:=f_0^{-1}(B_0)$.
It is not hard to show $\mathscr Y_0$ is Zariski dense in $Y$ (e.g. by dimension reasons; compare also \cite{payne2009fibers}).
Notably, our construction of $f$ here is purely non-archimedean and does not use any Floer theory so far.
\begin{equation}
	\label{Y_analytic_domain_eq}
\xymatrix{
 			& \mathscr Y \ar@{-->}[dl]_f \ar[d]^F   \\
 B\ar[r]^j_{\cong} & \hat B 
}
\end{equation}

\section{Family Floer mirror construction: an easy-to-use review}
\label{s_family_review}

One may first skip \S \ref{s_family_review} for the mere affine geometry matching in \S \ref{ss_one_page_proof_intro} (see a reader guide in Remark \ref{floer_omit_rmk}).
For the legibility, we plan to substantially simplify the exposition of the mirror construction in \cite{Yuan_I_FamilyFloer} in an easy-to-use way.
These costs we pay are briefly explained in Remark \ref{rigor_affinoid_viewpoint_chart_rmk},
\ref{rigor_affinoid_gluing_rmk},  \ref{rigor_superpotential_rmk} for serious readers.

\subsection{Statement}
Let $(X,\omega, J)$ be a K\"ahler manifold of real dimension $2n$.
Suppose there is a Lagrangian torus fibration $\pi_0:X_0 \to B_0$ on some open domain $X_0\subset X$; we require it is \textit{semipositive} in the sense that there is no holomorphic stable disk of negative Maslov index bounding a Lagrangian fiber.
By \cite[Lemma 3.1]{AuTDual}, every special Lagrangian (or graded Lagrangian) satisfies this condition.
Further, we require that all Lagrangian fibers are weakly unobstructed (see e.g. \cite[Page 7]{Au_Special}) in the sense that their associated minimal model $A_\infty$ algebras have vanishing \textit{obstruction ideal} (slightly different from the Maurer-Cartan equations in the literature).
Thanks to Solomon \cite{Solomon_Involutions}, a nice sufficient condition is when each Lagrangian fiber is preserved by an anti-symplectic involution $\varphi$ (see also \cite{Solomon_Symmetry_Lag}). For example, the Gross Lagrangian fibration in \S \ref{ss_Gross_fib} or in (\ref{pi_intro_eq}) admits the involution given by the complex conjugations $z_i\mapsto \bar z_i$.
In general, such an involution $\varphi$ gives a pairing on $\pi_2(X, L_q)$ via $\beta\leftrightarrow -\varphi_*\beta$, inducing a pairwise canceling for the obstruction formal power series.
Beware that it does not mean the virtual counts of Maslov-0 disks vanish, and they do still contribute to the homological perturbations for the minimal model $A_\infty$ algebras and the wall-crossing $A_\infty$ homotopy equivalence.

Now, the family Floer mirror construction in \cite{Yuan_I_FamilyFloer} can be stated as follows:

\begin{thm}
	\label{Main_theorem_thesis_thm}
	Given $(X,\pi_0)$ as above, there is a triple
	$(X_0^\vee,W_0^\vee, \pi_0^\vee)$
	consisting of a non-archimedean analytic space $X_0^\vee$ over $\Lambda$, a global analytic function $W_0^\vee$, and a {\textit{dual affinoid torus fibration}} $\pi_0^\vee: X_0^\vee\to B_0$ such that
	
	\begin{enumerate}[(a)]
		\item The non-archimedean analytic structure of $(X_0^\vee, W_0^\vee, \pi_0^\vee)$ is unique up to isomorphism.
		\item The integral affine structure on $B_0$ induced by $\pi_0^\vee$ coincides with the one induced by $\pi_0$
		\item The set of closed points in $X_0^\vee$ coincides with 
\begin{equation}
	\label{set_describe_mirror}
	\textstyle
	\bigcup_{q\in B_0} H^1(L_q; U_\Lambda)
\end{equation}
		where $U_\Lambda$ is the unit circle in the Novikov field $\Lambda$.
	\end{enumerate}
\end{thm}


Beware that the mere homotopy invariance of Maurer-Cartan sets is quite insufficient to develop the \textit{analytic topology} on $X_0^\vee$
for which we must seek more structure and information.
In our specific SYZ context, the Maurer-Cartan sets are simply $H^1(L_q;\Lambda_0)$, which can at most give certain set-theoretic or local approximation. This is one key point missing in \cite{Tu}; see Remark \ref{rigor_affinoid_viewpoint_chart_rmk}, \ref{rigor_affinoid_gluing_rmk} for more discussions.
Indeed, the new ud-homotopy theory in \cite{Yuan_I_FamilyFloer} is necessary to upgrade the conventional Maurer-Cartan picture to a higher level, matching adic-convergent formal power series instead of just bijection of sets.
The virtual counts of Maslov-0 disks lead to an analytic space structure on the above fiber-wise union (\ref{set_describe_mirror}) of the space of $U_\Lambda$-local systems.
Moreover, the counts of Maslov-2 disks give rise to the global potential function $W_0^\vee$ on $X_0^\vee$.
%

\subsection{Local affinoid tropical charts}
\label{ss_FF_local_affinoid_tropical_chart}

Let $U\subset B_0$ be a contractible open subset, and choose a point $q_0$ near $U$ ($q_0\notin U$ is possible).
We require $U$ is sufficiently small and $q_0$ is sufficiently close to $U$ so that the reverse isoperimetric inequalities hold uniformly over a neighborhood of $U\cup\{q_0\}$.
Let 
\[
\chi: (U,q_0) \xrightarrow{\cong} (V,c)\subset \mathbb R^n
\]
be a (pointed) integral affine coordinate chart such that $\chi(q_0)=c$. Then, we have an identification
\begin{equation}
	\label{affinoid_torus_chart}
\tau:	(\pi_0^\vee)^{-1}(U) \xrightarrow{\cong} \trop^{-1}(V-c)
\end{equation}
with $\trop\circ \ \tau=\chi\circ \pi_0^\vee$. Let's call $\tau$ a (pointed) \textit{affinoid tropical chart} as in \S \ref{ss_NA_integrable_affinoid_torus}.
Recall that the left side is the disjoint union $(\pi_0^\vee)^{-1}(U)\equiv \bigcup_{q\in U} H^1(L_q; U_\Lambda)$ set-theoretically.
A closed point $ \mathbf y$ in the dual fiber $H^1(L_q; U_\Lambda)$ can be viewed as a group homomorphism $\pi_1(L_q)\to U_\Lambda$ (i.e. a flat $U_\Lambda$-connection modulo gauge equivalence), and we have the natural pairing 
\begin{equation}
	\label{connection_view}
	\pi_1(L_q)\times H^1(L_q; U_\Lambda)\to U_\Lambda, \qquad  (\alpha, \mathbf y) \mapsto \mathbf y^\alpha
\end{equation}
Write $\chi=(\chi_1,\dots, \chi_n)$, and it gives rise to a continuous family $e_i=e_i(q)$ of $\mathbb Z$-bases of $\pi_1(L_q)$ for all $q\in U$, c.f. (\ref{local_system_H_2_H_1_eq}).
Then, the corresponding affinoid tropical chart $\tau$ has a very concrete description:
\begin{equation}
	\label{affinoid_torus_explicit_eq}
\tau(\mathbf y)= (T^{\chi_1(q)} \mathbf y^{e_1(q)},\dots, T^{\chi_n(q)} \mathbf y^{e_n(q)} )
\end{equation}

\begin{rmk}
	\label{rigor_affinoid_viewpoint_chart_rmk}
The above picture is oversimplified. To develop the analytic topology, we must study local affinoid spaces or equivalently the corresponding affinoid algebras.
Specifically, in the above (\ref{affinoid_torus_chart}), we may first assume $V=\chi(U)$ is a rational polytope in $\mathbb R^n$.
Given a base point $q_0$, any $q\in U$ can be viewed as a vector, denoted by $q-q_0$, in $H^1(L_{q_0}; \mathbb R)\cong T_{q_0} B$.
Instead of (\ref{affinoid_torus_chart}), we should more precisely identify $(\pi_0^\vee)^{-1}(U)$ with the maximal ideal spectrum (or the multiplicative seminorm spectrum) of the \textit{polyhedral affinoid algebra} $\Lambda\langle U, q_0 \rangle$ \cite{EKL}. It consists of
the Laurent formal power series inside $\Lambda[[\pi_1(L_{q_0})]]$	that have the form $\sum_{\alpha\in\pi_1(L_{q_0})} c_\alpha Y^\alpha$ with $c_\alpha\in\Lambda$ and $\val(c_\alpha)+\langle \alpha, q-q_0 \rangle \to \infty$ as $|\alpha|\to\infty$ for any $q \in U$. Here $Y$ is a symbol and $Y^\alpha$ are the monomials.
Now,
a closed point $\mathbf y\in H^1(L_{q}; U_\Lambda)$ for some $q\in U$ corresponds to the maximal ideal
in $\Lambda\langle U, q_0\rangle$ generated by $Y^\alpha- T^{\langle \alpha, q-q_0\rangle }\mathbf y^{\alpha(q)}$ for all $\alpha\in\pi_1(L_{q_0})$, where $\alpha(q)\in \pi_1(L_q)$ denotes the induced class; c.f. (\ref{local_system_H_2_H_1_eq}).
Finally, we must choose $U$ to be sufficiently small, subject to Groman-Solomon's reverse isoperimetric inequality \cite{ReverseI, ReverseII} for the non-archimedean convergence. We also need to generalize it to a uniform version as in \cite{Yuan_I_FamilyFloer}; otherwise, as this inequality depends on the base point, we cannot ensure the convergence for the analytic gluing (c.f. Remark \ref{rigor_affinoid_gluing_rmk}).
\end{rmk}

\subsection{Gluing}
\label{ss_FF_gluing_maps}
In the construction of Theorem \ref{Main_theorem_thesis_thm}, we start with various local affinoid tropical charts as above, and then we can develop transition maps (or call gluing maps) among them in a choice-independent manner. 
This process encodes the quantum corrections of the pseudo-holomorphic disks bounded by smooth $\pi$-fibers but possibly meeting the singular $\pi$-fibers (Red disks in Figure \ref{figure_area_sing_fiber}).

Let's take two \textit{pointed} integral affine charts. Replacing the two domains by their intersection, we may assume the two charts have the same domain $U\subset B_0$, but the base points may be different and outside of $U$. 
Namely, as before in Section \ref{ss_FF_local_affinoid_tropical_chart}, we take two pointed integral affine charts 
\[
\chi_a=(\chi_{a1},\dots, \chi_{an}):(U,q_a) \to (V_a,c_a)
\]
for $a=1,2$. Then, $\upchi:=\chi_2\circ \chi_1^{-1}:V_1\to V_2$ is an integral affine transformation.
Due to (\ref{affinoid_torus_chart}) above, we have two affinoid tropical charts on the same domain:
\[
\tau_a: (\pi_0^\vee)^{-1}(U)\to \trop^{-1}(V_a-c_a)
\]
 such that $\chi_a\circ \pi_0^\vee=\trop\circ \ \tau_a$ for $a=1,2$.


The transition map between the two charts $\tau_1$ and $\tau_2$ is an automorphism map $\phi$ that, concerning the Fukaya's trick, captures the wall-crossing information of a Lagrangian isotopy from $L_{q_1}$ to $L_{q_2}$.
In brief, we can view it as a fiber-preserving map:
\begin{equation}
	\label{gluing_map_phi_domain_target}
	\phi: \bigcup_{q\in U} H^1(L_q; U_\Lambda)\to \bigcup_{q\in U} H^1(L_q;U_\Lambda) 
\end{equation}
But, be careful, the source and the target of $\phi$ should correspond to the two different affinoid tropical charts $\tau_1$ and $\tau_2$ respectively, although they are set-theoretically the same. In particular, the base points $q_a$ for $\tau_a$ matter a lot for the analytic structure.
Indeed, by the two affinoid tropical charts $\tau_1$ and $\tau_2$, the gluing map $\phi$ can be regarded an analytic map between open subdomains in $(\Lambda^*)^n$ as follows:
\begin{equation}
	\label{Phi_review_eq}
	\Phi:= \tau_2\circ \phi \circ \tau_1^{-1}:
	\trop^{-1}(V_1-c_1)\to \trop^{-1}(V_2-c_2)
\end{equation}

By definition, if $\mathbf y $ is a point in $ H^1(L_q; U_\Lambda)$ for some $q\in U$, then $\tilde { \mathbf y} :=\phi(\mathbf y)$ is a point in the same fiber $H^1(L_q;U_\Lambda)$ and is subject to the following condition:
\begin{equation}
	\label{gluing_map_pointwise}
	\tilde{\mathbf y}^\alpha= \mathbf y^\alpha \exp \langle
	\alpha, \pmb {\mathfrak F}(\mathbf y) 
	\rangle
\end{equation}
where we use the pairing in (\ref{connection_view}) and $\pmb {\mathfrak F}$ is a vector-valued formal power series\footnote{Abusing the notations, this really means $\mathfrak T^*\pmb {\mathfrak F}$ for the $\pmb{\mathfrak F}\in \Lambda[[\pi_1(L_{q_1})]]\hat\otimes H^1(L_{q_1})$ in Remark \ref{rigor_affinoid_gluing_rmk} and for the homomorphism $\mathfrak T: Y^{\alpha(q)}\mapsto T^{\langle \alpha, q-q_1\rangle} Y^{\alpha(q_1)}$. But, we often suppress this to make the notations simpler.} in $ \Lambda[[\pi_1(L_{q})]]\hat\otimes H^1(L_{q})$.
Roughly, the $\pmb{\mathfrak F}$ is decided by the virtual counts of Maslov-0 disks\footnote{Unless we use Fukaya's trick, one can roughly think only $J$-holomorphic curves for a fixed $J$ are considered.} along a Lagrangian isotopy from $L_{q_1}$ to $L_{q_2}$.
The existence and uniqueness of such a $\pmb{\mathfrak F}$ is proved in \cite{Yuan_I_FamilyFloer}.
By definition, the Novikov coefficients of $\pmb {\mathfrak F}$ have positive valuations, and one can prove the $\exp\langle \alpha, \pmb{\mathfrak F}\rangle$ has valuation zero for sufficiently small $U$ by the reverse isoperimetric inequality again. In particular, $\phi$ preserves the fibers.

\begin{rmk}
	\label{rigor_affinoid_gluing_rmk}
	Again, we must work with the category of affinoid algebras to be completely rigorous.
	The gluing map $\phi\equiv \psi^*$ comes from an affinoid algebra homomorphism:
	\begin{equation}
		\label{affinoid_algebra_perspective_homomorphism} 
		\psi: \Lambda \langle U, q_2\rangle \to \Lambda\langle U, q_1\rangle, \qquad
		Y^{\alpha(q_2)} \mapsto T^{\langle \alpha, q_1-q_2\rangle} Y^{\alpha(q_1)} \exp\langle \alpha, \pmb {\mathfrak F}(Y)\rangle
	\end{equation}
	for 
	\begin{equation}
		\label{definition_formula_F_ell_eq}
		\pmb {\mathfrak F}=\sum_{\mu(\beta)=0} T^{E(\beta)} Y^{\partial\beta} \f_{0,\beta}
	\end{equation}
where $\f=\{\f_{k,\beta}\} \ (k\ge 0, \beta\in\pi_2(X,L_q))$ is, up to Fukaya's trick, \textit{some} $A_\infty$ (ud-)homotopy equivalence between two $A_\infty$ algebras associated to $L_{q_1}$ and $L_{q_2}$.
We can finally check that the simplified description (\ref{gluing_map_pointwise}) agrees with (\ref{affinoid_algebra_perspective_homomorphism}) using the perspective of Remark \ref{rigor_affinoid_viewpoint_chart_rmk}.

The idea of finding (\ref{affinoid_algebra_perspective_homomorphism}) is to use the coordinate change in \cite[(1.6)]{FuCyclic} to the morphism \cite[(3.6.37)]{FOOOBookOne}, and it is first discovered by J. Tu in \cite{Tu}. 
But, we cannot just naively use the homotopy invariance of Maurer-Cartan sets to get the analytic topology on $X_0^\vee$. Indeed, the idea of Maurer-Cartan invariance is overall correct but need to be carried out in a more precise level, matching adic-convergent formal power series rather than just bijections of sets.
	For this, we need the stronger ud-homotopy in \cite{Yuan_I_FamilyFloer}; we also need a uniform version of reverse isoperimetric inequality (c.f. Remark \ref{rigor_affinoid_viewpoint_chart_rmk}) for the non-archimedean convergence issues when we move between adjacent local tropical charts.
	
Finally, note that the $A_\infty$ morphism $\f$ is obtained by a parameterized moduli space of holomorphic disks and is highly choice-sensitive. But surprisingly, the gluing map $\phi$ is actually unchanged for \textit{any} such $A_\infty$ (ud-)homotopy equivalence. This is missed in \cite{Tu} but is proved in \cite{Yuan_I_FamilyFloer} by the ud-homotopy relations.
In our opinion, the choice-independence of the gluing maps is the cornerstone of everything about the family Floer mirror construction, including the results in this paper as well as those in \cite{Yuan_e.g._FamilyFloer} \cite{Yuan_c_1}.
\end{rmk}

\subsection{Void wall-crossing}

Let $B_1\subset B_0$ be a contractible open set.
Let $B_2=\{x\in B_0\mid \dist (x, B_1) < \epsilon\}$ be a slight thickening of $B_1$ in $B_0$. We assume it is also contractible and $\epsilon>0$ is a sufficiently small number so that the estimate constant in the reverse isoperimetric inequalities for any Lagrangian fiber over $B_1$ exceeds $\epsilon$ uniformly (c.f. Remark \ref{rigor_affinoid_viewpoint_chart_rmk}, \ref{rigor_affinoid_gluing_rmk}, and \cite{Yuan_I_FamilyFloer}).
Then, we have:

\begin{prop}
	\label{trivial_translation_prop}
	Let $\chi: B_2\xhookrightarrow{} \mathbb R^n$ be an integral affine coordinate chart. If for every $q\in B_1$, the Lagrangian fiber $L_{q}$ bounds no non-constant Maslov index zero holomorphic disk, then there is an affinoid tropical chart $(\pi_0^\vee)^{-1} (B_2) \cong \trop^{-1}(\chi(B_2))$.
\end{prop}

\begin{proof}
	First, since $B_2$ is contractible, we can first single out a fixed pointed integral affine chart $\chi: (B_2,  q_0) \to (V,  c) \subset \mathbb R^n$ for some point $q_0\in B_2$.
	Next, we can cover $B_2$ by pointed integral affine coordinate charts $\chi_i: (U_i, q_i) \to (V_i, v_i)$, $ i\in \mathcal I$. We may require $\chi_i=\chi|_{U_i}$ and the diameters of $U_i$ are less than $\epsilon$. In particular, we may require all $q_i$'s are contained in $B_1$, and there will be no Maslov-0 disks along a Lagrangian isotopy among the fibers between any pair of $q_i$'s inside $B_1$.
	On the other hand, just like (\ref{affinoid_torus_chart}), we have many affinoid tropical charts $\tau_i: (\pi_0^\vee)^{-1}(U_i)\cong \trop^{-1}( V_i-v_i)$.
	The gluing maps among these tropical charts take the form as in (\ref{Phi_review_eq}). However, due to the non-existence of the Maslov-0 holomorphic disks, they have no twisting terms and take the simplest form $y_i\mapsto T^{c_i} y_i$. In conclusion, we can get a single affinoid tropical chart by gluing all these $\tau_i$'s.
\end{proof}

\subsection{Superpotential}
\label{ss_review_superpotential}
Assume $\beta\in \pi_2(X,L_{q_0})$ has Maslov index two, i.e. $\mu(\beta)=2$, and it also induces $\beta \equiv \beta(q)\in \pi_1(L_q)$ for any $q$ in a small contractible neighborhood of $q_0$ in $B_0$.
Denote by $\mathsf n_\beta\equiv \mathsf n_{\beta(q)}$ the corresponding \textit{open Gromov-Witten invariant}.\footnote{Briefly, it is the `count' of Maslov-2 holomorphic disks. Specifically, it comes from the $A_\infty$ algebra associated to $L_q$, and in general, we need to go to its minimal model $A_\infty$ algebra to define it. See \cite{Yuan_I_FamilyFloer}.}
It depends on the base point $q$ and the almost complex structure $J$ in use. For our purpose, unless the Fukaya's trick is applied, we always use the same $J$ in this paper. Then, due to the wall-crossing phenomenon, one may roughly think the numbers $\mathsf n_{\beta(q)}$ will vary dramatically in a discontinuous manner when we move $q$.

Now, we describe the superpotential $W^\vee:=W_0^\vee$ in Theorem \ref{Main_theorem_thesis_thm}.
Fix a pointed integral affine chart $\chi:(U,q_0)\to (V,c)$, and pick an affinoid tropical chart $\tau$ that covers $\chi$ as in (\ref{affinoid_torus_chart}).
Recall that the domain $U$ must be sufficiently small.
Then, the local expression of $W^\vee$ with respect to $\tau$ is given by 
\begin{equation}
	\label{superpotential_in_chart_pointwise_eq}
	W^\vee|_\tau : \bigcup_{q\in U} H^1(L_q; U_\Lambda)\to \Lambda,
	\qquad
	\mathbf y\mapsto \sum_{\beta\in \pi_2(X,L_q), \mu(\beta)=2} T^{E(\beta)} \mathbf y^{\partial\beta} \mathsf n_{\beta(q_0)}
\end{equation}
where $\mathbf y\in H^1(L_q; U_\Lambda)$ for any $q\in U$ and we use $\mathsf n_{\beta(q_0)}$ for the fixed $q_0$.
Alternatively, by (\ref{affinoid_torus_chart}), one may think of $W^\vee|_\tau$ as
\[
\mathcal W_\tau \equiv W^\vee \circ \tau^{-1}  : \trop^{-1}(V-c)\to (\pi_0^\vee)^{-1}(U)\to \Lambda
\]

The delicate story of the wall-crossing uncertainty can be well narrated by the gluing maps (\ref{gluing_map_pointwise}) among the atlas of various affinoid tropical charts (\ref{affinoid_torus_chart}) in view of Theorem \ref{Main_theorem_thesis_thm}. Specifically, we take two pointed integral affine charts $\chi_a: (U,q_a)\to (V_a, c_a)$ and two corresponding affinoid tropical charts $\tau_a: (\pi_0^\vee)^{-1}(U)\to \trop^{-1}(V_a-c_a)$ for $a=1,2$ as before in \S \ref{ss_FF_gluing_maps}. Let $\phi$ be the gluing map from the chart $\tau_1$ to the $\tau_2$ as in (\ref{gluing_map_phi_domain_target}, \ref{gluing_map_pointwise}).
Then, we must have $	W^\vee|_{\tau_2} (\phi(\mathbf y)) = W^\vee|_{\tau_1} (\mathbf y)$. Equivalently, if we set $\Phi=\tau_2\circ \phi \circ \tau_1^{-1}$ like (\ref{Phi_review_eq}), this means
\begin{equation}
	\label{superpotential_gluing_map_eq}
\mathcal W_{\tau_2}(\Phi(y))=\mathcal W_{\tau_1}(y)
\end{equation}

\begin{rmk}
	\label{rigor_superpotential_rmk}
	To make it precise, we need to view $W^\vee|_{\tau_a}$ as elements in the affinoid algebra $\Lambda\langle U, q_a\rangle$ as in Remark \ref{rigor_affinoid_viewpoint_chart_rmk} again.
	Denote them by $W_a$ ($a=1,2$) respectively. If $\psi$ is the affinoid algebra homomorphism such that $\phi\equiv \psi^*$ as in Remark \ref{rigor_affinoid_gluing_rmk}, then equation (\ref{superpotential_gluing_map_eq}) means $\psi(W_2)=W_1$. 
\end{rmk}

\subsection{Maslov-0 determinism}
\label{ss_Maslov_zero_determine}

In our Floer-theoretic mirror construction, the counts of the Maslov-0 disks are overwhelmingly more important than that of Maslov-2 disks.
Although the mirror superpotential $W^\vee$ is given by the counts of Maslov-2 disks locally in each chart, the local-to-global gluing among the various local expressions (\ref{superpotential_in_chart_pointwise_eq}) are all given by the counts of Maslov-0 disks.

In practice, there is a very useful observation as follows:
The Lagrangian fibration $\pi_0$ can be placed in different ambient symplectic manifolds, say $\overline X_1$ and $\overline X_2$. It often happens that the Maslov-0 disks are the same in both situations, and then the mirrors associated to $(X_i,\pi_0)$, $i=1,2$, in Theorem \ref{Main_theorem_thesis_thm} can be denoted by $(X_0^\vee, W_i^\vee, \pi_0^\vee)$, \textit{sharing the same mirror analytic space $X^\vee_0$ and the same dual affinoid torus fibration $\pi_0^\vee:X_0^\vee\to B_0$} but having different superpotentials $W_i^\vee$.

The `Maslov-0 open Gromov-Witten invariant' should be all the counting as a whole rather than any single of them. All the virtual counts of the Maslov-0 disks along an isotopy, only taken together, can form an invariant. Roughly, it forms the ud-homotopy class of a morphism in the category $\UD$ in \cite{Yuan_I_FamilyFloer}.

\section{T-duality matching and dual singular fibers}
\label{s_T_duality}

By Theorem \ref{Main_theorem_thesis_thm},
we denote the mirror triple associated to $\pi_0$ (\S \ref{ss_Gross_fib}), placed in $X$ rather than $X_0$, by $(X_0^\vee, W_0^\vee, \pi_0^\vee)$. Alternatively, by Remark \ref{floer_omit_rmk}, the reader may first skip all of \S \ref{s_family_review}, \S \ref{ss_Clifford_Chekanov}, \S \ref{ss_gluing_with_symmetry}, \S \ref{ss_miror_space} to get some preliminary ideas.


\subsection{Affinoid tropical charts for the Clifford and Chekanov chambers}
\label{ss_Clifford_Chekanov}

In \S \ref{ss_action_coordinates_-}, we have introduced the two integral affine charts $\chi_\pm$ on $B'_\pm$.
By Proposition \ref{trivial_translation_prop}, we have the following two affinoid tropical charts on $(\pi)^{-1}(B_\pm')$: (c.f. (\ref{affinoid_torus_explicit_eq}))
\begin{equation}
	\label{tau_+_eq}
	\tau_+: \bigcup_{q\in B'_+} H^1(L_q; U_\Lambda) \to (\Lambda^*)^n, \qquad \mathbf y\mapsto \big(
	T^{q_1} \mathbf y^{\sigma_1},\dots, T^{q_{n-1}} \mathbf y^{\sigma_{n-1}}, T^{\psi_+(q)} \mathbf y^{\partial\beta_n}
	\big)
\end{equation}
\begin{equation}
	\label{tau_-_eq}
	\tau_- : \bigcup_{q\in B'_-} H^1(L_q; U_\Lambda) \to (\Lambda^*)^n, \qquad \mathbf y \mapsto \big(T^{q_1} \mathbf y^{\sigma_1},\dots, T^{q_{n-1}} \mathbf y^{\sigma_{n-1}}, T^{\psi_-(q)} \mathbf y^{\partial\hat\beta} \big)
\end{equation}
The images of $\tau_\pm$ in $(\Lambda^*)^n$ are just given by the explicit integral affine charts $\chi_\pm$ (\ref{chi_-_eq}, \ref{chi_+_eq}) as follows:
\begin{equation}
	\label{T_+-_eq}
T_\pm:=\trop^{-1}(\chi_\pm(B'_\pm))
\end{equation}
which are the analytic open subdomains in $(\Lambda^*)^n$. Clearly, we have
$
	\trop\circ \ \tau_\pm = \chi_\pm \circ \pi_0^\vee
$; that is,
\[
\xymatrix{
	(\pi_0^\vee)^{-1}(B'_\pm) \ar[rr]^{\tau_\pm} \ar[d]^{\pi_0^\vee} & & T_\pm \ar[d]^{\trop} \\
	B'_\pm \ar[rr]^{\chi_\pm} & & \chi_\pm (B'_\pm)
}
\]

\subsection{Gluing with a symmetry}
\label{ss_gluing_with_symmetry}

By Theorem \ref{Main_theorem_thesis_thm}, the above two local expressions over $B'_\pm$ must be glued by some nontrivial automorphisms over $B'_+\cap B'_-\equiv \bigsqcup_{1\le \ell\le n}\mathscr N_\ell$. We denote them by
\begin{equation}
	\label{phi_ell_eq}
	\phi_\ell : \bigcup_{q\in \mathscr N_\ell}  H^1(L_q; U_\Lambda) \to \bigcup_{q\in\mathscr N_\ell} H^1(L_q; U_\Lambda)
\end{equation}
Be cautious that, despite of the same underlying sets, the two sides should refer to the two affinoid tropical charts $\tau_+$ and $\tau_-$ separately; in fact, we adopt a simplified expression as in (\ref{gluing_map_phi_domain_target}).
Given a point $\mathbf y\in H^1(L_q; U_\Lambda)$, the image point $\tilde {\mathbf y}:=\phi_\ell(\mathbf y)$ is contained in $H^1(L_q; U_\Lambda)$ and satisfies (c.f. (\ref{gluing_map_pointwise}))
\begin{equation}\label{gluing_map_z}
\tilde{\mathbf y}^\alpha =\mathbf y^\alpha\exp \langle \alpha, \pmb{\mathfrak F}_\ell  (\mathbf y)  \rangle
\end{equation}
for a formal power series $\pmb{\mathfrak F}_\ell(Y)$. Basically, the mirror analytic space $X^\vee_0$ is completely determined by these gluing maps $\phi_\ell$. There is no general algorithm for the gluing maps, and only the existence and uniqueness are proved in \cite{Yuan_I_FamilyFloer}.
 But, in our case, we first have a natural fiber-preserving $T^{n-1}$-action (\ref{S_1_action}), and the $T^{n-1}$-symmetry makes the gluing maps much simpler: (c.f. \cite{Yuan_e.g._FamilyFloer} or \cite[Theorem 8.4]{AAK_blowup_toric})

\begin{lem}
	\label{F_gamma_lem}
	$\langle \sigma_k , \pmb {\mathfrak F}_\ell \rangle=0$, for $1\le k\le n-1$ and $1\le \ell \le n$. In particular, if we set $\tilde {\mathbf y}=\phi_\ell(\mathbf y)$, then
	\[
	\tilde{\mathbf y}^{\sigma_k}=\mathbf y^{\sigma_k}
	\]
\end{lem}
\begin{proof}
Consider the Lefschetz fibration $w=z_1\cdots z_n$.
Let $u$ be a holomorphic disk bounded by a $\pi$-fiber $L=L_q$ for some $q$ in the wall $H_\ell$, and $\gamma:=[u]$ has Maslov index 0. Then, one can show $w\circ u\equiv 0$ (c.f. \cite[Lemma 5.4]{AuTDual}). Thus, the boundary $\partial u$ is contained in the sub-torus $T':=L\cap \bigcup_i D_i \cong T^{n-1}$, and the evaluation map $\ev: \mathcal M_{1,\gamma}(L)\to L$ is supported in this sub-torus $T'$.
Recall that a monomial in $\pmb{\mathfrak F}_\ell$ is a class in $H^1(L)\cong H_{n-1}(L)$ contributed by the counts of Maslov-zero disks. Namely, it is given by the pushforward of the evaluation map of the moduli spaces and is therefore dual to $T'\cong T^{n-1}$ in $L\cong T^n$.
Since $\pi_1(T')=H_1(T')\cong \mathbb Z^{n-1}$ is generated by $\sigma_1,\dots ,\sigma_{n-1}$, this means
the $\pmb{\mathfrak F}_\ell$ vanishes along these directions. Finally, we use (\ref{gluing_map_z}).
\end{proof}

We can express the $\phi_\ell$'s explicitly with respect to the two affinoid tropical charts $\tau_\pm$. Define:
\begin{equation}
	\label{Phi_ell_eq}
	\Phi_\ell:= \tau_-\circ \phi_\ell \circ \tau_+^{-1}
	:
	T_{+}^\ell \to  T_{-}^\ell
\end{equation}
where $T_\pm^\ell:=\trop^{-1}(\chi_\pm(\mathscr N_\ell))$ are analytic open subdomains in $T_\pm\subset (\Lambda^*)^n$.
Specifically, there exist formal power series $f_\ell(y_1,\dots, y_{n-1})$ for $1\le \ell \le n$ such that
\begin{equation}
	\label{Phi_ell_explicit_eq}
	\Phi_\ell(y_1,\dots, y_n)=
	\begin{cases}
		(y_1,\dots, y_{n-1}, y_ky_n\exp(f_k(y_1,\dots, y_{n-1})) 
		& \text{if} \ (y_1,\dots, y_n) \in T_+^k, \ 1\le k<n \\
		(y_1,\dots, y_{n-1}, y_n \exp(f_n(y_1,\dots, y_{n-1})) 
		& \text{if} \ (y_1,\dots, y_n)\in T_+^n
	\end{cases}
\end{equation}
Indeed, the first $n-1$ coordinates are preserved by Lemma \ref{F_gamma_lem}. Since the boundary of any Maslov-0 disk is spanned by $\sigma_k$ ($1\le k<n$), the definition formula (\ref{definition_formula_F_ell_eq}) implies that each $f_\ell$ does not involve $y_n$.
We also recall that $\partial\hat\beta=\partial\beta_n$ over $\mathscr N_n$ but $\partial\hat\beta=\sigma_k+\partial\beta_n$ over $\mathscr N_n$.

\subsection{Mirror analytic space}
\label{ss_miror_space}
As there is no Maslov-2 holomorphic disks in $X=\mathbb C^n\setminus \mathscr D$ bounded by the $\pi$-fibers, the mirror Landau-Ginzburg superpotential vanishes $W^\vee\equiv 0$ identically.
But, as indicated in \S \ref{ss_Maslov_zero_determine}, we can choose some larger ambient symplectic manifold $\overline X$ without adding new Maslov-0 disks.
No matter what $\overline X$ is, the structure of the mirror affinoid torus fibration $(X_0^\vee, \pi_0^\vee)$ will stay the same. In contrast, there can be new Maslov-2 disks that give rise to a new mirror superpotential $W^\vee$ on $X_0^\vee$.

Here we are mainly interested in the case $\overline X=\mathbb C^n$. But, we will study many others like $\overline X=\mathbb {CP}^n$ in \S \ref{s_folklore}.
We place the Gross Lagrangian fibration $\pi_0:X_0\to B_0$ (\S \ref{ss_Gross_fib}) in $\mathbb C^n$. By \S \ref{ss_Maslov_zero_determine}, 
the mirror space and the affinoid torus fibration associated to $(X,\pi_0)$ and $(\mathbb C^n, \pi_0)$ is actually the same, denoted by $(X_0^\vee, \pi_0^\vee)$.
But, the latter is equipped with an extra superpotential $W^\vee:=W^\vee_{\mathbb C^n}$.

For $q\in B_+$, the fiber $L_q$ is Hamiltonian isotopic to a product torus in $(\mathbb C^*)^n$. It follows from \cite{Cho_Oh} that the open GW invariants (\S \ref{ss_review_superpotential}) are $\mathsf n_{\beta_j}=1$ for the disks $\beta_j$'s (\S \ref{ss_topological_disk}).
For $q\in B_-$, we use the maximal principle to show the only nontrivial open GW invariant is $\mathsf n_{\hat\beta}=1$ (see \cite{AuTDual} \cite[Lemma 4.31]{CLL12}).
Now, by (\ref{superpotential_in_chart_pointwise_eq}), the restrictions $W^\vee_{\pm}$ of $W^\vee$ on the two chambers $(\pi_0^\vee)^{-1}(B_\pm)$ are as follows:
\begin{equation}
	\label{W_C_eq}
	\begin{aligned}
		W^\vee_{+}(\mathbf y)
	&=	\textstyle
	\sum_{j=1}^n T^{E(\beta_j)} \mathbf y^{\partial\beta_j} \mathsf n_{\beta_j} 
	=
	 	\textstyle T^{\psi_+(q)} \mathbf y^{\partial\beta_n} \big(1+ \sum_{k\neq n} T^{q_k} \mathbf y^{\sigma_k}\big)
	  \\ 
	W^\vee_{-}(\mathbf y) &= T^{E(\hat\beta)} \mathbf y^{\partial\hat\beta}
	\mathsf n_{\hat\beta}
	= T^{\psi_-(q)} \mathbf y^{\partial\hat\beta}
	\end{aligned} 
\end{equation}

Moreover, in view of Proposition \ref{trivial_translation_prop}, both of them can be extended slightly to the thickened domains
$(\pi_0^\vee)^{-1}(B'_\pm)$.
Then, for the affinoid tropical charts $\tau_\pm$ (\ref{tau_+_eq}), we write 
$
	\mathcal W_{\pm}:= W^\vee_{\pm}\circ \tau_\pm^{-1}
$
and obtain:
\begin{equation}
	\label{W_C_explicit}
	\begin{aligned}
		\mathcal W_{+} (y) &= y_n(1+y_1+\cdots+y_{n-1}) && \text{if} \ y=(y_1,\dots, y_n) \in T_+\\
		\mathcal W_{-} (y) &= y_n && \text{if} \ y=(y_1,\dots, y_n) \in T_-
	\end{aligned}
\end{equation}

According to the wall-crossing property (\ref{superpotential_gluing_map_eq}), we must have
$\mathcal W_-(\Phi_\ell(y))=\mathcal W_+(y)$ for any $y\in T_+^\ell$. Along with (\ref{Phi_ell_explicit_eq}), this completely determines all $\Phi_\ell$ for any $1\le \ell \le n$ as follows:
\begin{equation}
	\label{Phi_ell_explicit_determined_eq}
\Phi_\ell: 
T_{+}^\ell \to T_{-}^\ell \quad
\quad
(y_1,\dots, y_n) \mapsto \big(y_1,\dots, y_{n-1}, y_n(1+y_1+\cdots +y_{n-1}) \big)
\end{equation}
Remark that although the $\Phi_\ell$'s have the same formula, the domains and targets differ and depend on $\ell$.
Note that
$\Phi_\ell\circ \tau_+=\tau_-\circ \phi_\ell $ and $\trop\circ \ \Phi_\ell =\upchi_\ell\circ \trop$.


In conclusion, the mirror analytic space $X_0^\vee$ is isomorphic to
	the quotient of the disjoint union $T_+\sqcup T_-$ modulo the relation $\sim$: we say $y\sim y'$ if there exists some $1\le \ell \le n$ such that $y\in T_{+}^\ell$, $y'\in T_{-}^\ell$, and $\Phi_\ell(y)=y'$.
That is to say, we have an identification
\begin{equation}
	\label{identification_Floer_use_eq}
	X_0^\vee \textstyle \cong T_+\sqcup T_-/ \sim
\end{equation}
which is the adjunction space obtained by gluing $T_+$ and $T_-$ via all these $\Phi_\ell$'s.
By (\ref{integral_affine_trans_ell_eq}), one can check 
$
\trop  \circ \ \Phi_\ell=  \chi_-\circ \chi_+^{-1}\circ \trop
$ on the domains, so the dual affinoid torus fibration $\pi_0^\vee:X_0^\vee \to B_0$ can be identified with the gluing of the two maps $\chi_\pm^{-1}\circ \trop$ via the $\Phi_\ell$'s.
Under this identification, if we write $\pi_0^\vee=(\pi^\vee_1,\dots, \pi^\vee_n)$, then for $y=(y_1,\dots, y_n)\in T_\pm$, one can check
\begin{equation}
	\label{identification_pi_check_property_eq}
\mathsf v(y_k) = \pi_k^\vee (y)  \quad (1\le k< n) \qquad \text{and} \quad \val(y_n)= \psi_\pm(\pi_0^\vee(y))
\end{equation}

	\[
	\xymatrix{
		& & & B_0
		\\
		T_+  
		\ar[rr]
		\ar@/^1pc/[rrru]^{\chi_+^{-1}\circ \trop} & & 
		X_0^\vee \equiv T_+\sqcup T_- / \sim
		\ar@{-->}[ru]^{\pi_0^\vee}
		& 
		\\
		\bigsqcup_\ell T_+^\ell
		 \ar[rr]^{\sqcup_\ell \Phi_\ell}
		 \ar@{^{(}->}[u]^{\mathrm{Incl}} & & T_- \ar[u] 	
		 \ar@/_1pc/[ruu]_{\chi_-^{-1}\circ \trop}
	}
	\]

\subsection{The analytic embedding $g$}
\label{ss_g_analytic_embedding}

	From now on, we will always identify the $X_0^\vee$ with $T_+\sqcup T_-/\sim$ and identify the $\pi_0^\vee$ with the one obtained as above (\ref{identification_Floer_use_eq}, \ref{identification_pi_check_property_eq}). See the above diagram.
Next, we define
\begin{equation}
	\label{g+}
	g_+:  T_+ \to \Lambda^2 \times (\Lambda^*)^{n-1}
\end{equation}
\[
(y_1,\dots,y_{n-1},y_n)
\mapsto
\left(
\frac{1}{y_n} \ , \   y_n \ h \ , \  y_1,\dots, y_{n-1}
\right)
\]
and define
\begin{equation}
	\label{g-}
	g_-: T_-\to   \Lambda^2 \times (\Lambda^*)^{n-1} 
\end{equation}
\[
(y_1,\dots,y_{n-1},y_n)\mapsto \left(
\frac{h}{y_n} \ ,  \ y_n \ , \  y_1,\dots, y_{n-1} \right)
\]

Recall that $T_\pm\subsetneq (\Lambda^*)^n$ correspond to the Clifford and Chekanov tori respectively.
For all $1\le \ell\le n$, it is direct to check that
$
g_+=g_-\circ \Phi_\ell
$
on their various domains.
Hence, by (\ref{identification_Floer_use_eq}), we can glue $g_\pm$ to obtain an embedding analytic map
\begin{equation}
	\label{g_analytic_map_eq}
g :X_0^\vee\to \Lambda^2\times (\Lambda^*)^{n-1}
\end{equation}
such that the following diagrams commute:
	\[
\xymatrix{
	& & &  \Lambda^2\times (\Lambda^*)^{n-1} 
	\\
	T_+  
	\ar[rr]
	\ar@/^1pc/[rrru]^{g_+} & & 
	X_0^\vee \equiv T_+\sqcup T_- / \sim
	\ar@{-->}[ru]^{g}
	& 
	\\
	\bigsqcup_\ell T_+^\ell
	\ar[rr]^{\sqcup_\ell \Phi_\ell}
	\ar@{^{(}->}[u]^{\mathrm{Incl}} & & T_- \ar[u] 	
	\ar@/_1pc/[ruu]_{g_-}
}
\]

The image $g(X_0^\vee)$ is clearly contained in the algebraic variety
$
Y$ defined by $x_0 x_1 =1+y_1+\cdots+y_{n-1}$.

\begin{rmk}
	The above formula of $g$ is given by Gross-Hacking-Keel in \cite[Lemma 3.1]{GHK_birational}. 
	However, the difference is that $T_\pm$ are merely analytic subdomains of the torus $(\Lambda^*)^n$ in the finer Berkovich topology, as opposed to the Zariski topology, making them inaccessible using just algebraic geometry.
	Finally, this inspires us to modify Kontsevich-Soibelman's model \cite[Page 44-45]{KSAffine} for the sake of integral affine structure matching (cf. \S \ref{ss_literature}).
	In turn, unlike \cite{GHK_birational}, we follow the non-archimedean perspective in \cite{KSAffine} more closely.
\end{rmk}

\begin{rmk}
	\label{floer_omit_rmk}
	\textit{The only place we use the family Floer theory \cite{Yuan_I_FamilyFloer} is the identification $X_0^\vee\cong T_+\sqcup T_-/\sim$ in (\ref{identification_Floer_use_eq}) together with the corresponding characterization for $\pi_0^\vee$. }
	The reader's guide is as follows: all of the analytic subdomains $T_\pm$ in (\ref{T_+-_eq}), the gluing relation $\sim$ from  $\Phi_\ell$ in (\ref{Phi_ell_eq}, \ref{Phi_ell_explicit_determined_eq}), the dual fibration $f$ in (\ref{Y_analytic_domain_eq}), and the analytic embedding $g$ in (\ref{g_analytic_map_eq}) can be defined directly, regardless of the whole \S \ref{s_family_review}.
	The key Theorem \ref{fibration_preserving_thm} below is also a purely non-archimedean statement.
	All these non-Floer-theoretic ingredients are already sufficient to prove a weaker version of Theorem \ref{Main_this_paper_fundamental_example} dropping the T-duality condition (b) in Definition \ref{SYZmirror_defn}.
Moreover, we stress again that it is still very difficult to achieve the various coincidences in Definition \ref{SYZmirror_defn} (a).
	Compare Remark \ref{only_place_Floer_intro_rmk} and the discussions around (\ref{explicit_formula_introduction_eq}).
\end{rmk}

\subsection{Fibration preserving}
\label{ss_fibration_preserving}

Recall
$
F=(F_0,F_1, G_1,\dots, G_{n-1}):Y\to\mathbb R^{n+1}$
is defined in \S \ref{ss_NA_integrable_affinoid_torus} as follows: for $z=(x_0,x_1,y_1,\dots,y_{n-1})$,
\[
\begin{aligned}
	F_0(z)&=\min\{  
	\val(x_0),   -\psi_0(\val(y_1),\dots, \val(y_{n-1}) )+\min\{0,\val(y_1),\dots, \val(y_{n-1})\}  \} \\
	F_1(z)&= \min \{ \val(x_1), \ \ \ \psi_0(\val(y_1),\dots, \val(y_{n-1})) \} \\
	G_k(z)&=\val(y_k) \qquad \text{for} \  1\le k< n
\end{aligned}
\]
The next is the key result that puts all the previous constructions together and will prove Theorem \ref{Main_this_paper_fundamental_example}.

\begin{thm}
	\label{fibration_preserving_thm}
	$F\circ g=j\circ \pi_0^\vee$. Namely, we have the following commutative diagram
	\[
	\xymatrix{
		X_0^\vee \ar[rr]^{g\quad } \ar[d]^{\pi_0^\vee} & & Y \ar[d]^F \\
		B_0 \ar[rr]^{j} & & \mathbb R^{n+1}
	}
	\]
\end{thm}

\begin{proof} 
Fix $\mathbf y$ in $X_0^\vee\equiv T_+\cup T_-/\sim$ (\ref{identification_Floer_use_eq}), and set $q=\pi_0^\vee(\mathbf y)$. Then,
$
j \circ \pi_0^\vee(\mathbf y)=j(q)=(\theta_0(q),\theta_1(q),  \bar q)
$
where (recalling (\ref{j_homeo_eq}))
\begin{align*}
	\theta_0(q)
&:=
\min\{ -\psi(q) , -\psi_0( \bar q ) \}  + \min \{0, \bar q \} 
\\
\theta_1(q) 
&:=
\min\{\ \ \ \psi(q) ,  \ \ \ \psi_0( \bar q )\}
\end{align*}
where $\psi(q)$ and $\psi_0(\bar q) \equiv \psi(\bar q,0)$ are given in (\ref{psi_global_in_B_eq},\ref{psi_0_eq}).
Recall that $\Pi$ is the tropical hypersurface associated to $\min\{0, \bar q\}$ (see \S \ref{ss_Gross_fib}).
We aim to check $(\theta_0(q),\theta_1(q),\bar q)$ always agrees with $F\circ g(\mathbf y)$:

\begin{enumerate}
	\item If $q\in B'_+$, then $\mathbf y$ is identified with a point $y=(y_1,\dots, y_n)$ in $T_+$. By (\ref{identification_pi_check_property_eq}), $\val(y_k)=q_k$ for $1\le k<n$ and $\val(y_n)=\psi_+(q)$.
	Therefore, as desired, we get $G_k(g_+(y))=q_k$, and by (\ref{psi_global_in_B_eq}),
	\begin{align*}
	F_0(g_+(y))
	=\min\{ 
	-\psi_+(q), -\psi_0(\bar q)+\min\{0,\bar q\}
	\} = \min\{-\psi(q),-\psi_0(\bar q)\} + \min\{0,\bar q\}=\theta_0(q)
	\end{align*}
However,
	\begin{align*}
	F_1(g_+(y))
	=
	\min\{ \psi_+( q) +\val(1+y_1+\cdots+y_{n-1}),  \ \ \ \psi_0(\bar q)\}
	\end{align*}
and we need to further deal with the ambiguity of $\val(1+y_1+\cdots+y_{n-1})$ as follows:
\begin{enumerate}
	\item [(1a)] If $\bar q\in \Pi$, then since $q\in B'_+$, we must have $q_n>0$. Using the non-archimedean triangle inequality and Lemma \ref{symplectic_area_increasing_lem} infers that
		$\psi_+(q)+\val(1+y_1+\cdots+y_{n-1}) \ge \psi_+(q)+\min \{0,\bar q\} \equiv \psi(q) > \psi_0(\bar q)$.
	Hence, 
	\[F_1(g_+(y))=\psi_0(\bar q)=\min\{\psi(q),\psi_0(\bar q)\}=\theta_1(q)
	\]
	\item [(1b)] If $\bar q\notin\Pi$, then the minimum of the values $\val(y_k)=q_k$ for $1\le k<n$ cannot be attained twice. Thus,
	\[
	\val(1+y_1+\cdots+y_{n-1})=\min\{ \val(1), \val(y_1),\dots, \val(y_{n-1})\}=\min\{0,\bar q\}
	\]
	and $F_1(g_+(y))=\theta_1(q)$.
\end{enumerate}
	
	\item If $q\in B'_-$, then $\mathbf y$ is identified with a point $y=(y_1,\dots, y_n)$ in $T_-$. By (\ref{identification_pi_check_property_eq}), $\val(y_k)=q_k$ for $1\le k<n$ and $\val(y_n)=\psi_-(q)$. Hence, as desired, we also get $G_k(g_-(y))=q_k$ and 
\[
F_1(g_-(y))=\min\{  \psi_-(q),  \psi_0(\bar q)  \}
\equiv \min\{ \psi(q),\psi_0(\bar q) \} 
=\theta_1(q) \]
However, in turn,
	\[
	F_0(g_-(y))= \min\{  - \psi_-(q) + \val(1+y_1+\cdots+y_{n-1}), -\psi_0(\bar q)+\min\{0,\bar q\}\}
	\]
has some ambiguity, and we similarly argue by cases:
	\begin{enumerate}
		\item [(2a)] If $\bar q\in \Pi$, then $q_n<0$. By the non-archimedean triangle inequality and Lemma \ref{symplectic_area_increasing_lem}, we similarly obtain $-\psi_-(q) + \val(1+y_1\cdots+y_{n-1}) \ge-\psi_0(\bar q)+ \min\{0,\bar q\}$. Hence 
		\[
		F_0(g_-(y))=-\psi_0(\bar q)+\min\{0,\bar q\}  = \min\{-\psi(q), -\psi_0(\bar q)\} +\min\{0,\bar q\}=\theta_0(q)
		\]
		\item [(2b)] If $\bar q\notin \Pi$, then we similarly get 
		$	\val(1+y_1+\cdots+y_{n-1})=\min\{0,\bar q\}$
		and $F_0(g_-(y))=\theta_0(q)$.
	\end{enumerate}
\end{enumerate}
(In the sense of Remark \ref{floer_omit_rmk}, all the proof here does not rely on any family Floer theory in \S \ref{s_family_review}.)
\end{proof}

\begin{proof}[Proof of Theorem \ref{Main_this_paper_fundamental_example}]
Recall the $f: \mathscr Y\to B$ constructed in \S \ref{ss_defn_f_B_fibration} satisfies that $F=j\circ f$. 
The key Theorem \ref{fibration_preserving_thm} above implies that the image $g(X_0^\vee)$ is exactly given by the total space of the affinoid torus fibration $F$ restricted over $j(B_0)$. In other words, this image coincides with the $\mathscr Y_0$ defined in \S \ref{ss_defn_f_B_fibration}.
Moreover, it also implies that
$
\pi_0^\vee= f_0\circ g
$, where $f_0:=f|_{B_0}$ is the restriction of $f$ over $B_0$.
\end{proof}

\subsection{Dual singular fiber is not a Maurer-Cartan set}
\label{ss_singular_discussion}

It has been long expected that, at least set-theoretically, the mirror dual fiber of a Lagrangian fiber $L$ should be the set $\mathcal {MC}(L)$ of Maurer-Cartan solutions (also known as the bounding cochains) for an $A_\infty$ algebra associated to $L$.
This is mostly correct for the smooth fibers, although we need to be more careful for the analytic topology as indicated in Remark \ref{rigor_affinoid_viewpoint_chart_rmk}, \ref{rigor_affinoid_gluing_rmk}.
This is also a basic point of the original family Floer homology program.
Thus, it is natural for us to believe the `dual singular fibers' are also the corresponding Maurer-Cartan sets.

\textit{Nevertheless, our result responds negatively to this expected Maurer-Cartan picture.}

We have obtained the non-archimedean analytic fibration $f$ over $B$, extending the affinoid torus fibration $f_0\cong \pi_0^\vee$ tropically continuously. Besides, as explained in \S \ref{sss_dual_singular_intro}, its construction is compatible with the above Maurer-Cartan picture over $B_0$ \cite{Yuan_I_FamilyFloer} and is meanwhile backed up by lots of evidence \cite{GHK_birational,Abouzaid_Sylvan,AAK_blowup_toric,Gammage_localSYZ,AuTDual,Au_Special,KSTorus,KSAffine}.
Therefore, the $f$-fibers over the singular locus $\Delta=B\setminus B_0$ are basically the only reasonable candidates for the `dual singular fibers'.
But unfortunately, they are indeed larger than the Maurer-Cartan sets as shown in (\ref{MC_not_eq}) below.

Let's elaborate this as follows.
For clarity, let's assume $n=2$; then, $B=\mathbb R^2$ and the singular locus $\Delta$ consists of a single point $0=(0,0)$. 
Let's also forget about the analytic topology again and just look at the sets of closed points.
Notice that
$
f^{-1}(0)=F^{-1}(j(0))=F^{-1}(0,-\psi_0, \psi_0)
$
where $\psi_0:=\psi_0(0)>0$ is represented by the symplectic area of a holomorphic disk bounding the immersed Lagrangian $L_0=\pi^{-1}(0)$ (Figure \ref{figure_area_sing_fiber}, yellow). By definition, this consists of points $(x_0,x_1,y)$ in the variety $Y=\{x_0x_1=1+y\}$ in $\Lambda^2\times \Lambda^*$
so that $\val(y)=0$, $\min\{\val(x_0), -\psi_0\}=-\psi_0$, and $\min\{\val(x_1), \psi_0\}=\psi_0$. 
The last two conditions mean $\val(x_0)\ge -\psi_0$ and $\val(x_1)\ge \psi_0$.
We can take the coordinate change $z_0:=T^{\psi_0} x_0$, $z_1:=T^{-\psi_0}x_1$ so that $z_0z_1=x_0x_1$. Then, the dual singular fiber becomes
\[
\mathbf S:= f^{-1}(0)=\{ (z_0,z_1,y)\in Y \mid   \val(z_0)\ge 0, \ \val(z_1)\ge 0, \ \val(y)=0\}
\]
Recall that $\Lambda_0=\{z\in \Lambda\mid \val(z)\ge 0\}$ denotes the Novikov ring and $\Lambda_+ =\{z\in \Lambda\mid \val(z)>0\}$ is its maximal ideal. Recall also that $U_\Lambda=\{z\in\Lambda\mid \val(z)=0\}$, so the condition $\val(y)=0$ means $y\in U_\Lambda \equiv \mathbb C^*\oplus \Lambda_+$.
Next, we decompose $\mathbf S$ by considering the two cases of the variable $y\in U_\Lambda$:

\begin{itemize}
	\item If $y\in -1+\Lambda_+$, then $\val(z_0)+\val(z_1)=\val(1+y)>0$, and $(z_0,z_1)\in \Lambda_0\times \Lambda_+ \cup \Lambda_+\times \Lambda_0$. In turn, such a pair $(z_0,z_1)$ determines $y=-1+z_0z_1$ in $-1+\Lambda_+$.
\item If $y\notin -1+\Lambda_+$, then $0\le \val(z_0)+\val(z_1)=\val(1+y)=0$, so $z_0,z_1\in U_\Lambda$. In turn, for all these pairs $(z_0,z_1)$, the number $y=-1+z_0z_1$ can run over all the values in $\mathbb C^*\setminus \{ -1\} \oplus \Lambda_+$. If we write $\bar z_0,\bar z_1\in\mathbb C^*$ for the reductions of $z_0,z_1$ with respect to $U_\Lambda \twoheadrightarrow \mathbb C^*$, then this says $\bar z_0\bar z_1\neq 1$.
\end{itemize}
Therefore,
\[
\mathbf S=\mathbf S_1 \sqcup \mathbf S_2
\]
where $\mathbf S_1$ and $\mathbf S_2$ correspond to the above two bullets and are thus given by
\begin{align*}
\mathbf S_1 &\cong \Lambda_0\times \Lambda_+ \cup \Lambda_+\times \Lambda_0 \\
\mathbf S_2 & \cong \{ (z_0, z_1)\in  U_\Lambda\times U_\Lambda  \mid \bar z_0\bar z_1\neq 1\} \cong U_\Lambda\times \big(\mathbb C^*\setminus \{-1\} \oplus \Lambda_+\big)
\end{align*}

Going back to the A side, we can define the Maurer-Cartan set $\mathcal {MC}(L_0)$ by the Akaho-Joyce theory \cite{AJImm}.
Moreover, Hong, Kim, and Lau \cite[\S 3.2]{hong2018immersed} prove that we only have
\begin{equation}
	\label{MC_not_eq}
\mathcal {MC}(L_0)= \mathbf S_1  \ \ \subsetneq f^{-1}(0)
\end{equation}
(compare also many related papers \cite{cho2021noncommutative,cho2017localized,cho2018gluing}).
In conclusion, as the points in $\mathbf S_2$ are missed,
the Maurer-Cartan picture fails over the singular locus. 
One possibility is that we need additional `deformation data' for the conventional Maurer-Cartan sets.
Moreover, as discussed in \S \ref{sss_dual_singular_intro}, the non-archimedean analytic topology may be more relevant for the singular part.

The framework of Definition \ref{SYZmirror_defn} offers a preliminary attempt, and should be sufficient when the mirror space is expected to be algebraic rather than transcendental.
But admittedly, an ultimate answer would require further researches.

\section{Generalizations}
\label{s_generalization}

By combining the ideas in \S \ref{s_A_side}, \ref{s_B_side}, \ref{s_T_duality} with some basics of tropical and toric geometry, we can obtain quite a lot of generalizations with little change of ideas.

Let $N\cong \mathbb Z^n$ be a lattice and $M=\Hom(N,\mathbb Z)$.
Set $N_{\mathbb R}=N\otimes \mathbb R$ and $M_{\mathbb R}=M\otimes \mathbb R$.
Let $\Sigma\subset N_{\mathbb R}$ be a simplicial smooth fan with all maximal cones $n$-dimensional. Then, the primitive generators $v_1,\dots, v_n$ of rays in a maximal cone form a $\mathbb Z$-basis of $N$. Denote by $v_1^*,\dots, v_n^*$ the dual basis of $M$, so $M_{\mathbb R}\cong \mathbb R^n$.
Denote the remaining rays in $\Sigma$ by $v_{n+1},\dots, v_{n+r}$ for a fixed $r\ge 0$.

Suppose the toric variety $\mathcal X_\Sigma$ associated to $\Sigma$ is Calabi-Yau.
This means there exists $m_0\in M$ such that $\langle m_0, v\rangle =1$ for any generator $v$ of a ray in $\Sigma$. It follows that $m_0=v_1^*+\cdots +v_n^*=(1,\dots, 1)$. If we set
$
v_{n+a}=\sum_{j=1}^n k_{aj}v_j
$
for $k_{aj}\in\mathbb Z$, then we have $\sum_{j=1}^n k_{aj}=1$. Denote the corresponding toric divisors by $D_1,\dots, D_{n+r}$. Then,
$D_i+\sum_{a=1}^r k_{ai} D_{n+a}  = (\chi^{v_i^*}) \sim 0 $ for $1\le i\le n$ and $\sum_{j=1}^{n+r}D_j=(\chi^{m_0})\sim 0$.

\subsection{Lagrangian fibration}
Let $w=\chi^{m_0}$ be the character of $m_0$, and define $\mathscr D=w^{-1}(1)$.
We equip $\mathcal X_\Sigma$ with a toric K\"ahler form $\omega$, and the corresponding moment map $\mu: \mathcal X_\Sigma\to P$ is onto an unbounded polyhedral $P=P_\omega$ in $M_{\mathbb R}$ described by a set of inequalities of the form 
\begin{equation}
	\label{ell_i_half_spaces_eq}
\ell_i(m):=\langle m, v_i\rangle +\lambda_i \ge 0, \qquad 1\le i\le n+r, \quad  m\in M_{\mathbb R}
\end{equation}
for some $\lambda_i\in\mathbb R$.
The sublattice $\bar N:=\{n\in N\mid \langle m_0,n\rangle=0\}$ has a basis $\sigma_k=v_k-v_n$ for $1\le k<n$.
Consider its dual $\bar M:= \Hom(\bar N, \mathbb Z) \equiv M/\mathbb Zm_0$.
We identify
$\bar M_{\mathbb R}:=M_{\mathbb R}/\mathbb Rm_0$ with a copy of $\mathbb R^{n-1}$ in $M_{\mathbb R}\cong \mathbb R^n$ consisting of $(m_1,\dots, m_n)$ with $m_n=0$.
Then, the projection $p:M_{\mathbb R}\to \bar M_{\mathbb R}$ takes the form $(m_1,\dots, m_n)\mapsto (m_1-m_n, \dots, m_{n-1}-m_n)$.
We can also show that $p$ induces a homeomorphism from $\partial P$ to $\bar M_{\mathbb R}$. Note that $\bar \mu:= p\circ \mu$ is the moment map associated to the action of the subtorus $T_{\bar N} \equiv (\mathbb R/2\pi \mathbb Z)\cdot \{ \sigma_1,\dots, \sigma_{n-1}\}$. The critical points of $\bar \mu$ are the codimension-two toric strata in $X$, so the image $\Pi$ in $\bar M_{\mathbb R}$ of them is the union of $\Delta_{ij}:=p(P_{ij})$ for all $i\neq j$ where $P_{ij}:=\{\ell_i=\ell_j=0\}$.
By (\ref{ell_i_half_spaces_eq}), we can explicitly describe each $\Delta_{ij}$ in $\bar M_{\mathbb R}\cong \mathbb R^{n-1}$ with coordinates $\bar q=(q_1,\dots,q_{n-1})$:
\[
\begin{cases}
	\Delta_{ij} =\{   q_i+\lambda_i=q_j+\lambda_j\} 		
	& \text{if} \ 1\le i<j<n \\
	\Delta_{ij}=\{ q_i+\lambda_i = \lambda_n \} 
	& \text{if} \ 1\le i<j=n \\
	\Delta_{ij}=\{ 	q_i+\lambda_i= \sum_{s=1}^{n-1}k_{j-n,s} \ q_s +\lambda_j \}							
	& \text{if} \ 1\le i <n <j  \\
	\Delta_{ij}=\{ 	\lambda_n= \sum_{s=1}^{n-1}k_{j-n,s} \ q_s +\lambda_j \}							
	& \text{if} \  i =n <j  \\
	\Delta_{ij}=\{ 	 \sum_{s=1}^{n-1}k_{i-n,s} \ q_s +\lambda_i= \sum_{s=1}^{n-1}k_{j-n,s} \ q_s +\lambda_j \}							
	& \text{if} \  n< i <j  \\
\end{cases}
\]
It turns out that $\Pi\equiv \bigcup_{i<j}\Delta_{ij}$ is the tropical hypersurface in $\mathbb R^{n-1}$ associated to the tropical polynomial
\begin{equation}
	\label{tropical_polynomial_eq}
h_{\mathrm{trop}}(q_1,\dots, q_{n-1})= \min\Big\{ \lambda_n,  \{q_k+\lambda_k\}_{1\le k<n}, \{\textstyle \sum_{s=1}^{n-1}k_{as}q_s+\lambda_{n+a}\}_{1\le a\le r}
\Big\}
\end{equation}
We will realize it as the tropicalization of the Laurent polynomial $h$ in (\ref{tropical_original_h_eq}) later.

Define $X= \mathcal X_\Sigma \setminus \mathscr D$, and the \textit{Gross special Lagrangian fibration} \cite{Gross_ex} is given by
\begin{equation}
	\label{pi_fibration_generalized_eq}
\pi=(\bar\mu, \log |w-1| ):X\to B
\end{equation}
which maps onto $B:=\bar M_{\mathbb R}\times \mathbb R\cong \mathbb R^n$ for the above-mentioned identification.
The discriminant locus of $\pi$ is $\Delta=\Pi\times \{0\}$, and define $B_0:=B\setminus \Delta$.
Let $\hat H_i=p(\ell_i^{-1}(0)\cap P)\times \{0\}$, and $H_i:=\hat H_i\setminus \bigcup_j \Delta_{ij}$. The set $H=\bigcup_i H_i$ in $B_0$ is called the \textit{wall} in the sense that the Lagrangian fiber $L_q:=\pi^{-1}(q)$ bounds a nontrivial Maslov-0 holomorphic disk if and only if $q\in H$.

\begin{rmk}
Although the recent developments of the SYZ conjecture (e.g. \cite{collins2021special, collins2020syz,Li2019syz,evans2021constructing}) focus mainly on the existence of special Lagrangian fibration, their techniques should be very useful to find graded or zero Maslov class Lagrangian fibrations as well (cf. \cite{neves2007singularities}, \cite[Example 2.9]{SeidelGraded}).
The latter is easier to achieve and already sufficient to ensure the existence of the dual affinoid torus fibration $\pi_0^\vee$ \cite{Yuan_I_FamilyFloer}.
\end{rmk}

\subsection{Action coordinates}
Let $\mathscr N_i$ be a sufficiently small neighborhood of $H_i$ in $B_0$.
Let $B_\pm$ be the open subset of $B_0$ consisting of those points whose last coordinate is $>0$ or $<0$. Set $B'_\pm= B_\pm \cup\bigcup_i \mathscr N_i$.

The fiber $L_q$ for $q\in B_+$ is of Clifford type and can deform into a product torus in $(\mathbb C^*)^n\cong N\otimes \mathbb C^*$. So, there is a canonical isomorphism $\pi_1(L_q)\cong N$. Also, $\pi_2( \mathcal X_\Sigma, L_q)$ is naturally isomorphic to $\mathbb Z^{n+r}$ via $\beta \mapsto (\beta\cdot D_i)_{1\le i\le n+r}$, and let $\{\beta_i\}$ denote the corresponding basis of $\pi_2(\mathcal X_\Sigma, L_q)$.
Under these identifications, the boundary map $\partial: \pi_2( \mathcal X_\Sigma, L_q)\to \pi_1(L_q)$ now satisfies $\partial\beta_i=v_i$ for $1\le i\le n+r$.
There is an exact sequence $0\to \pi_2(\mathcal X_\Sigma)\to \pi_2(\mathcal X_\Sigma, L_q)\to \pi_1(L_q)\to 0$, so we may choose a basis $\{\mathcal S_a\}_{0\le a\le r} $ in $\pi_2(\mathcal X_\Sigma)$ such that
$
\mathcal S_a=\beta_{n+a} -(k_{a1}\beta_1+\cdots +k_{an}\beta_n)
$.
It is a standard result (see e.g. \cite{Guillemin1994kaehler}) that 
$
\frac{1}{2\pi} [\omega]=\sum_{i=1}^{n+r} \lambda_i \cdot \mathrm{PD}(D_i)
$.
Hence,
\begin{equation}
	\label{energy_H_a_eq}
	E(\mathcal S_a)=\lambda_{n+a}-\sum_{j=1}^n k_{aj}\lambda_j
\end{equation}
Just like \S \ref{ss_action_coordinates_-}, we use the natural action of $T_{\bar N}=\bar N\otimes \mathbb R/2\pi \mathbb Z$ and $\partial\beta_n$ to determine a local chart $\chi_+(q)=(q_1,\dots, q_{n-1}, \psi_+(q))$ of action coordinates over $B'_+$. Changing it by constants if necessary, we may assume $q_k+\lambda_k-\lambda_n=E(\beta_k)-E(\beta_n)$ for $1\le k<n$ and $\psi_+(q)+\lambda_n=E(\beta_n)$.

On the other hand, the fiber $L_q$ for $q\in B_-$ is of Chekanov type, and $|w-1|<1$ over here. As before, there is a natural topological disk $\hat\beta$ in $\pi_2(\mathcal X_\Sigma, L_q)$ that is a section of $w$ over a disk centered at $1\in\mathbb C$.
Similarly, we use the $T_{\bar N}$-action and $\partial\hat\beta$ to determine a local chart $\chi_-(q)=(q_1,\dots, q_{n-1}, \psi_-(q))$ of action coordinates over $B'_-$. We may assume $\psi_-(q)=E(\hat\beta)$ up to a constant change.
As before, we may extend $\hat\beta$ over $B'_-$ and check $\hat\beta|_{\mathscr N_i}$ equals $\beta_i$ for $1\le i\le n+r$. 
Besides, similar to (\ref{psi_+_psi_-_min_eq}), we can show that
\begin{equation*}
	\psi_-(q)=
	\begin{cases}
		\psi_+(q)+q_k+\lambda_k 
		& \text{if \ } q\in \mathscr N_k, \ 1\le k <n \\
		\psi_+(q)+\lambda_n
		&\text{if \ } q\in \mathscr N_n, \\
		\psi_+(q) +\sum_{s=1}^{n-1}k_{as} q_s +\lambda_{n+a}
		&\text{if \ } q\in\mathscr N_{n+a}, \ 1\le a \le r
	\end{cases}
\end{equation*}
In other words,
\begin{equation*}
	\psi_-(q)=\psi_+(q)+h_{\mathrm{trop}}( \bar q)
\end{equation*}
on their overlapping domains $\bigsqcup_i\mathscr N_i$. Similar to (\ref{psi_global_in_B_eq}), we obtain a continuous map $\psi: B\to \mathbb R$ such that $\psi|_{B'_-}=\psi_-$ and $\psi|_{B'_+}=\psi_+(q)+h_{\mathrm{trop}}(\bar q)$.
Set $\psi_0(\bar q) := \psi(\bar q,0)$.

Just as \S \ref{ss_j_homeomorphism}, we define a topological embedding $j:B\to \mathbb R$ by $q\mapsto (\theta_0(q),\theta_1(q), \bar q)$ where
\begin{equation*}
	\label{theta_0_1_generalized_eq}
	\begin{aligned}
		\theta_0(q_1,\dots, q_n)
		&:=
		\min\{ -\psi(q) , -\psi_0( \bar q ) \}  + h_{\mathrm{trop}}(\bar q)
		\\
		\theta_1(q_1,\dots, q_n) 
		&:=
		\min\{\ \ \ \psi(q) ,  \ \ \ \psi_0( \bar q )\}
	\end{aligned}
\end{equation*}
The image $j(B)$ can be described as in \S \ref{sss_describe_j}: let $R_{\bar q}\subset \mathbb R^2$ be the half broken line $r_{\bar q}$ in (\ref{jmath_eq}) with $h_{\mathrm{trop}}$ in (\ref{tropical_polynomial_eq}) replacing $\min\{0,\bar q\}$, and then the image $j(B)$ is the union of all $R_{\bar q}\times \{\bar q\}$ in $\mathbb R^{n+1}$ like (\ref{j(B)_explicit}).

\subsection{Mirror analytic structure}
According to Theorem \ref{Main_theorem_thesis_thm}, we can build up an analytic space $X_0^\vee$, which is set-theoretically $\bigcup_{q\in B_0} H^1(L_q; U_\Lambda)$, and the natural map $\pi_0^\vee:X_0^\vee \to B_0$ becomes an affinoid torus fibration.
By the same formulas as (\ref{tau_+_eq}, \ref{tau_-_eq}), the above two integral affine charts $\chi_\pm$ give rise to two affinoid tropical charts 
\[
\tau_\pm: (\pi_0^\vee)^{-1}(B'_\pm) \xrightarrow{\cong} T_\pm\subset (\Lambda^*)^n
\]
and their images are 
$
T_\pm := \trop^{-1}(\chi_\pm(B'_\pm)) \subsetneq (\Lambda^*)^n
$ just like (\ref{T_+-_eq}).
The two charts $\tau_\pm$ are glued by several automorphisms $\phi_i$ on $(\pi_0^\vee)^{-1}(\mathscr N_i)$ as (\ref{phi_ell_eq}); equivalently, we use $\Phi_i:=\tau_-\circ \phi_i\circ \tau_+^{-1}: T_+^i \to T_-^i$ to glue the analytic subdomains $T_\pm^i\equiv \trop^{-1}(\chi_\pm(\mathscr N_i))$ in $T_\pm$ (\ref{Phi_ell_eq}).
Roughly, placing the $\pi$ in the larger ambient manifold $\mathcal X_\Sigma$, one can check the Maslov-0 holomorphic disks keep the same. Thus, the analytic space $X_0^\vee$ and the affinoid torus fibration $\pi_0^\vee$ are unchanged.
In particular, the gluing maps $\phi_i$ and $\Phi_i$ are also unchanged.
On the other hand, there are new Maslov-2 disks that contribute to two analytic functions $W_\pm$ on the two affinoid tropical charts $(\pi_0^\vee)^{-1}(B'_\pm)\cong T_\pm$ (\ref{W_C_eq}).
By maximal principle, one can show that $W_-(\mathbf y)=T^{\psi_-(q)} \mathbf y^{\partial\hat\beta}$ for $\mathbf y\in H^1(L_q;U_\Lambda)$ with $q\in B'_-$.
Hence, 
$
\mathcal W_-(y):=W_-(\tau_-^{-1}(y))=y_n
$
for $y=(y_1,\dots, y_n)\in T_-$.
It follows from \cite{Cho_Oh} that
$
W_+(\mathbf y)= \sum_{i=1}^{n+r} T^{E(\beta_i)} \mathbf y^{\partial\beta_i} 
(1+\delta_i)
$
where 
\begin{equation}
	\label{delta_i_in_Lambda_+_eq}
\delta_i:=\sum_{\alpha\in H_2(\mathcal X_\Sigma) \setminus\{0\}} T^{E(\alpha)} \mathsf n_{\beta_i+\alpha}\in \Lambda_+
\end{equation}
and $\mathsf n_\beta$ is the count of holomorphic stable disks of class $\beta$.
Unlike the previous case when $\mathcal X_\Sigma=\mathbb C^n$, the Cho-Oh's result is not strong enough to determine the coefficients $\delta_i$'s in general, as the sphere bubbles can contribute if the corresponding toric divisor is compact \cite[Proposition 5.3]{CLL12}.
But, if $D_i$ is non-compact, then we can use the maximal principle to prove that $\delta_i=0$.
Anyway, regarding the chart $\tau_+$, we can check that
$
\mathcal W_+(y):=W_+(\tau_+^{-1}(y))
=
y_n\cdot h(y_1,\dots, y_{n-1})
$
where $y=(y_1,\dots, y_n) \in T_+$ and
\begin{equation}
	\label{tropical_original_h_eq}
h(y_1,\dots, y_{n-1})=
T^{\lambda_n}(1+\delta_n) + \sum_{s=1}^{n-1}T^{\lambda_s} y_s (1+\delta_s)  + \sum_a T^{\lambda_{n+a}} (1+\delta_{n+a}) \prod_{s=1}^{n-1} y_s^{k_{as}}  
\end{equation}
Remark that this more or less agrees with many previous works (e.g. \cite{AuTDual} \cite{Au_Special} \cite{CLL12}).
However, working over the non-archimedean field $\lambda$ is very crucial for the following observation: the tropicalization of $h$ is precisely the tropical polynomial (\ref{tropical_polynomial_eq}). This picture would be totally missed over $\mathbb C$; c.f. \S \ref{sss_ex_SYZ_converse_intro}.

Due to (\ref{superpotential_gluing_map_eq}), we have $\mathcal W_-(\Phi_i(y))=\mathcal W_+(y)$ for any $y\in T_+^i$. Besides, as before, the $T_{\bar N}$-symmetry of the Lagrangian fibration $\pi$ implies that $\Phi_i$ preserves the first $n-1$ coordinates (\ref{Phi_ell_explicit_eq}). Finally, just like (\ref{Phi_ell_explicit_determined_eq}), we can show that
\[
\Phi_i(y_1,\dots, y_n)=\big(y_1,\dots, y_{n-1}, y_n h(y_1,\dots, y_{n-1}) \big)
\]
for any $y=(y_1,\dots, y_n)\in T_+^i$ and $1\le i\le n+r$.
In the same way as (\ref{identification_Floer_use_eq}), we have an identification $X_0^\vee \cong T_+\sqcup T_-/\sim$ for a similar gluing relation defined by the above $\Phi_i$'s. There is also a similar characterization of $\pi_0^\vee$ as (\ref{identification_pi_check_property_eq}).
\textit{Note that the viewpoint of Remark \ref{floer_omit_rmk} still works here.}

Just as \S \ref{ss_g_analytic_embedding}, we obtain an embedding $g$ from $X_0^\vee $ into (the analytification of) the algebraic $\Lambda$-variety:
\[
Y:=\left\{ (x_0,x_1, y_1,\dots, y_{n-1})\in \Lambda^2\times (\Lambda^*)^{n-1} \mid x_0x_1= h(y_1,\dots, y_{n-1}) \right\}
\]

\subsection{Dual analytic fibration}

Given $z=(x_0,x_1,y_1,\dots, y_{n-1})$ in $\Lambda^2\times (\Lambda^*)^{n-1}$, we define:
\[
\begin{aligned}
	F_0(z)&=\min\{  
	\val(x_0),   -\psi_0(\val(y_1),\dots, \val(y_{n-1}) )+h_{\mathrm{trop}}(\mathsf v(y_1),\dots, \mathsf v(y_{n-1}))  \} \\
	F_1(z)&= \min \{ \val(x_1), \ \ \ \psi_0(\val(y_1),\dots, \val(y_{n-1})) \} \\
	G_k(z)&=\val(y_k) \qquad \text{for} \  1\le k< n
\end{aligned}
\]
This is only a slight modification of (\ref{F_KS_eq}) with $h_{\mathrm{trop}}$ in (\ref{tropical_polynomial_eq}) replacing $\min\{0,q_1,\dots, q_{n-1}\}$.
Now, we define $F:=(F_0,F_1, G_1,\dots, G_{n-1}): Y\to\mathbb R^{n+1}$.
One can similarly describe the image of $F$ as before in \S \ref{sss_describe_F}.
Roughly, the image $\mathfrak B=F(Y)$ also takes the form of (\ref{F(Y)_mathfrak_B_eq}) as is the union of all $S_{\bar q}\times \{\bar q\}$ in $\mathbb R^{n+1}$, but the broken line $s_{\bar q}$ in (\ref{broken_line_s_barq_eq}) is modified by replacing $\min\{0,\bar q\}$ by $h_{\mathrm{trop}}$.
The image $j(B)$ then agrees with the open subset $\hat B=\{(u_0,u_1,\bar q)\in\mathfrak B\mid u_1>0\}$, and $j$ exactly sends the singular locus $\Delta$ of $\pi$ to the singular locus of $F$ as in Lemma \ref{j(B)_lem}.
Just like \S \ref{ss_defn_f_B_fibration}, we define $\mathscr Y=F^{-1}(\hat B)$ and $f=j^{-1}\circ F:\mathscr Y\to B$; also, the restriction $f_0:=f|_{B_0}$ gives an affinoid torus fibration.
In the end, the proof of Theorem \ref{fibration_preserving_thm} can be repeated verbatim here, obtaining
$F\circ g=j\circ \pi_0^\vee$ and thus $\pi_0^\vee=f_0\circ g$.
This completes the proof of Theorem \ref{Main_this_paper_general}.

\appendix
\section{Folklore conjecture for the critical values of Landau-Ginzburg models}
\label{s_folklore}

In this appendix, we check some computations for the well-known folklore Conjecture \ref{conjecture_folklore} as mentioned in \S \ref{ss_folklore_introduction}.
A conceptual proof of the folklore conjecture is also given in \cite{Yuan_c_1} with much generalities.
We recommend a brief reading of \S \ref{s_family_review} in advance.

\subsection{General aspects}
\label{ss_log_deri}
Let $(X_0^\vee,\pi_0^\vee, W_0^\vee)$ be given in Theorem \ref{Main_theorem_thesis_thm}, and we often omit the subscript $0$ if there is no confusion.
By Remark \ref{rigor_superpotential_rmk}, the superpotential should be (locally) viewed as a formal power series in $\Lambda[[\pi_1(L_{q_0})]]$ for some base point ${q_0}$.
In general, let's take an arbitrary formal power series $F=\sum_{j=1}^\infty c_jY^{\alpha_j}$ in $\Lambda[[\pi_1(L_{q_0})]]$ for $c_j\in \Lambda$ and $\alpha_j\in \pi_1(L_{q_0})$. Given $\theta\in H^1(L_{q_0})\cong \Hom (\pi_1(L_{q_0}),\mathbb R)$, we define the \textbf{logarithmic derivative} along $\theta$ of $F$ by
\[
\textstyle
D_\theta F=\sum_{j=1}^\infty c_j \langle \alpha_j, \theta\rangle Y^{\alpha_j}
\]
By \S \ref{ss_FF_local_affinoid_tropical_chart}, we take a local affinoid tropical chart of the affinoid torus fibration $\pi_0^\vee$:
\[
\tau: (\pi_0^\vee)^{-1}(U)\cong \trop^{-1}(V-c)
\]
that covers a pointed integral affine chart $\chi:(U,q_0)\to (V,c)$.
By \S \ref{ss_review_superpotential}, the superpotential $W^\vee$ in this chart is given by:
\[
W^\vee|_\tau(\mathbf y)=\sum_{\beta\in \pi_2(X,L_q) , \ \mu(\beta)=2} T^{E(\beta)} \mathbf y^{\partial\beta} \mathsf n_{\beta(q_0)}
\]
for $\mathbf y\in H^1(L_q; U_\Lambda)$ with $q\in U$.
Its logarithmic derivative along $\theta=\theta_q\in H^1(L_q; \mathbb R)$ is given by
\[
D_\theta W^\vee|_\tau (\mathbf y)=\sum_{\mu(\beta)=2} \langle \partial\beta, \theta\rangle T^{E(\beta)} \mathbf y^{\partial\beta} \mathsf n_{\beta(q_0)}
\]

\begin{defn}
	\label{critical_points_defn}
	A point $\mathbf y$ in $H^1(L_q; U_\Lambda)\subset X_0^\vee$ is called a \textbf{critical point} of $W^\vee$ if $D_\theta W^\vee|_\tau (\mathbf y)=0$ for all $\theta$ and for some affinoid tropical chart $\tau$.
\end{defn}

This definition does not depend on the choice of the affinoid tropical chart $\tau$ as proved in \cite{Yuan_c_1}. So, we often omit writing the $\tau$ in the notations.
In our situation, we just have two affinoid tropical charts:
\[
\tau_\pm: (\pi_0^\vee)^{-1}(B'_\pm) \xrightarrow{\cong} T_\pm \subset (\Lambda^*)^n
\]
as in (\ref{tau_+_eq}, \ref{tau_-_eq}).
We will use $\overline X$ to denote a compactification of $X$ and use $W^\vee$ to denote the consequent LG superpotential. Denote by $W^\vee_{\pm}:=W^\vee_{\overline X, \pm}$ the restrictions of $W^\vee$ on $(\pi_0^\vee)^{-1}(B'_\pm)$.

The superpotential on the Clifford chamber $T_+$ is usually easy to find by \cite{Cho_Oh}.
The superpotential on the Chekanov chamber $T_-$ is hard to compute by the classic ideas. But, it can be now computed by the superpotential-preserving property (\ref{superpotential_gluing_map_eq}) of the family Floer gluing maps as did in \cite{Yuan_e.g._FamilyFloer}. Another approach over $\mathbb C$ using the Lagrangian mutations has been studied in \cite{PT_mutation}.
However, the computations in \cite{Yuan_e.g._FamilyFloer} over $\Lambda$ rather than over $\mathbb C$ is crucial to check the folklore conjecture.

\subsection{Examples}
For clarity, we only study the case $X=\mathbb C^n\setminus \mathscr D$ as in Theorem \ref{Main_this_paper_fundamental_example}. But, one may obtain many other computations in the case of Theorem \ref{Main_this_paper_general} using similar arguments.

\subsubsection{}
Assume $\overline X=\mathbb {CP}^n$.
There is a topological disk $\beta'\equiv \beta'(q) \in \pi_2(\mathbb {CP}^n, L_q)$ that intersects the divisor $\mathbb {CP}^n\setminus \mathbb C^n$ once and satisfies
\begin{equation}
	\label{beta'_relation}
	\beta'=\mathcal H-\beta_1-\cdots -\beta_n
\end{equation}
where $\mathcal H$ is the complex line that generates $\pi_2(\mathbb {CP}^n)$.
We also view $\beta'$ as a local section of $\mathscr R_2$ in (\ref{local_system_H_2_H_1_eq}) over $B'_+$.
Adding an extra subscript to distinguish, we use $W^\vee_{\mathbb C^n,\pm }$ to denote the superpotentials obtained before in (\ref{W_C_eq}). 
First, by \cite{Cho_Oh}, the superpotential over $B'_+$ is easy to find:
\begin{equation*}
	\begin{aligned}
		W^\vee_{+} (\mathbf y) 
		&= T^{E(\beta')} \mathbf y^{\partial\beta'} \mathsf n_{\beta'} + \sum_{j=1}^n T^{E(\beta_j)} \mathbf y^{\partial\beta_j} \mathsf n_{\beta_j} = T^{E(\beta')} \mathbf y^{\partial\beta'} + W_{\mathbb C^n,+}^\vee (\mathbf y) \\
		&=
		T^{E(\mathcal H) -q_1-\cdots -q_{n-1}-n\psi_+(q)} \mathbf y^{-n\partial\beta_n-(\sigma_1+\cdots +\sigma_{n-1})}
		+
		T^{\psi_+(q)} \mathbf y^{\partial\beta_n} \big(1+ \sum_{k\neq n} T^{q_k} \mathbf y^{\sigma_k}\big)
	\end{aligned}
\end{equation*}
We can also use \cite{Cho_Oh} again to conclude $\mathsf n_{\beta'}=1$.
By \cite{Yuan_e.g._FamilyFloer}, the superpotential over $B'_-$ is
\begin{equation}
	\label{W_P_+_eq}
	W^\vee_{-}( {\mathbf y})
	=
	T^{\psi_-}  {\mathbf y}^{\partial\hat\beta}
	+
	T^{-n\psi_-+E(\mathcal H)}   {\mathbf y}^{-n\partial\hat\beta-(\sigma_1+\cdots+\sigma_{n-1})}\big(1+T^{q_1}   {\mathbf y}^{\sigma_1}+\cdots+ T^{q_{n-1}}  {\mathbf y}^{\sigma_{n-1}} \big)^n
\end{equation} 

Under the identifications $(\pi_0^\vee)^{-1}(B'_\pm)\cong T_\pm$ via $\tau_\pm$, they have the following expressions:
For $y=(y_1,\dots, y_n)$ in $T_+$ or $T_-$, we respectively have
\begin{equation}
	\label{W_CP_explicit}
	\begin{aligned}
	W^\vee_+ (y) &= y_n(1+y_1+\cdots+y_{n-1}) +  \frac{T^{E(\mathcal H)}}{y_1\cdots y_{n-1} \cdot y_n^n}
		\\
	W^\vee_- (y) &= y_n + \frac{T^{E(\mathcal H)}(1+y_1+\cdots + y_{n-1})^n}{y_1\cdots y_{n-1}\cdot y_n^n} \
	\end{aligned}
\end{equation}
Moreover, under the analytic embedding $g$ in (\ref{g_analytic_map_eq}) into the $\Lambda$-variety 
\[
Y=\{ \pmb z=(x_0,x_1,y_1,\dots, y_{n-1})  \in \Lambda^2\times (\Lambda^*)^{n-1} \mid  x_0x_1=1+y_1+\cdots +y_{n-1} \}
\]
we can check that the $W_\pm^\vee$ together give rise to a single expression
\begin{equation}
	\label{W_single_CPn}
	W^\vee (\pmb z) = x_1+\frac{T^{E(\mathcal H)} \cdot x_0^n}{ y_1\cdots y_{n-1}}
\end{equation}
It is first only defined on a subdomain $\mathscr Y_0$ in $Y$; but thanks to Theorem \ref{Main_this_paper_fundamental_example}, its domain can be extended to the whole $Y$.

Next, we aim to find the critical points of $W^\vee$.
The methods are not unique. For instance, we may use various \textit{tropical charts} of $Y$ in the sense of \cite[\S 3]{Formes_Chambert_2012} (c.f. the recent \cite{gubler2021forms}). The critical points are actually coordinate-free.
Anyway, for clarity, let's use the two familiar charts $\tau_\pm$.

\begin{itemize} 
	\item In the affinoid tropical chart $\tau_+$, we have
	\begin{align*}
		D_k W^\vee_+ 
		&= - \frac{T^{E(\mathcal H)} }{y_1\cdots y_{n-1} y_n^n} + y_ny_k  && 1\le k <n \\
		D_n W^\vee_+ 
		&= -\frac{n \ T^{E(\mathcal H)}}{y_1\cdots y_{n-1} y_n^n} + y_n(1+y_1+\cdots +y_{n-1})
	\end{align*}
	for $y=(y_1,\dots, y_n)\in T_+$. Then, we can solve and obtain the critical points (if exist in $T_+$):
	\[
	\begin{cases}
		y_k=1 & 1\le k<n \\
		y_n= T^{\frac{1}{n+1} E(\mathcal H)}  e^{\frac{2\pi i s}{n+1}} & s\in \{0,1,\dots, n\}
	\end{cases}
	\]
	The corresponding critical values are $(n+1)T^{\frac{1}{n+1} E(\mathcal H)}  e^{\frac{2\pi i s}{n+1}}$.

	\item 
In the affinoid tropical chart $\tau_-$, we also have
	\begin{align*}
		D_k W_-^\vee 
		&=
		-\frac{T^{E(\mathcal H)} (1+y_1+\cdots +y_{n-1})^n }{y_1\cdots y_{n-1} y_n^n} 
		+
		n \frac{T^{E(\mathcal H)} (1+y_1+\cdots +y_{n-1})^{n-1} y_k}{y_1\cdots y_{n-1}y_n^n} \\
		D_n W_-^\vee
		&=
		y_n - \frac{n T^{E(\mathcal H)} (1+y_1+\cdots +y_{n-1})^n}{y_1\cdots y_{n-1}y_n^n}
	\end{align*}
	for $y=(y_1,\dots, y_n)\in T_-$. Then, we can solve and obtain the critical points (if any):
	\[
	\begin{cases}
		y_k=1 & 1\le k < n \\
		y_n = nT^{\frac{1}{n+1} E(\mathcal H)} e^{\frac{2\pi i s}{n+1}} & s\in\{0,1,\dots, n\}
	\end{cases}
	\]
	The corresponding critical values are also $(n+1)T^{\frac{1}{n+1} E(\mathcal H)}  e^{\frac{2\pi i s}{n+1}}$.
	
\end{itemize}

In either case, the critical values agree with the known eigenvalues of the quantum product in $\mathbb {CP}^n$ of the first Chern class.
Under the analytic embedding $g$, the critical points in either bullets have the uniform expressions as follows:
\[
\pmb z^{(s)} =(x_0,x_1,y_1,\dots, y_{n-1}) \ \in Y: \qquad 
\begin{cases}
	x_0 =  T^{-\frac{E(\mathcal H)}{n+1}} e^{-\frac{2\pi i s}{n+1}}	&		\\
	x_1 = 	n T^{\frac{E(\mathcal H)}{n+1}	} e^{\frac{2\pi i s}{n+1}} & 	\\
	y_k =1 & 1\le k<n
\end{cases}
\]
for $s\in\{0,1,\dots, n\}$.
Since $\val(y_k)=0$, the critical points are all contained in the dual fiber over a fixed single point $\hat q=(0,\dots, 0, a_\omega)$ in $B\equiv \mathbb R^n$ for some $a_\omega\in\mathbb R$, and $\val(x_1)=\psi(\hat q)=\frac{1}{n+1}E(\mathcal H)=\frac{1}{2\pi(n+1)}\omega(\mathcal H)$.
In particular, the K\"ahler form $\omega$ determines $a_\omega$ and the base point $\hat q$. Besides, the critical points are in $T_+$ if $a_\omega>0$ and in $T_-$ if $a_\omega<0$. But, it is not completely clear when $a_\omega=0$ and $\hat q$ lies on the singular locus.

\subsubsection{}
Assume $\overline X=\mathbb {CP}^m\times \mathbb {CP}^{n-m}$ for $1\le m <n$. Let $\mathcal H_1$ (resp. $\mathcal H_2$) be the class of a complex line in $\mathbb {CP}^m$ (resp. $\mathbb {CP}^{n-m}$).
There are two new disk classes $\beta'_1$ and $\beta'_2$ bounding the $\pi$-fibers over $B_+$.
Besides, $\mathcal H_1=\beta'_1+\sum_{i=1}^m \beta_i$ and $\mathcal H_2=\beta'_2+\sum_{i=m+1}^n \beta_i$.
One can finally show that the superpotential under the analytic embedding $g$ is given by
\[
W^\vee (\pmb z)
=
x_1 +  \frac{T^{E(\mathcal H_1)} x_0^m}{y_1\cdots y_m} +  \frac{T^{E(\mathcal H_2)} x_0^{n-m}}{ y_{m+1}\cdots y_{n-1}}
\]
for $\pmb z=(x_0,x_1,y_1,\dots, y_{n-1})$ on the same variety $Y$. We have $(m+1)(n-m+1)$ critical points
\[
\begin{cases}
	x_0 =	\big( T^{\frac{E(\mathcal H_2)}{n-m+1}} e^{\frac{2\pi i s}{n-m+1}} \big)^{-1}	& \\
	x_1 =	T^{\frac{E(\mathcal H_2)}{n-m+1}} e^{\frac{2\pi i s}{n-m+1}}  \cdot \left(m \  T^{\frac{E(\mathcal H_1)}{m+1}}  e^{\frac{2\pi i r}{m+1}} \cdot \big( T^{\frac{E(\mathcal H_2)}{n-m+1}} e^{\frac{2\pi i s}{n-m+1}} \big)^{-1}+ n-m\right) & \\
	y_k = T^{\frac{E(\mathcal H_1)}{m+1}}  e^{\frac{2\pi i r}{m+1}} \cdot \big( T^{\frac{E(\mathcal H_2)}{n-m+1}} e^{\frac{2\pi i s}{n-m+1}} \big)^{-1} & 1\le k\le m \\
	y_\ell= 1 & m<  \ell <n 
\end{cases}
\]
for $r\in\{0,1,\dots, m\}$ and $s\in\{0,1,\dots, n-m\}$.
Then, the corresponding critical values are
\[
(m+1) T^{\frac{E(\mathcal H_1)}{m+1}} e^{\frac{2\pi i r}{m+1}} + (n-m+1) T^{\frac{E(\mathcal H_2)}{n-m+1}} e^{\frac{2\pi i s}{n-m+1}} 
\]

Let's further look into a special case when $m=1$ and $n=2$. Namely, $\overline X=\mathbb {CP}^1\times \mathbb {CP}^1$. Then,
\[
W^\vee=x_1+ \frac{T^{E(\mathcal H_1)}x_0}{y_1}+{T^{E(\mathcal H_2)} x_0}
\]
By the above computation, we have four critical points given by
\[
\begin{cases}
	x_0=\big( T^{\frac{E(\mathcal H_2)}{2}}  e^{\pi i s}	\big)^{-1} \\
	x_1=  T^{\frac{E(\mathcal H_2)}{2}} e^{\pi i s} \big(
	T^{\frac{E(\mathcal H_1)-E(\mathcal H_2)}{2}} e^{\pi i s} e^{\pi i r} +1
	\big) \\
	y_1=T^{\frac{E(\mathcal H_1)-E(\mathcal H_2)}{2}} e^{\pi i r} e^{\pi i s}
\end{cases}
\]
for $r,s\in\{0,1\}$.
The base point in $B\equiv \mathbb R^2$ is given by
$
\hat q=\left( \frac{E(\mathcal H_1)-E(\mathcal H_2)}{2}, a_\omega\right)
$
for some $a_\omega\in\mathbb R$.
More examples can be similarly found out by either choosing a different compactification $\overline X$ or working with a more general $X$ as Theorem \ref{Main_this_paper_general}; compare also \cite{Yuan_e.g._FamilyFloer}.


\bibliographystyle{abbrv}

\bibliography{mybib_local_SYZ}

\end{document}